\documentclass[12pt, oneside, a4paper]{book}

\usepackage[latin1]{inputenc}
\usepackage{fancyhdr,amssymb,amsmath,multicol,tabularx,makeidx,color,}
\usepackage{amsfonts,amsthm,graphicx, mathrsfs}
\usepackage{algorithmic}
\usepackage{algorithm2e}
\usepackage{tikz}
\usepackage{hyperref}

\usepackage{listings}

\usepackage[italian]{babel}
\usepackage{picinpar}
\usepackage{textcomp, setspace}
\DeclareGraphicsExtensions{.pdf,.png,.jpg,.mps}
\usepackage[center, bf, small]{caption}

\theoremstyle{plain} 
\newtheorem{thm}{Teorema}[section]
\newtheorem{cor}[thm]{Corollario} 
 
\newtheorem{prop}[thm]{Proposizione}
\newtheorem{oss}[thm]{Osservazione}
\newtheorem{ossi}[thm]{Osservazioni}
\newtheorem{defn}[thm]{Definizione}
\newtheorem{es}[thm]{Esempi}
\newtheorem{eso}[thm]{Esempio}
\newtheorem{congettura}[thm]{Congettura}
\theoremstyle{remark}

\fancypagestyle{plain}

\begin{document}

\thispagestyle{empty}
\begin{center}

{\scshape\large \textbf{UNIVERSIT\`{A} DELLA CALABRIA}}

\vspace{0.45 cm}
\includegraphics[scale=.4, keepaspectratio]{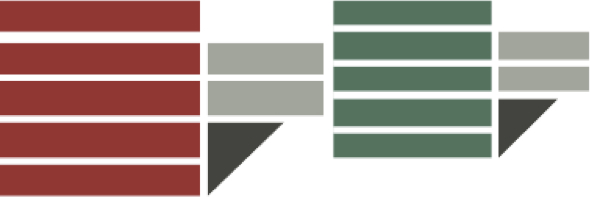}\\
{\scshape\large \textbf{Facolt\`{a} di Scienze Matematiche, Fisiche e Naturali}}

\vspace{0.45cm}

{Corso di Laurea in Matematica}

\vspace{0.45 cm}

{\scshape\large Tesi di Laurea Magistrale}
\vspace{4.0cm}

{\huge \textbf{$LS$-successioni di punti nel quadrato} }
\end{center}
\vspace{4.0cm}

\par
\noindent
\begin{minipage}[t]{0.47\textwidth}
{\large{Relatore:\\
Chiar.mo Prof.\\
Aljo\v{s}a Vol\v{c}i\v{c}}}\\

\vspace{0.5 cm}
\noindent

{\large{Correlatore:\\
Chiar.ma Dott.ssa\\
Ingrid Carbone}}
\end{minipage}
\hfill
\begin{minipage}[t]{0.47\textwidth}\raggedleft
{\large{Candidato:\\
Maria Rita Iac\`{o} \\ Mtr. 138146}}
\end{minipage}

\vspace{1.5 cm}
\begin{center}
\noindent
{\scshape\large Anno Accademico 2010/2011}

\end{center}

\frontmatter
\tableofcontents
\listoffigures

\clearpage{\pagestyle{empty}\cleardoublepage}

\chapter{Abstract}
The main purpose of this master thesis is to study the $LS$-sequences of points introduced by Carbone in \cite{Carbone} and find two generalizations of them to the unit square. Here we also present a new algorithm proposed by the same author in \cite{Carbone2} and we implement it in order to have a graphical description of these sequences.
\\Chapter 1 includes a collection of results concerning the uniform ditribution theory and the discrepancy (we refer to \cite{Drmota_Tichy} and \cite{Kuipers_Niederreiter} for a complete survey on the matter).
\\In Chapter 2 we focuse our attention on the $LS$-sequences of partitions and of points in the unit interval, giving particular attention to the ordering of the points \lq\lq \`{a} la van der Corput \rq\rq and finding a way to compute them related to the digit expansion of natural numbers in base $L+S$.
\\ In Chapter 3 we move on the unit square where we find two generalizations of the $LS$-sequences following the historical development of the van der Corput sequence in the multidimensional case. This is the reason why we call the first sequence $LS$- sequence of points \`{a} la van der Corput-Hammersley and the second one $LS$-sequence \`{a} la Halton.
\chapter{Introduzione}
Da lungo tempo lo studio della Teoria dell'\emph{Uniforme Distribuzione} suscita grande interesse da parte di molti matematici.\\ Gli albori di questa teoria risalgono agli inizi del '900 e precisamente l'anno di nascita viene fatto coincidere con il 1916, quando H. Weyl pubblica il lavoro dal titolo \emph{\"{U}ber die Gleichverteilung von Zahlen mod. Eins} \cite{Weyl2}, in cui generalizza i risultati ottenuti dal matematico Kronecker sull'approssimazione di equazioni diofantee.\\ Da quella data molti passi in avanti sono stati compiuti ed il campo di indagine si \`{e} esteso dalla Teoria dell'Approssimazione a rami della matematica che vanno dall'Analisi Funzionale alla Teoria dei Numeri, dalla Teoria della Probabilit\`{a} all'Algebra Topologica. \\ Lo scopo principale di questa tesi \`{e} stato quello di estendere al quadrato unitario le $LS$-successioni di punti con $L, S \in \mathbb{N}$ soddisfacenti un'opportuna condizione, introdotte, insieme alle $LS$-successioni di partizioni, da I. Carbone in \cite{Carbone} nell'intervallo $[0,1[$. \smallskip\\ In particolare, questo lavoro si articola in tre capitoli cos\`{i} ripartiti.\smallskip\\
Nel primo capitolo si \`{e} scelto di presentare i concetti di base ed i criteri classici della Teoria dell'Uniforme Distribuzione modulo $1$, necessari per la comprensione dei capitoli successivi e che permettono di stabilire quando una successione di numeri reali \`{e} uniformemente distribuita nell'intervallo unitario, seguendo l'impostazione di \cite{Kuipers_Niederreiter}. Tra questi, uno dei risultati pi\`{u} rilevanti, in quanto si tratta di un criterio sufficiente per determinare classi di successioni uniformemente distribuite, \`{e} il Criterio di  Weyl, di cui abbiamo preso anche in esame alcune tra le sue molteplici applicazioni. \\Queste nozioni sono state poi estese, in maniera naturale, a successioni a valori in $\mathbb{R}^{s}$. \\ Infine, si \`{e} scelto di presentare i risultati esposti in precedenza da un punto di vista quantitativo e abbiamo introdotto la quantit\`{a} detta \emph{discrepanza} che ci permette di distinguere tra \lq \lq buone\rq\rq e \lq\lq cattive\rq\rq successioni di punti uniformemente distribuite. Infatti, tra le successioni uniformemente distribuite modulo 1 ve ne sono alcune che presentano una buona distribuzione, al contrario di altre che sono soltanto uniformemente distribuite.
\smallskip\\ Il Capitolo $2$ \`{e} caratterizzato dall'estensione della teoria dell'uniforme distribuzione alle successioni di partizioni, concetto introdotto dal matematico S. Kakutani in \cite{Kakutani} nel $1976$. In particolare, abbiamo preso in esame la costruzione di successioni di partizioni con la tecnica degli $\alpha$-raffinamenti introdotti dallo stesso Kakutani e pi\`{u} in generale il metodo dei $\rho$-raffinamenti introdotti da A. Vol\v{c}i\v{c} in \cite{Volcic} e abbiamo presentato la classe delle $LS$-successioni di partizioni in $[0,1[$ che fanno uso di tale costruzione.\\
 Abbiamo visto che esiste uno stretto rapporto tra le successioni di partizioni e le successioni di punti ad esse associate e abbiamo presentato due metodi, introdotti da Carbone in \cite{Carbone} e \cite{Carbone2}, per la determinazione di queste ultime, cercando di mettere in luce come il secondo metodo sia pi\`{u} efficiente rispetto al primo. \\ Abbiamo anche dato la definizione di discrepanza di successioni di partizioni, e presentato le stime per la discrepanza delle $LS$-successioni di partizioni e anche per quelle di punti ad esse associate, dimostrate in \cite{Carbone}.\\ In questo capitolo, inoltre, abbiamo dato grande spazio agli esempi in modo da rendere la trattazione degli argomenti pi\`{u} chiara e a questo scopo molti degli esempi sono stati corredati di grafici. Infine, abbiamo concluso questa seconda parte con la presentazione di un algoritmo che descrive la procedura \lq \lq  alla van der Corput\rq\rq  per determinare le $LS$-successioni di punti. Quest'ultimo ci ha permesso di implementare i risultati descritti precedentemente.
\smallskip\\ Nel terzo capitolo abbiamo dapprima considerato gli esempi classici di successioni uniformemente distribuite sul quadrato di lato unitario, cio\`{e} le successioni di J. C. van der Corput \cite{vanderCorput}, J. M. Hammersley \cite{Hammersley} e J. H. Halton \cite{Halton}. Trattandosi di successioni aventi discrepanza asintotica ottimale (dette a bassa discrepanza) abbiamo descritto brevemente il Metodo Quasi-Monte Carlo, che fa uso di successioni di questo tipo.\\Anche una sottofamiglia numerabile delle $LS$-successioni di punti (e di partizioni) \`{e} a bassa discrepanza. Pertanto, in maniera analoga a quanto fatto da van der Corput, nel passaggio dalla dimensione $1$ alla dimensione $2$, costruiamo una prima generalizzazione delle $LS$-successioni di punti sul quadrato che abbiamo chiamato $LS$-successioni ``alla van der Corput''. I risultati ottenuti fanno ben sperare che si tratti di una successione uniformemente distribuita. \\In seguito, siamo passati alla trattazione vera e propria di successioni (infinite) di tipo $LS$ sul quadrato, costruendo coppie di $LS$-successioni di punti che abbiamo chiamato $LS$-successioni``alla Halton''. Questi ultimi studi hanno portato alla luce risultati molto variegati tra loro. Dai grafici di queste successioni, infatti, si possono osservare distribuzioni di punti abbastanza soddisfacenti ed altre che presentano un comportamento alquanto bizzarro. Tra queste ultime, quella che ha suscitato il maggiore interesse \`{e} stata la successione ottenuta considerando le coppie $L=S=1$ ed $L=4$, $S=1$, il cui grafico risulta essere il pi\`{u} irregolare come vedremo in dettaglio nel corso del capitolo. \smallskip\\La tesi si conclude con la congettura che le $LS$-successioni ``alla van der Corput'' sono uniformemente distribuite. Si tratta di una congettura, dal momento che nessun risultato teorico avalla ancora la nostra tesi, ma i risultati numerici e grafici ottenuti in questa tesi rendono questo risultato plausibile. 

\mainmatter

\chapter{Successioni di punti uniformemente distribuite modulo 1}
Lo scopo di questo primo capitolo \`{e} quello di presentare i concetti di base della Teoria dell'\emph{Uniforme Distribuzione}, necessari per la comprensione dei capitoli successivi. Inizieremo con alcune definizioni ed alcuni risultati classici sulle successioni di punti uniformemente distribuite, forniremo la definizione di \emph{discrepanza} di una successione di punti e presenteremo alcuni esempi ben noti in letteratura, seguendo l'impostazione classica di \cite{Kuipers_Niederreiter}.
\section{Definizioni e nozioni preliminari}
Sia $(x_n)_{n\in \mathbb{N}}$ una successione di numeri reali. Ogni $x_n$ pu\`{o} essere scomposto nella somma della sua parte intera $\lfloor x_n\rfloor=\sup \{m \in \mathbb{Z} \mid m \leq x_n\}$ e della sua parte frazionaria $\{x_n\}=x_n - \lfloor x_n\rfloor$. Dato che la parte frazionaria di $x_n$ coincide proprio con il valore di $x_n$ modulo 1, lo studio di successioni $(x_n)$ di numeri reali modulo 1 si riconduce allo studio delle successioni delle parti frazionarie di quest'ultime nell'intervallo $[0,1[$. Se la successione \`{e} a valori nell'intervallo unitario, allora essa coincide con la successione delle sue parti frazionarie.
\begin{defn}
\label{def:1.1.1}
\rm{Una successione $(x_n)_{n \in \mathbb{N}}$ di numeri reali si dice \emph{uniformemente distribuita modulo 1} ($u.d.$ $mod\ \mathit{1}$) se per ogni coppia di numeri reali $a$ e $b$, con $0\leq a < b \leq 1$, si ha
\begin{equation}
\lim_{N \to \infty}\frac{1}{N} \sum_{n=1} ^{N} \chi_{[a, b[}(\{x_n\})=b-a\ , \label{eq:1.1}
\end{equation}
dove $\chi_{[a, b[}$ denota la funzione caratteristica dell'intervallo $[a,b[$.}
\end{defn}
Una definizione equivalente si ottiene se consideriamo intervalli del tipo $[0,c[$ , con $0 < c \leq 1$.  Infatti, vale il seguente risultato.
\begin{prop}
Una successione $(x_n)_{n \in \mathbb{N}}$ di numeri reali \`{e} u.d. mod\ 1 se e solo se
\begin{equation}
\lim_{N \to \infty}\frac{1}{N} \sum_{n=1} ^{N} \chi_{[0,c[}(\{x_n\})=c  \label{eq:1.2}
\end{equation}
per ogni numero reale $c$ in $]0,1]$.
\end{prop}
\begin{proof}
La dimostrazione \`{e} immediata. Infatti, se $(x_n)$ \`{e} una successione di numeri reali e se $[a,b[$ \`{e} un sottointervallo dell'intervallo unitario, allora
\begin{equation*}
\lim_{N \to \infty}\frac{1}{N} \sum_{n=1} ^{N} \chi_{[a, b[}(\{x_n\}) =  
\end{equation*}
\begin{equation*}
=\lim_{N \to \infty}\frac{1}{N} \sum_{n=1} ^{N} \chi_{[0, b[}(\{x_n\}) - \lim_{N \to \infty}\frac{1}{N} \sum_{n=1} ^{N} \chi_{[0, a[}(\{x_n\})=b-a\ .
\end{equation*}
Dunque la successione \`{e} u.d. mod 1 in $[a,b[$. \\ Il viceversa \`{e} ovvio.
\end{proof}
Per brevit\`{a} di notazione, nel seguito porremo $I=[0, 1[$.\\ Vediamo adesso un primo risultato che prende il nome dal matematico tedesco H. Weyl il quale, nel 1914, fu il primo a dimostrare questo importante criterio che fornisce una condizione necessaria e  sufficiente affinch\'{e} una successione sia u.d. mod 1.
\begin{thm}
\label{thm:1.1.3}
Una successione $(x_n)_{n\in\mathbb{N}}$ \`{e} u.d. mod 1 se e soltanto se per ogni funzione f definita su $\bar{I}$, continua e a valori reali, vale
\begin{equation}
\lim_{N \to \infty}\frac{1}{N} \sum_{n=1} ^{N} f(\{x_n\})= \int_0 ^{1} f(x)dx\ . \label{eq:1.3}
\end{equation} 
\end{thm}
\begin{proof}
Dimostriamo che la condizione \`{e} necessaria supponendo dapprima che la $f$ sia una funzione semplice definita sull'intervallo unitario. \\Sia dunque $(x_n)$ u.d. mod 1 e la $f$ definita da $$f(x)=\sum_{i=0}^{k-1}c_i\chi_{[a_i,a_{i+1}[}(x)\ ,$$ 
dove $0=a_0 < a_1 < \ldots < a_k =1$ e $c_i \in \mathbb{R}$ per ogni $i=0,\ldots , k$. \smallskip\\ Da questa relazione e dalla (\ref{eq:1.1}) segue che
\begin{eqnarray*}
& &\lim_{N \to \infty}\frac{1}{N} \sum_{n=1} ^{N} f(\{x_n\}) =  \lim_{N \to \infty}\frac{1}{N} \sum_{n=1} ^{N} \sum_{i=0}^{k-1}c_i\chi_{[a_i,a_{i+1}[}(\{x_n\})=\\
&=& \sum_{i=0}^{k-1}c_i(a_{i+1}-a_i) = \sum_{i=0}^{k-1} \int_0 ^{1} c_i\chi_{[a_i,a_{i+1}[}(x)dx =\int_0 ^{1} f(x) dx\ .
\end{eqnarray*}
Sia ora $f$ una funzione reale continua su $\bar{I}$. Si tratta dunque di una funzione integrabile secondo Riemann e quindi, fissato $\varepsilon > 0$, esistono due funzioni semplici $f_1, f_2$ tali che
$$f_1(x)\leq f(x) \leq f_2(x)\quad \forall x \in \bar{I}$$
ed inoltre
$$\int_0 ^{1} [f_2(x)-f_1(x)]dx\leq \varepsilon \ .$$
Pertanto
$$\int_0 ^{1} f_1(x)dx \leq \int_0 ^{1} f_2(x)dx \leq \int_0 ^{1} f_1(x)dx + \varepsilon \ .$$
Da questa catena di disuguaglianze e dalla (\ref{eq:1.1}) si ha che
\begin{eqnarray*}
& &\int_0 ^{1} f(x)dx-\varepsilon \leq  \int_0 ^{1} f_1(x)dx=\lim_{N \to \infty}\frac{1}{N} \sum_{n=1} ^{N} f_1(\{x_n\})\\
&\leq & \underset{N \to \infty}{\underline{\lim}}\frac{1}{N} \sum_{n=1} ^{N} f(\{x_n\}) \leq \underset{N \to \infty}{\overline{\lim}}\frac{1}{N} \sum_{n=1} ^{N} f(\{x_n\}) \leq  \lim_{N \to \infty}\frac{1}{N} \sum_{n=1} ^{N} f_2(\{x_n\})\\
&=&\int_0 ^{1} f_2(x)dx \leq \int_0 ^{1} f_1(x) dx + \varepsilon \leq \int_0 ^{1} f(x) dx + \varepsilon\ .
\end{eqnarray*}
Dall'arbitrariet\`{a} di $\varepsilon$ segue la tesi.\\ 
Viceversa, sia $(x_n)_{n \in \mathbb{N}}$ una successione di numeri reali che verifica la $(\ref{eq:1.3})$ per ogni funzione continua in $\bar{I}$ a valori reali. Sia $[a,b[\, \subset I$ un arbitrario sottointervallo di $I$. Fissato $\varepsilon > 0$, esistono due funzioni continue $g_1$ e $g_2$ tali che
\begin{equation*}
g_1(x)\leq \chi_{[a, b[}(x)\leq g_2(x) \quad \forall x \in \bar{I}
\end{equation*}
ed inoltre
\begin{equation*}
\int_0 ^{1}[g_2(x)-g_1(x)]dx \leq \varepsilon \ .
\end{equation*}
Allora
\begin{eqnarray*}
b-a- \varepsilon &=&  \int_0 ^{1} \chi_{[a, b[}(x) dx - \varepsilon \leq \int_0 ^{1} g_2(x)dx - \varepsilon \leq \int_0 ^{1} g_1(x)dx  \\
&=& \lim_{N \to \infty}\frac{1}{N} \sum_{n=1} ^{N} g_1(\{x_n\})\leq \underset{N \to \infty}{\underline{\lim}}\frac{1}{N} \sum_{n=1} ^{N} \chi_{[a, b[}(\{x_n\}) \\
& \leq & \underset{N \to \infty}{\overline{\lim}}\frac{1}{N} \sum_{n=1} ^{N} \chi_{[a, b[}\{(x_n\})\leq \lim_{N \to \infty}\frac{1}{N} \sum_{n=1} ^{N} g_2(\{x_n\})\\ 
&=& \int_0 ^{1} g_2(x)dx \leq \int_0 ^{1} \chi_{[a, b[}(x) dx + \varepsilon = b-a + \varepsilon \ .
\end{eqnarray*}
Essendo $\varepsilon$ arbitrariamente piccolo, segue la (\ref{eq:1.1}) e quindi la tesi.
\end{proof}
Conseguenza del precedente teorema \`{e} il seguente
\begin{cor}
Una successione $(x_n)_{n \in \mathbb{N}}$ \`{e} u.d. mod 1 se e solo se per ogni funzione $f$ integrabile secondo Riemann su $\bar{I}$ vale la $(\ref{eq:1.3})$.
\end{cor}
\begin{proof}
Per la condizione sufficiente occorre solo osservare che ogni funzione continua \`{e} integrabile secondo Riemann e quindi applicare il teorema precedente. \\ La parte non banale della dimostrazione consiste nel provare la condizione necessaria. Procediamo in maniera analoga a quanto fatto nella prima parte del Teorema \ref{thm:1.1.3}. Supponiamo che $f$ sia una funzione integrabile secondo Riemann e che valga la condizione $(\ref{eq:1.1})$. Allora, fissato $\varepsilon > 0$, esistono due funzioni semplici $h_1, h_2$ tali che
\begin{equation*}
h_1(x)\leq f(x)\leq h_2(x) \quad \forall x \in \bar{I}
\end{equation*}
e
\begin{equation*}
\int_0 ^{1}[h_2(x)-h_1(x)]dx \leq \frac{\varepsilon}{2} \ .
\end{equation*}
Allora esiste $\bar{N}$ tale che per ogni $N > \bar{N}$ si ha
\begin{equation*}
\left\vert \frac{1}{N} \sum_{n=1} ^{N} h_k(\{x_n\}) - \int_0 ^{1} h_k(x)dx \right\vert < \frac{\varepsilon}{2} \qquad \text{per} \ k=1,2
\end{equation*}
perch\`{e} la $(\ref{eq:1.3})$ vale per funzioni semplici. Di conseguenza
\begin{eqnarray*}
\int_0 ^{1} h_1(x)dx - \frac{\varepsilon}{2} & < &  \frac{1}{N} \sum_{n=1} ^{N} h_1(\{x_n\})\leq \frac{1}{N} \sum_{n=1} ^{N} f\{(x_n\}) \\
&\leq & \frac{1}{N} \sum_{n=1} ^{N} h_2(\{x_n\}) < \int_0 ^{1} h_2(x)dx + \frac{\varepsilon}{2} \ .
\end{eqnarray*}
Dalla condizione
\begin{equation*}
\int_0 ^{1} h_1(x) dx \leq \int_0 ^{1} f(x) dx \leq \int_0 ^{1} h_2(x) dx
\end{equation*}
segue la tesi.
\end{proof}
Vediamo ora una generalizzazione del Teorema \ref{thm:1.1.3} al caso di funzioni periodiche a valori complessi, che ci sar\`{a} utile nel paragrafo successivo.
\begin{cor}
\label{cor:1.1.5}
Una successione $(x_n)_{n \in \mathbb{N}}$ \`{e} u.d. mod 1 se e solo se per ogni funzione continua $f:\mathbb{R} \longrightarrow \mathbb{C}$ periodica di periodo unitario si ha
\begin{equation}
 \lim_{N \to \infty}\frac{1}{N} \sum_{n=1} ^{N} f(x_n)= \int_0 ^{1} f(x)dx \ . \label{eq:1.4}
\end{equation}
\end{cor}
\begin{proof}
Supponiamo che $f:\mathbb{R} \rightarrow \mathbb{C}$ sia continua. Allora, applicando il Teorema \ref{thm:1.1.3} separatamente alla parte reale ed alla parte immaginaria di $f$, si ottiene che la (\ref{eq:1.3}) 
vale anche per funzioni a valori complessi. Per ottenere la (\ref{eq:1.4}) \`{e} sufficiente osservare che la periodicit\`{a} della $f$ implica che $f(\{x_n\})=f(x_n)$. \\ Per dimostrare l'altra implicazione basta richiedere che le funzioni $g_1$ e $g_2$, utilizzate nella dimostrazione del Teorema \ref{thm:1.1.3}, soddisfino l'ulteriore condizione $g_1(0)=g_1(1)$ e $g_2(0)=g_2(1)$, cosicch\`{e} entrambe possono essere prolungate su tutto $\mathbb{R}$ per periodicit\`{a} e soddisfano la (\ref{eq:1.4}).  
\end{proof}
La prossima  proposizione rappresenta un metodo efficace per generare nuove successioni u.d. a partire da una successione di cui \`{e} nota la propriet\`{a} di uniforme distribuzione.
\begin{prop}
Se $(x_n)_{n \in \mathbb{N}}$ \`{e} una successione u.d. mod 1 e se $(y_n)_{n \in \mathbb{N}}$ \`{e} una successione tale che $\lim_{n \to \infty}|x_n - y_n|=0$, allora $(y_n)_{n \in \mathbb{N}}$ \`{e} u.d. mod 1.
\end{prop}
\begin{proof}
Fissiamo $a\in\ ]0,1[$ e $\varepsilon > 0$ tale che $\varepsilon < \min \{a, 1-a\}$. Allora dall'ipotesi $\lim_{n \to \infty}|x_n - y_n|=0$ segue che, per $n$ sufficientemente grande, 
\begin{eqnarray*}
& &\text{se}\quad \{x_n\}<a-\varepsilon \qquad \text{allora} \qquad \{y_n\}<a\\
& &\text{e\ se}\quad \{y_n\}<a \qquad \text{segue\ che} \qquad \{x_n\}<a+\varepsilon\ .
\end{eqnarray*}
Pertanto
\begin{eqnarray*}
& & a- \varepsilon = \lim_{N \to \infty}\frac{1}{N} \sum_{n=1} ^{N}  \chi_{[0, a- \varepsilon[}(\{x_n\})\leq \underset{N \to \infty}{\underline{\lim}}\frac{1}{N} \sum_{n=1} ^{N} \chi_{[0, a[}(\{y_n\}) \\
&\leq & \underset{N \to \infty}{\overline{\lim}}\frac{1}{N} \sum_{n=1} ^{N} \chi_{[0, a[}(\{y_n\})\leq \lim_{N \to \infty}\frac{1}{N} \sum_{n=1} ^{N} \chi_{[0, a+ \varepsilon[}(\{x_n\}) = a+\varepsilon \ .
\end{eqnarray*}
Nuovamente dall'arbitrariet\`{a} di $\varepsilon$ segue la tesi.
\end{proof}

\section{Il Criterio di Weyl}
In questo paragrafo tratteremo l'importante criterio di Weyl. 
Esso prende in esame le funzioni della forma $f(x)=e^{2\pi ihx}$, dove $h$ \`{e} un intero non nullo. Dimostreremo che queste funzioni soddisfano le ipotesi del Corollario \ref{cor:1.1.5}, e pertanto esse forniscono un criterio sufficiente a determinare se una successione \`{e} u.d..
\begin{thm}[Criterio di Weyl]
\label{thm:1.2.1}
La successione $(x_n)_{n\in\mathbb{N}}$ \`{e} u.d. mod 1 se e solo se
\begin{equation}
\lim_{N \to \infty}\frac{1}{N} \sum_{n=1} ^{N} e^{2\pi ihx_n} =0 \label{eq:1.5}
\end{equation}
per ogni intero $h\neq 0$.
\end{thm}
\begin{proof}
Sia $(x_n)$ una successione u.d. mod 1. Allora, per il Corollario \ref{cor:1.1.5} e per la periodicit\`{a} della funzione integranda, si ha
\begin{equation*}
\lim_{N \to \infty}\frac{1}{N} \sum_{n=1} ^{N} e^{2\pi ihx_n}= \int_0 ^{1} e^{2\pi ihx} dx=\dfrac{e^{2\pi ihx}}{2\pi ih} \Bigg|_{x=0}^{x=1} =0 \ .\smallskip
\end{equation*}
Viceversa, supponiamo che la successione verifichi la condizione (\ref{eq:1.5}) e che la $f$ sia una generica funzione continua su $[0,1]$, con $f(0)=f(1)$. Per il teorema di approssimazione di Weierstrass, fissato $\varepsilon >0$ esiste un polinomio trigonometrico $\psi (x)$, cio\`{e} una combinazione lineare di funzioni del tipo $e^{2\pi ihx}$, tale che
\begin{equation*}
\sup_{x \in [0,1]}|f(x)-\psi (x)|< \frac{\varepsilon}{3}\ .
\end{equation*}
Allora
\begin{eqnarray*}
& &\left |\int_0 ^{1} f(x)dx-\frac{1}{N} \sum_{n=1} ^{N}f(\{x_n\})\right | \leq \left |\int_0 ^{1} (f(x)-\psi (x))dx\right | + \\
&+&\left |\int_0 ^{1}\psi (x) - \frac{1}{N} \sum_{n=1} ^{N} \psi (x_n)\right | + \left |\frac{1}{N} \sum_{n=1} ^{N} (\psi (x_n)-f(\{x_n\}))\right | \leq \varepsilon
\end{eqnarray*}
poich\`{e}, per la (\ref{eq:1.5}), il secondo addendo \`{e} maggiorato da $\frac{\varepsilon}{3}$ per $N$ sufficientemente grande.
\end{proof}
\subsection*{Applicazioni del Criterio di Weyl}
Volgiamo ora la nostra attenzione ad alcune applicazioni del criterio di Weyl. Nella prima parte di questo paragrafo prendiamo in esame una successione numerica speciale.
\begin{eso} 
\rm{Sia $\theta$ un numero irrazionale. Allora la successione $(n\theta)_{n\in \mathbb{N}}$ \`{e} u.d. mod 1. Infatti
\begin{equation*}
\left| \frac{1}{N} \sum_{n=1} ^{N} e^{2\pi ihn\theta} \right|=\left| \frac{1}{N}\left( \sum_{n=0} ^{N} e^{2\pi ihn\theta}-1 \right)\right|=\left| \frac{1}{N}\left( \frac{1-( e^{2\pi ih\theta})^{N+1}}{1-e^{2\pi ih\theta}}-1\right)\right| \medskip
\end{equation*}
\begin{equation*}
=\frac{1}{N}\left| \frac{1-e^{2\pi ih\theta(N+1)}-1+ e^{2\pi ih\theta}}{1-e^{2\pi ih\theta}}\right| = \frac{1}{N}\left| e^{2\pi ih\theta}\right|\left| \frac{1-e^{2\pi ihN \theta}}{1-e^{2\pi ih\theta}} \right| =\medskip
\end{equation*}
\begin{equation*}
=\frac{1}{N}\left| \frac{1-e^{2\pi ihN \theta}}{1-e^{2\pi ih\theta}} \right|\ .
\end{equation*}
\medskip
A questo punto osserviamo che, posto $\alpha = \pi h \theta$, si ha:
\begin{eqnarray*}
|1-e^{2 i \alpha}|&=& |1-\cos(2 \alpha)-i \sin(2\alpha)|=\sqrt{(1-\cos(2\alpha))^{2}+\sin^{2}(2\alpha)} \\
& = & \sqrt{2(1-\cos(2\alpha))}=\sqrt{2(1-\cos^{2}\alpha+\sin^{2}\alpha)}=2|\sin \alpha|\ .
\end{eqnarray*}
Allora vale la $(\ref{eq:1.5})$ perch\`{e}
\begin{equation*}
\left| \frac{1}{N} \sum_{n=1} ^{N} e^{2\pi ihn\theta} \right|= \frac{1}{N}\left|\frac{1-e^{2\pi ihN \theta}}{2\sin( \pi h \theta)} \right| \leq \frac{1}{N|\sin (\pi h \theta)|}
\end{equation*}
\`{e} una quantit\`{a} infinitesima per $N$ che tende ad infinito. Dal Teorema \ref{thm:1.2.1} segue che la successione $(n\theta)$ \`{e} u.d. mod 1.}
\bigskip \\
\rm{Il fatto che la successione $(n\theta)_{n\in\mathbb{N}}$ sia u.d. mod 1 se $\theta$ \`{e} un numero irrazionale fu stabilito in maniera indipendente da P. Bohl in \cite{Bohl}, W. Sierpi\'{n}ski in \cite{Sierpinski1} e \cite{Sierpinski2} e Weyl in \cite{Weyl1} nel 1909-1910. L'importanza di questo risultato risiede nel fatto che esso segna l'inizio della teoria dell'uniforme distribuzione. 
Infatti la successione $(n\theta)$ fu il primo esempio di successione u.d. mod 1 fornito dallo stesso Weyl ed esso perfeziona l'enunciato di un teorema di Kronecker, il quale asserisce che la successione $(n\theta)$ \`{e} densa in $[0,1]$.}
\end{eso}
Presentiamo ora un risultato che \`{e} una conseguenza del criterio di Weyl e che ci consentir\`{a} di determinare molte classi di successioni u.d..
\begin{thm}
\label{thm:1.2.3}
Sia $(f(n))_{n\in\mathbb{N}}$ una successione di numeri reali tale che la successione $\Delta f(n)=f(n+1)-f(n),\ n=1,2,\ldots,$ sia monotona. Supponiamo inoltre che
\begin{equation}
\lim_{n \to \infty} \Delta f(n)=0 \quad e \quad \lim_{n \to \infty} n|\Delta f(n)|=+\infty \ . \label{eq:1.6}
\end{equation}
Allora la successione $(f(n))$ \`{e} u.d. mod 1.
\end{thm}
\begin{proof}
Fissiamo due numeri reali $u$ e $v$. Allora, integrando per parti $4 \pi^{2} \int_0 ^{u-v}(u-v-w)e^{2\pi iw}dw$ ed effettuando le opportune maggiorazioni si ha che
\begin{equation*}
|e^{2\pi iu}-e^{2\pi iv}-2\pi i (u-v)e^{2\pi iv}|= |e^{2\pi iv}|\ |e^{2\pi i(u-v)}-1-2\pi i(u-v)|=
\end{equation*}
 \begin{equation*}
 = 4 \pi^{2}\left | \int_0 ^{u-v}(u-v-w)e^{2\pi iw}dw \right| \leq 4 \pi^{2}\int_0 ^{u-v}|u-v-w|dw=
 \end{equation*}
 \begin{equation*}
 = 4 \pi^{2}\left.\left[\left( (u-v)w -\frac{w^{2}}{2} \right)\right]\right |_{w=0}^{w=u-v}  = 4 \pi^{2}\dfrac{(u-v)^{2}}{2} = 2 \pi^{2}(u-v)^{2}\ .
 \medskip
 \end{equation*}
Ponendo nella disuguaglianza precedente $u=hf(n+1)$ e $v=hf(n)$, con $h \in \mathbb{Z},\ h\neq 0$, si ottiene \medskip
\begin{equation*}
\left |e^{2\pi ihf(n+1)}-e^{2\pi ihf(n)}-2\pi ih\Delta f(n)e^{2\pi ihf(n)}\right |\leq 2 \pi^{2}h^{2}\left(\Delta f(n)\right)^{2} \qquad (n\geq 1) \ .
\end{equation*}
Dividendo ambo i membri della disuguaglianza per $\left|\Delta f(n)\right |$ si ha
\begin{equation*}
\left | \dfrac{e^{2\pi ihf(n+1)}}{\Delta f(n)}- \dfrac{e^{2\pi ihf(n)}}{\Delta f(n)} -2\pi ihe^{2\pi ihf(n)} \right | \leq 2 \pi^{2}h^{2}\left |\Delta f(n)\right |\ . \medskip
\end{equation*}
Da questa relazione segue che \medskip
\begin{equation*}
\left | \dfrac{e^{2\pi ihf(n+1)}}{\Delta f(n+1)}- \dfrac{e^{2\pi ihf(n)}}{\Delta f(n)} -2\pi ihe^{2\pi ihf(n)} \right |
\medskip
\end{equation*}
\begin{equation*}
\leq \left | \dfrac{e^{2\pi ihf(n+1)}}{\Delta f(n+1)}- \dfrac{e^{2\pi ihf(n+1)}}{\Delta f(n)}\right |+ \left | \dfrac{e^{2\pi ihf(n+1)}}{\Delta f(n)}- \dfrac{e^{2\pi ihf(n)}}{\Delta f(n)} -2\pi ihe^{2\pi ihf(n)}\right |
\medskip
\end{equation*}
\begin{equation*}
\leq \left | \dfrac{1}{\Delta f(n+1)}-\dfrac{1}{\Delta f(n)} \right | + 2 \pi^{2}h^{2}|\Delta f(n)|
\end{equation*}
per $n \geq 1$. \\ Dunque otteniamo la stima
$$\left | 2\pi ih \sum_{n=1} ^{N-1}e^{2\pi ihf(n)} \right |$$ \medskip
$$= \left | \sum_{n=1} ^{N-1}\left(2\pi ie^{2\pi ihf(n)}- \dfrac{e^{2\pi ihf(n+1)}}{\Delta f(n+1)}+\dfrac{e^{2\pi ihf(n)}}{\Delta f(n)} \right)+\dfrac{e^{2\pi ihf(N)}}{\Delta f(N)}-\dfrac{e^{2\pi ihf(1)}}{\Delta f(1)} \right |$$ \medskip
$$\leq \sum_{n=1} ^{N-1}\left |2\pi ie^{2\pi ihf(n)}- \dfrac{e^{2\pi ihf(n+1)}}{\Delta f(n+1)}+\dfrac{e^{2\pi ihf(n)}}{\Delta f(n)} \right | + \left | \dfrac{e^{2\pi ihf(N)}}{\Delta f(N)} \right |+ \left | \dfrac{e^{2\pi ihf(1)}}{\Delta f(1)} \right |$$ \medskip
$$\leq \sum_{n=1} ^{N-1}\left |\dfrac{1}{\Delta f(n+1)}-\dfrac{1}{\Delta f(n)} \right |+ \sum_{n=1} ^{N-1} 2 \pi^{2}h^{2}|\Delta f(n)|+ \frac{1}{|\Delta f(N)|}+\dfrac{1}{|\Delta f(1)|} \ . \medskip$$
Da ci\`{o}, sfruttando l'ipotesi di monotonia si ottiene \medskip
\begin{equation*}
\left |\sum_{n=1} ^{N-1}e^{2\pi ihf(n)}\right | \leq \frac{1}{2 \pi |h|}\left( \frac{1}{|\Delta f(N)|}+\dfrac{1}{|\Delta f(1)|} \right)+ \frac{1}{2 \pi |h|}\left | \frac{1}{|\Delta f(N)|}+\dfrac{1}{|\Delta f(1)|} \right | +
\end{equation*}
\begin{equation*}
+\pi |h|\sum_{n=1} ^{N-1}|\Delta f(n)| \leq \frac{1}{\pi |h|}\left( \frac{1}{|\Delta f(N)|}+\dfrac{1}{|\Delta f(1)|} \right)+\pi |h|\sum_{n=1} ^{N-1}|\Delta f(n)|\ .
\medskip 
\end{equation*}
Allora
\begin{equation*}
\left | \frac{1}{N} \sum_{n=1} ^{N} e^{2\pi ihf(n)} \right | \leq \frac{1}{\pi |h|}\left( \frac{1}{N|\Delta f(N)|}+\dfrac{1}{N|\Delta f(1)|} \right)+\frac{\pi |h|}{N}\sum_{n=1} ^{N-1}|\Delta f(n)|\ .
\end{equation*}
Tenendo conto delle ipotesi (\ref{eq:1.6}), dal Teorema \ref{thm:1.2.1} segue la tesi.
\end{proof}
Conseguenza immediata del teorema precedente (e, quindi, del criterio di Weyl) \`{e} il noto
\begin{thm}[di Fej\'{e}r]
\label{thm:1.2.4}
Sia $f$ una funzione definita e differenziabile per $x \geq 1$. Se $f'$ \`{e} monotona per ogni $x\geq 1$ ed inoltre
\begin{equation}
\lim_{x \to +\infty}f'(x)=0  \label{eq:1.7} \\
\end{equation}
e
\begin{equation}
\lim_{x \to +\infty}x|f'(x)|=+\infty \ ,\label{eq:1.8}
\end{equation}
allora la successione $(f(n))_{n \in \mathbb{N}}$ \`{e} u.d. mod 1.
\end{thm}
\begin{proof}
Per il teorema del valor medio esiste $\tau_n\in [n,n+1]$ tale che
\begin{equation*}
\Delta f(n)=f'(\tau_n)\ .
\end{equation*}
Pertanto, $(\Delta f(n))$ \`{e} monotona e soddisfa $\lim_{n \to \infty}\Delta f(n)=0$ in virt\`{u} dell'ipotesi (\ref{eq:1.7}).\\
Quindi, per poter applicare il Teorema \ref{thm:1.2.3}, rimane da verificare che $\Delta f(n)$ soddisfa la seconda in (\ref{eq:1.6}).
Supponiamo dapprima che la $f'$ sia decrescente. Allora, per $n$ sufficientemente grande risulta che
\begin{equation*}
(n+1)f'(n+1)-f'(n+1)=nf'(n+1) \leq n f'(\tau_n) \leq nf'(n).
\end{equation*}
Passando al limite e tenendo conto delle ipotesi (\ref{eq:1.7}) e (\ref{eq:1.8}), otteniamo che $\lim_{n \to \infty} n|\Delta f(n)|=+\infty$.\\ Se, invece, $f'$ \`{e} crescente, basta scambiare $f$ con $-f$.\\In entrambi i casi possiamo applicare il teorema precedente ed ottenere la tesi.
\end{proof}
\begin{es}
\rm{Come abbiamo gi\`{a} osservato, il teorema di Fej\'{e}r fornisce un criterio sufficiente per stabilire quando una successione di punti  \`{e} u.d. mod 1. Vediamone alcuni esempi proposti in \cite{Kuipers_Niederreiter}.\bigskip \\
$\mathbf{1.}$ $(\alpha n^{\sigma} \log^{\tau}n)_{n \geq 2}$, con $\alpha \neq 0$, $\tau\in\mathbb{R}$, $0<\sigma<1$. \\Vogliamo provare che la funzione $f(x)=\alpha x^{\sigma} \log^{\tau}x$, con $x\geq 2$, associata alla successione di termine generale $f(n)=\alpha n^{\sigma} \log^{\tau}n$, verifica le ipotesi del teorema di Fej\'{e}r. \\ Osserviamo innanzitutto che si tratta di una funzione monotona. Infatti, si pu\`{o} provare che la derivata prima
\begin{equation*}
f'(x)= \alpha \sigma  x^{\sigma -1} \log^{\tau}x+\frac{\alpha x^{\sigma}\tau \log^{\tau -1}x}{x}=\alpha x^{\sigma -1} \log^{\tau -1}x(\sigma\log x+\tau)
\end{equation*}
\`{e} sempre monotona. Inoltre, per $x\to +\infty$ questa derivata tende a zero e
\begin{equation*}
x |f'(x)|= |\alpha x^{\sigma} \log^{\tau -1}x(\sigma\log x+\tau)| \rightarrow +\infty \ .
\end{equation*}
Pertanto, poich\`{e} la $f$ verifica le condizioni (\ref{eq:1.7}) e (\ref{eq:1.8}), possiamo applicare il teorema di Fej\'{e}r e dedurre che la successione \`{e} u.d. mod 1.
\bigskip \\
$\mathbf{2.}$ $(\alpha \log^{\tau}n)_{n\geq 1}$, con $\alpha \neq 0$ e $\tau >1$.\\ Indichiamo con $g(n)$ il termine generale della successione e poniamo $g(x)=\alpha\log^{\tau}x$ per $x\geq 1$. Allora
\begin{equation*}
g'(x)=\frac{ \alpha \tau \log^{\tau -1}x}{x}\longrightarrow0 \qquad (x \to +\infty)
\end{equation*}
e inoltre si vede che $g'$ \`{e} monotona e
\begin{equation*}
x|g'(x)|= |\alpha| \tau \log^{\tau -1}x\longrightarrow +\infty \qquad (x \to +\infty)\ .
\end{equation*}
Per il Teorema \ref{thm:1.2.4}, la successione $(\alpha \log^{\tau}n)_{n\geq 1}$ \`{e} u.d. mod 1.
\bigskip \\
$\mathbf{3.}$ $(\alpha n \log^{\tau}n)_{n\geq 2}$, con $\alpha \neq 0$ e $\tau <0$. \\ Anche in questo caso possiamo verificare che sono soddisfatte le ipotesi del teorema di  Fej\'{e}r. Infatti, indicato con $h(n)$ il termine generale della successione e posto $h(x)=\alpha x\log^{\tau}x$ per ogni $x\geq 2$, \`{e} facile osservare che
\begin{equation*}
h'(x)=\alpha \log^{\tau}x + \frac{ \alpha \tau x \log^{\tau -1}x}{x}=\alpha\log^{\tau-1}x(\log x+\tau)\longrightarrow 0\quad (x \to +\infty)\ ,
\end{equation*}
che $h'$ \`{e} monotona e che 
\begin{equation*}
x|h'(x)|=|\alpha \log^{\tau-1}x (\log x + \tau)| \longrightarrow +\infty \qquad (x \to +\infty)\ .
\end{equation*}
Quindi la successione data risulta u.d. mod 1.
}
\end{es}
\section{Il caso multidimensionale}
Generalizziamo quanto detto finora al caso di successioni di vettori, partendo da alcune notazioni. \\ Sia $s$ un intero, $s\geq 2$. Siano $\textbf{a}=(a_1,\ldots ,a_s)$ e $\textbf{b}=(b_1,\ldots ,b_s)$ due vettori di $\mathbb{R}^{s}$. Diciamo che $\textbf{a} < \textbf{b}$ (rispettivamente, $\textbf{a} \leq \textbf{b}$) se $a_i < b_i$ (rispettivamente, $a_i\leq b_i$) per ogni $i=1,\ldots ,s$. L'insieme dei punti $\textbf{x}\in \mathbb{R}^{s}$ tali che $\textbf{a}\leq  \textbf{x} <\textbf{b}$ (rispettivamente, $\textbf{a}\leq \textbf{x}\leq \textbf{b}$) sar\`{a} denotato da $[\textbf{a},\textbf{b}[$ (rispettivamente, $[\textbf{a},\textbf{b}]$) e verr\`{a} chiamato sottointervallo. \\Il cubo unitario $s$-dimensionale $I^{s}$ \`{e} l'intervallo $[\textbf{0},\textbf{1}[$, dove $\textbf{0}=(0,\ldots, 0)$ e $\textbf{1}=(1,\ldots,1)$. \\ La parte intera di $\textbf{x}=(x_1,\ldots,x_s)$ \`{e} $\lfloor\textbf{x}\rfloor=(\lfloor x_1\rfloor,\ldots,\lfloor x_s\rfloor)$ e la parte frazionaria di $\textbf{x}$ \`{e} $\{\textbf{x}\}=(\{x_1\},\ldots,\{x_s\})$.\\ Sia $(\textbf{x}_n)$, $n=1,\ldots,$ una successione di vettori in $\mathbb{R}^{s}$.
\begin{defn}
\rm{La successione $(\textbf{x}_n)_{n\in\mathbb{N}}$ \`{e} detta \emph{uniformemente distribuita modulo 1} ($u.d.\ mod\ \mathit{1}$) in $\mathbb{R}^{s}$ se
\begin{equation}
\lim_{N \to \infty}\frac{1}{N} \sum_{n=1} ^{N} \chi_{[\textbf{a}, \textbf{b}[}(\{\textbf{x}_n\})=\prod_{j=1}^{s}(b_j-a_j)
\end{equation}
per tutti i sottointervalli $[\textbf{a},\textbf{b}[\, \subset I^{s}$.}
\end{defn}
Considerare il caso multidimensionale ci permette di fornire una definizione di uniforme distribuzione per successioni di numeri complessi.
\begin{defn}
\rm{Sia $(z_n)_{n\in \mathbb{N}}$ una successione di numeri complessi. Siano Re $z_n=x_n$ e Im $z_n=y_n$. Allora la successione $(z_n)$ \`{e} detta \emph{uniformemente distribuita modulo 1} ($u.d.\ mod\ \mathit{1}$) in $\mathbb{C}$ se la successione $((x_n, y_n))_{n\in\mathbb{N}}$ \`{e} u.d. mod 1 in $\mathbb{R}^{2}$.}
\end{defn}
Ricordando che il cubo unitario $s$-dimensionale chiuso $\bar{I}^{s}$ \`{e} l'intervallo $[\textbf{0}, \textbf{1}]$, possiamo enunciare il seguente:
\begin{thm}
Una successione $(\textbf{x}_n)_{n\in\mathbb{N}}$ \`{e} u.d. mod 1 in $\mathbb{R}^{s}$ se e solo se per ogni funzione $f$ continua a valori complessi definita su $\bar{I}^{s}$ vale la seguente relazione:
\begin{equation*}
\lim_{N \to \infty}\frac{1}{N} \sum_{n=1} ^{N} f(\{\textbf{x}_n\})=\int_{\bar{I}^{s}}f(\textbf{x})dx \ .
\end{equation*}
\smallskip
\end{thm}
Infine, se per ogni $\textbf{x}=(x_1,\ldots, x_s)$ e $\textbf{y}=(y_1,\ldots, y_s)$ in $\mathbb{R}^{s}$ denotiamo con $\langle \textbf{x}, \textbf{y} \rangle$ il prodotto scalare standard in $\mathbb{R}^{s}$, ossia $\langle \textbf{x}, \textbf{y} \rangle =x_1y_1+x_2y_2+\ldots +x_sy_s$, possiamo scrivere il criterio di Weyl nel caso mutidimensionale, omettendone la dimostrazione che \`{e} essenzialmente simile a quella del caso $s=1$. 
\begin{thm}[Criterio di Weyl]
Una successione $(\textbf{x}_n)_{n\in\mathbb{N}}$ \`{e} u.d. mod 1 in $\mathbb{R}^{s}$ se e solo se per ogni $\textbf{h}\in \mathbb{Z}^{s}$, $\textbf{h}\neq$ \rm{\textbf{0}}, risulta
\begin{equation*}
\lim_{N \to \infty}\frac{1}{N} \sum_{n=1} ^{N}e^{2\pi i \langle \textbf{h}, \textbf{x}_n \rangle}=0\ .
\end{equation*}
\end{thm}
\section{Discrepanza}
Lo scopo di questo paragrafo \`{e} quello di presentare i concetti fin qui esposti da un punto di vista quantitativo, e non solo puramente qualitativo come fatto sinora. Per farlo introduciamo il concetto di \emph{discrepanza}, termine probabilmente coniato da J. C. van der Corput, matematico olandese che per primo, nel 1935, pubblic\`{o} in \cite{vanderCorput} una procedura che consente di generare successioni u.d. mod 1 e che inoltre stabil\`{i} risultati quantitativi riguardanti il comportamento nella distribuzione di una successione. \\ Questa quantit\`{a} ci permette di distinguere tra \lq \lq buone\rq\rq e \lq\lq cattive\rq\rq successioni u.d.. Infatti, tra le successioni u.d. mod 1 ve ne sono alcune che presentano una buona distribuzione, al contrario di altre che sono soltanto uniformemente distribuite.
\begin{defn}
\label{def:1.4.1}
\rm{Sia $x_1,\ldots, x_N$ una successione finita di numeri reali. La quantit\`{a}
\begin{equation}
D_N=D_N(x_1,\ldots, x_N)=\sup_{0\leq a < b\leq 1}\left|\frac{1}{N} \sum_{n=1} ^{N} \chi_{[a, b[}(\{x_n\})-(b-a) \right| 
\end{equation}
\`{e} detta \emph{discrepanza} della successione.}\\ Per una successione infinita $w$ di numeri reali (o per una successione di numeri reali finita contenente almeno $N$ termini), la discrepanza $D_N(w)$ si definisce come la discrepanza del segmento iniziale formato dai primi $N$ termini di $w$.
\end{defn}
\begin{thm}
\label{thm:1.4.2}
La successione $w$ \`{e} u.d. mod 1 se e solo se 
\begin{equation}
\lim_{N\to \infty}D_N(w)=0 \ . \label{eq:1.10}
\end{equation}
\end{thm}
\begin{proof}
Osserviamo che, per la Definizione \ref{def:1.4.1}, la (\ref{eq:1.10}) implica direttamente la (\ref{eq:1.1}), ossia la definizione di successione u.d. mod 1, per tutti gli intervalli $[a,b[\ \subset I$. Quindi si tratta solo di dimostrare che ogni successione u.d. mod 1 soddisfa la (\ref{eq:1.10}). Per far ci\`{o}, scegliamo un intero positivo $m$, con $m\geq 2$. \\ Per $0\leq k\leq m-1$ denotiamo con $I_k$ l'intervallo $I_k=\left[ \frac{k}{m}, \frac{k+1}{m}\right[$.\\ Dalla Definizione \ref{def:1.1.1} sappiamo che se una successione \`{e} u.d. mod 1 allora esiste $N_0=N_0(m)$ tale che per ogni $N\geq N_0$ si ha 
\begin{equation*}
\frac{1}{m}\left(1-\frac{1}{m}\right)\leq \frac{1}{N} \sum_{n=1}^{N} \chi_{I_k}(\{x_n\})\leq \frac{1}{m}\left(1+\frac{1}{m}\right) \ .
\end{equation*}
Consideriamo, ora, un arbitrario intervallo $J=[a, b[\ \subset I$. Chiaramente esistono due intervalli $J_1,J_2$, unioni finite di intervalli $I_k$, tali che $J_1\subset J\subset J_2$ e per cui sono verificate le seguenti relazioni:
\begin{equation}
\lambda (J) - \lambda(J_1)< \frac{2}{m}\label{eq:1.11}
\end{equation}
e
\begin{equation}\label{eq:1.12}
\lambda(J_2)-\lambda(J)<\frac{2}{m}\ ,
\end{equation}
dove con $\lambda$ indichiamo la misura di Lebesgue unidimensionale. \\Allora, per ogni $N>N_0$ risulta
\begin{eqnarray*}
\lambda (J_1)\left( 1-\frac{1}{m} \right) &\leq &\frac{1}{N}\sum_{n=1}^{N} \chi_{J_1}(\{x_n\})\leq \frac{1}{N}\sum_{n=1}^{N} \chi_{J}(\{x_n\})\\
&\leq & \frac{1}{N}\sum_{n=1}^{N} \chi_{J_2}(\{x_n\}) \leq \lambda(J_2)\left(1+\frac{1}{m}\right)\ .
\end{eqnarray*}
Da questa catena di disuguaglianze e dalle relazioni (\ref{eq:1.11}) e (\ref{eq:1.12}) si ottiene che
\begin{equation*}
\left(\lambda (J)-\frac{2}{m}\right)\left(1-\frac{1}{m}\right)<\frac{1}{N}\sum_{n=1}^{N} \chi_{J}(\{x_n\})<\left(\lambda(J)+\frac{2}{m}\right)\left(1+\frac{1}{m}\right)\ .
\end{equation*}
Infine, poich\`{e} $J\subset I$, dalla propriet\`{a} di monotonia della misura di Lebesgue segue che $\lambda(J)\leq 1$ e quindi
\begin{equation*}
-\frac{3}{m}-\frac{2}{m^{2}}<\frac{1}{N}\sum_{n=1}^{N} \chi_{J}(\{x_n\})- \lambda(J)<\frac{3}{m}+\frac{2}{m^{2}} \qquad \forall N\geq N_0\ ,
\end{equation*}
che in forma equivalente pu\`{o} essere scritta come
\begin{equation*}
\left| \frac{1}{N}\sum_{n=1}^{N} \chi_{J}(\{x_n\})- \lambda(J) \right|<\frac{3}{m}+\frac{2}{m^{2}}\ .
\end{equation*}
Poich\`{e} la stima precedente non dipende da $J$, possiamo considerare l'estremo superiore di ambo i membri su ogni possibile intervallo $J\subset I$ ed ottenere cos\`{i}
\begin{equation*}
D_N(\omega)< \frac{3}{m}+\frac{2}{m^{2}}\ .
\end{equation*}
Dal momento che la quantit\`{a} $\frac{3}{m}+\frac{2}{m^{2}}$ pu\`{o} essere resa arbitrariamente piccola, il teorema \`{e} completamente dimostrato.
\end{proof}
\begin{oss}
\rm{Vista l'arbitrariet\`{a} di $J$, il teorema precedente mette in luce un aspetto interessante sulle successioni u.d. mod 1, ossia che
\begin{equation*}
\lim_{N\to \infty}\frac{1}{N}\sum_{n=1}^{N} \chi_{J}(\{x_n\})= \lambda(J)
\end{equation*}
uniformemente rispetto a tutti i sottointervalli $J=[a, b[$ di $I$.}
\end{oss}
Dimostriamo, ora, un ulteriore teorema che mostra come $D_N$ non possa convergere a zero troppo rapidamente.
\begin{thm}
\label{thm:1.4.4}
Per ogni successione di finita $x_1,\ldots,x_N$ di numeri reali risulta
\begin{equation*}
\frac{1}{N}\leq D_N(x_1,\ldots,x_N) \leq 1\ .
\end{equation*}
\end{thm}
\begin{proof}
La disuguaglianza di destra segue in maniera evidente dalla definizione di discrepanza. Per dimostrare la prima disuguaglianza, fissiamo $\varepsilon >0$ e scegliamo uno degli elementi della successione, che chiamiamo $x$ per comodit\`{a}. Consideriamo l'intervallo $J=[x,x+\varepsilon[\ \cap\ I$. Poich\`{e} $x\in J$, si ha
\begin{equation*}
\frac{1}{N}\sum_{n=1}^{N} \chi_{J}(\{x_n\})- \lambda(J)\geq \frac{1}{N}- \lambda(J)\geq \frac{1}{N}-\varepsilon \ .
\end{equation*}
Dall'arbitrariet\`{a} di $\varepsilon$ segue la disuguaglianza di sinistra e ci\`{o} conclude la dimostrazione.
\end{proof}
A volte \`{e} utile restringere la famiglia di intervalli su cui calcoliamo l'estremo superiore nella definizione di discrepanza. Il tipo di restrizione pi\`{u} interessante consiste nel considerare intervalli della forma $[0,c[$, con $0< c\leq 1$. La Definizione \ref{def:1.4.1} viene modificata di conseguenza ed estesa nella maniera che segue.
\begin{defn}
\label{def:1.4.5}
\rm{Data una successione finita $x_1,x_2,\ldots, x_N$ di numeri reali, definiamo \emph{star-discrepanza} la quantit\`{a}
\begin{equation*}
D_N^{\ast}=D_N^{\ast}(x_1,\ldots, x_N)=\sup_{0<c\leq 1}\left| \frac{1}{N}\sum_{n=1}^{N} \chi_{[0,c[}(\{x_n\})-c\ \right| \ .
\end{equation*}}
\end{defn}
Analogamente a quanto succede per la discrepanza, la star-discrepanza $D_N^{\ast}(\omega)$ di una successione $\omega$ infinita (o di una successione finita di almeno $N$ termini) si definisce come la star-discrepanza dei primi $N$ termini di $\omega$.\medskip\\
Il prossimo risultato ci permette di stabilire una connessione tra $D_N$ e $D_N^{\ast}$.
\begin{thm}
\label{thm:1.4.6}
Tra le discrepanze $D_N$ e $D_N^{\ast}$ esiste la relazione
\begin{equation}
D_N^{\ast}\leq D_N\leq 2 D_N^{\ast} \ . \label{eq:1.13}
\end{equation} 
\end{thm}
\begin{proof}
La prima disuguaglianza segue direttamente dalle Definizioni \ref{def:1.4.1} e \ref{def:1.4.5}. \\ Per dimostrare la seconda, osserviamo che
\begin{equation*}
\sum_{n=1}^{N} \chi_{[a ,b[}(\{x_n\})=\sum_{n=1}^{N} \chi_{[0,b[}(\{x_n\})-\sum_{n=1}^{N} \chi_{[0,a[}(\{x_n\}) \qquad (0\leq a < b \leq 1)\ ,
\end{equation*}
da cui segue che
\begin{equation*}
\left| \frac{1}{N}\sum_{n=1}^{N} \chi_{[a ,b[}(\{x_n\})-(b -a )\right| \leq
\end{equation*}
\begin{equation*}
\left| \frac{1}{N}\sum_{n=1}^{N} \chi_{[0,b[}(\{x_n\})-b \right| + \left| \frac{1}{N}\sum_{n=1}^{N} \chi_{[0,a[}(\{x_n\})-a \right| \ .
\end{equation*}
Passando all'estremo superiore otteniamo la (\ref{eq:1.13}).
\end{proof}
\begin{cor}
\label{cor:1.4.7}
La successione $w$ \`{e} u.d. mod 1 se e solo se $$\lim_{N\to \infty}D_N^{\ast}(\omega)=0\ .$$
\end{cor}
\begin{proof}
\`{E} una conseguenza immediata di (\ref{eq:1.10}) e (\ref{eq:1.13}).
\end{proof}
Concludiamo questa sezione con una definizione che rende bene l'idea del carattere quantitativo della discrepanza. Appare chiaro dalle definizioni precedenti, infatti, che una successione di numeri reali \`{e} tanto pi\`{u} uniformemente distribuita quanto pi\`{u} velocemente converge a zero la sua discrepanza.\\
\`{E} stato dimostrato che le successioni $\omega$ di punti aventi la discrepanza asintoticamente migliore sono quelle per cui $D_N(\omega)$ \`{e} dell'ordine di $\frac{\log N}{N}$ quando $N\to\infty$. Vale, infatti, la seguente
\begin{defn}
\rm{Una successione $\omega$ di numeri reali \`{e} detta \emph{successione di punti a bassa discrepanza} se
\begin{equation*}
D_N(\omega)=\mathcal{O}(\log N/N)\qquad (N\to \infty)\ .
\end{equation*}}
\end{defn}
\subsection*{Una successione speciale}
\label{Par:1.4}
Esibiamo ora una successione molto particolare costruita sull'intervallo unitario: la \emph{successione di van der Corput}. Ci\`{o} che la rende un interessante oggetto di studio \`{e} che essa \`{e} una successione infinita con bassa discrepanza. \\ La costruzione di tale successione si basa sulla rappresentazione in base $2$ dei numeri naturali e sulla funzione radice inversa, di cui ci accingiamo a fornire la definizione.\\ Nel Capitolo $2$ vedremo in dettaglio la rappresentazione di un numero naturale in una qualsiasi base $b$ e la corrispondente funzione radice inversa.
\begin{defn}
\label{radinv}
\rm{Sia $n \in \mathbb{N}$. Se $n$ si scrive come $\sum_{i=0}^{M}a_i2^{i}$, dove $a_i\in \{0, 1\}$ e $M=\lfloor \log_2 n \rfloor$, la rappresentazione in base $2$ di $n$ \`{e} la sequenza di cifre $[n]_2=a_M a_{M-1}\ldots a_0$. La funzione $\phi_2:\mathbb{N}\longrightarrow [0,1]$, definita come
\begin{equation*}
\phi_2(n)=\sum_{i=0}^{M}\dfrac{a_i}{2^{i+1}}
\end{equation*}
\`{e} detta \emph{funzione radice inversa}, e la sua rappresentazione in base $2$ \`{e} \linebreak$[\phi_2(n)]_2=0.a_0\ldots a_{M-1}a_M$.}
\end{defn}
Descriviamo brevemente questa funzione attraverso un facile esempio.
\begin{eso}
\label{eso}
\rm{Consideriamo la rappresentazione binaria del numero naturale $7$. Si verifica facilmente che
\begin{equation*}
7=\sum_{i=0}^{2}2^{i}\ ,
\end{equation*} 
dove gli $a_i$ sono tutti uguali a 1 e quindi $[7]_2=111$. Allora il valore della funzione radice inversa in $7$ \`{e} 
\begin{equation*}
\phi_2(7)=\sum_{i= 0}^{2}\dfrac{1}{2^{i+1}}\ ,
\end{equation*}
la cui espressione binaria \`{e} $[\phi_2(7)]_2=0.111$.
}
\end{eso}
Ora, poich\`{e} il valore della funzione radice inversa pu\`{o} essere determinato per ogni numero naturale, risulta immediato considerare successioni numeriche il cui termine generale sia proprio $\phi_2(n)$.\\ La successione $\{\phi_2(n)\}$ si ottiene nel modo seguente:\\ $\mathit{1.}$ Si considera la rappresentazione binaria dei numeri naturali
\begin{equation*}
01, 10, 11, 100, 101, 110, 111, \ldots 
\end{equation*}
$\mathit{2.}$ Si considera la funzione radice inversa di ciascuna rappresentazione
\begin{equation*}
\frac{1}{2}, \frac{1}{4}, \frac{3}{4}, \frac{1}{8}, \frac{5}{8}, \frac{3}{8},\frac{7}{8},\ldots
\end{equation*}
o, equivalentemente, la loro espressione binaria
\begin{equation*}
0.1, 0.01, 0.11, 0.001, 0.101, 0.011, 0.111, \ldots \ .
\end{equation*}
Possiamo dunque presentare la seguente definizione. 
\begin{defn}
\label{van der Corput}
\rm{La \emph{successione di van der Corput} \`{e} definita come la successione di termine generale $\phi_2(n)$, dove $\phi_2$ \`{e} la funzione radice inversa.}
\end{defn}
Enunciamo ora il risultato pi\`{u} importante presentato dallo stesso van der Corput su tale successione.
\begin{thm}
La discrepanza $D_N(\omega)$ della successione di van der Corput $\omega=\{\phi_2(n)\}_{n\in\mathbb{N}}$ soddisfa
\begin{equation}
ND_N(\omega)\leq\dfrac{log(N+1)}{\log 2}
\end{equation}
per ogni $N\in\mathbb{N}$.
\end{thm}
La conseguenza immediata di tale teorema \`{e} il seguente
\begin{cor}
La successione di van der Corput $\{\phi_2(n)\}_{n\in\mathbb{N}}$ \`{e} u.d..
\end{cor}
\begin{proof}
Segue immediatamente dal Teorema \ref{thm:1.4.2} perch\`{e} $\frac{\log(N+1)}{N\log 2}$ \`{e} infinitesimo per $N\to\infty$.
\end{proof}
Ci sono varie generalizzazioni della successione di van der Corput. La prima estensione della procedura al caso bidimensionale avvenne per opera dello stesso van der Corput. Successivamente Hammersley \cite{Hammersley} e Halton \cite{Halton} hanno proposto altre generalizzazioni al caso di un ipercubo in qualsiasi dimensione, ma di questo ci occuperemo in maniera pi\`{u} estensiva nel terzo capitolo. 
\subsection*{Discrepanza multidimensionale}
Le definizioni di $D_N$ e $D_N^{\ast}$ possono essere estese in maniera piuttosto ovvia a successioni nello spazio $s$-dimensionale $\mathbb{R}^{s}$.
\begin{defn}
\rm{Sia $\textbf{x}_1,\ldots, \textbf{x}_N$ una successione finita in $\mathbb{R}^{s}$. La \emph{discrepanza} $D_N$ e la \emph{star-discrepanza} $D_N^{\ast}$ sono definite come segue:
\begin{equation*}
D_N=D_N(\textbf{x}_1,\ldots, \textbf{x}_N)=\sup_J \left|\frac{1}{N}\sum_{n=1}^{N} \chi_{J}(\{\textbf{x}_n\})-\lambda(J) \right|
\end{equation*}
e
\begin{equation*}
D_N^{\ast}=D_N^{\ast}(\textbf{x}_1,\ldots, \textbf{x}_N)=\sup_{J^{\ast}} \left|\frac{1}{N}\sum_{n=1}^{N} \chi_{J^{\ast}}(\{\textbf{x}_n\})-\lambda(J^{\ast}) \right|\ ,
\end{equation*}
dove $J$ varia tra tutti i possibili sottointervalli di $I^{s}$ della forma \linebreak $\{\textbf{x}\in \mathbb{R}^{s}:\textbf{a} \leq \textbf{x} < \textbf{b}\}$ e $J^{\ast}$ varia tra tutti i sottointervalli della forma $\{\textbf{x}\in \mathbb{R}^{s}:\textbf{0} \leq \textbf{x} < \textbf{b}\}$, e $\lambda$ denota la misura di Lebesgue $s$-dimensionale.\medskip\\
Infine, se la successione ($\textbf{x}_n)$ \`{e} infinita, la discrepanza $D_N$ \`{e} quella del segmento iniziale formato dai primi $N$ termini della successione, come abbiamo gi\`{a} visto nel caso unidimensionale.}
\end{defn}
Vediamo come la relazione tra $D_N$ e $D_N^{\ast}$ in pi\`{u} dimensioni sia molto simile alla (\ref{eq:1.13}) che vale per $s=1$. \\ L'idea \`{e} quella di partire da uno dei sottointervalli $J$ di $I^{s}$ e rappresentarlo in termini di intervalli del tipo $J^{\ast}$. Descriviamo il procedimento per $s=2$ perch\`{e} la generalizzazione a dimensioni superiori si ottiene in maniera analoga. \\ Sia $J$ il seguente sottointervallo di $I^{2}$:
\begin{equation*}
J=\{(x_1,x_2)\in I^{2}: \alpha_1\leq x_1< \beta_1 \wedge \alpha_2\leq x_2< \beta_2\}=[\alpha_1,\beta_1[\times [\alpha_2,\beta_2[\ ,
\end{equation*}
con $0\leq \alpha_i < \beta_i\leq 1$ per $i=1, 2$.  \\ Allora
\begin{eqnarray*}
J &=& \{\big([0,\beta_1[\times [0,\beta_2[\big)\setminus \big([0,\alpha_1[\times [0,\beta_2[\big)\}\\
&\setminus & \{\big([0,\beta_1[\times [0,\alpha_2[\big)\setminus \big([0,\alpha_1[\times[0,\alpha_2[\big)\}= (J_1^{\ast}\setminus J_2^{\ast})\setminus (J_3^{\ast}\setminus J_4^{\ast})\ .
\end{eqnarray*}
Di conseguenza, per le propriet\`{a} di cui gode la misura di Lebesgue, risulta
\begin{equation*}
\lambda(J)=\lambda(J_1^{\ast})-\lambda(J_2^{\ast})-\lambda(J_3^{\ast})+\lambda(J_4^{\ast})
\end{equation*}
e
\begin{equation*}
\sum_{n=1}^{N} \chi_{J}(\{x_n\})=\sum_{n=1}^{N} \chi_{J_1^{\ast}}(\{x_n\})-\sum_{n=1}^{N} \chi_{J_2^{\ast}}(\{x_n\})- \sum_{n=1}^{N} \chi_{J_3^{\ast}}(\{x_n\})+\sum_{n=1}^{N} \chi_{J_4^{\ast}}(\{x_n\}).
\end{equation*}
Con lo stesso ragionamento utilizzato nella dimostrazione del Teorema \ref{thm:1.4.6} si conclude che
\begin{equation}
D_N^{\ast}\leq D_N\leq 4 D_N^{\ast} \ .\label{eq:1.14}
\end{equation}
\paragraph*{}
Si pu\`{o} dimostrare che nel caso $s$-dimensionale la relazione tra $D_N$ e $D_N^{\ast}$ diventa
\begin{equation}
D_N^{\ast}\leq D_N\leq 2^{s} D_N^{\ast} \ .\label{eq:1.15}
\end{equation}
Analogamente a quanto fatto nel caso unidimensionale (Teorema \ref{thm:1.4.2} e Corollario \ref{cor:1.4.7}) si pu\`{o} dimostrare che una successione $w$ \`{e} u.d. mod 1 se e solo se $\lim_{N\to \infty}D_N^{*}(w)=\lim_{N\to \infty}D_N(w)=0$. \smallskip\\ Infine, generalizzando il Teorema \ref{thm:1.4.4}, ritroviamo anche nel caso multidimensionale la stima $D_N\geq \frac{1}{N}$.

\chapter{$LS$-successioni di partizioni e di punti in $[0,1[$}
In questo capitolo prenderemo in esame dapprima una famiglia di successioni di partizioni nell'intervallo unitario, introdotte per la prima volta da S. Kakutani in \cite{Kakutani}, ed il loro rapporto con la teoria dell'uniforme distribuzione. Successivamente, focalizzeremo la nostra attenzione sulle $LS$-successioni, introdotte per la prima volta da I. Carbone in \cite{Carbone}. Seguendo la linea di questo lavoro, prenderemo in esame le successioni di punti associate alle $LS$-successioni di partizioni, vedremo quale relazione intercorre con la successione di van der Corput, di cui abbiamo gi\`{a} studiato la costruzione nel Paragrafo \ref{Par:1.4} e di cui sappiamo che la discrepanza \`{\`{e}} ottimale ed, infine, introdurremo un algoritmo ``veloce'' \cite{Carbone2} completando il capitolo con la presentazione di un programma da noi implementato. L'utilizzo dell'aggettivo veloce non \`{e} casuale, ma nasce dall'esigenza di confronto tra quest'ultimo e l'algoritmo presentato in \cite{Carbone}, rispetto al quale ha una complessit\`{a} computazionale meno elevata. 
\section{Successioni di partizioni e raffinamenti}
\begin{defn}
\rm{Una \emph{partizione} $\pi$ dell'intervallo $I=[0,1[$ \`{e} una famiglia di suoi sottointervalli a due a due disgiunti, determinati da una successione crescente di punti $\{y_0,y_1,\dots, y_k\}$, ossia
\begin{equation*}
\pi=\{[y_i,y_{i+1}[\, :\, 0\leq i \leq k-1\}
\end{equation*}
con $I= \cup_{i=0}^{k-1}[y_i, y_{i+1}[$\ .}\\
\end{defn}
Il concetto di uniforme distribuzione si estende in maniera naturale alle successioni di partizioni dell'intervallo unitario nel modo seguente.
\begin{defn}
\rm{Una successione di partizioni $\{\pi_n\}_{n\in\mathbb{N}}$ di $[0,1[$, dove $\pi_n=\{[y_i^{(n)}, y_{i+1}^{(n)}[\, :\, 1\leq i\leq k(n)\}$, \`{e} detta \emph{uniformemente distribuita} ($u.d.$) se
\begin{equation}
\lim_{n \to \infty}\frac{1}{k(n)} \sum_{i=1} ^{k(n)} f(y_i^{(n)})= \int_0 ^{1} f(t)dt
\end{equation}
per ogni funzione continua $f$ definita su $[0,1]$.}
\end{defn}
Come abbiamo gi\`{a} visto nel Capitolo $1$, la definizione precedente pu\`{o} anche essere espressa nella seguente forma equivalente.
\begin{defn}
\rm{Diciamo che la successione di partizioni $\{\pi_n\}_{n\in\mathbb{N}}$ con $\pi_n=\{[y_i^{(n)}, y_{i+1}^{(n)}[\, :\, 1\leq i\leq k(n)\}$ \`{e} $u.d.$ se per ogni coppia di numeri reali $a,b$, con $0\leq a < b \leq 1$, si ha
\begin{equation}
\lim_{n\to\infty}\frac{1}{k(n)} \sum_{i=1} ^{k(n)} \chi_{[a, b[}(y_i^{(n)})=b-a 
\end{equation}
o, equivalentemente,
\begin{equation}
\lim_{n\to\infty}\frac{1}{k(n)} \sum_{i=1} ^{k(n)} \chi_{[0, c[}(y_i^{(n)})=c \qquad \text{per ogni $0 < c \leq 1$}\ .
\end{equation}}
\end{defn}
Ora descriviamo brevemente la procedura introdotta da Kakutani nel $1976$ per generare una classe particolare di successioni di partizioni u.d. dell'intervallo $[0,1]$.
\begin{defn}[$\alpha$-raffinamenti di Kakutani]
\rm{Sia $\alpha \in\ ]0,1[$ e sia $\pi =\{[y_{i},y_{i+1}[\, :\, 1 \leq i \leq k-1\}$ una partizione di $[0,1[$. Chiamiamo \emph{$\alpha$-raffinamento} di $\pi$, e lo indicheremo con $\alpha\pi$, la suddivisione in due parti dei soli intervalli di lunghezza massima di $\pi$, proporzionalmente ad $\alpha$ e $1-\alpha$, rispettivamente. \\ Se denotiamo con $\alpha^{n}\pi$ l'$\alpha$-raffinamento di $\alpha^{n-1}\pi$ per ogni $n\in\mathbb{N}$, la successione $\{\kappa_n\}_{n\in\mathbb{N}}$ ottenuta per successivi $\alpha$-raffinamenti della partizione banale $\omega=\{[0,1[\}$ \`{e} detta \emph{successione di partizioni di Kakutani}.}
\end{defn}
Osserviamo che per $\alpha=\frac{1}{2}$, la successione di partizioni di Kakutani coincide con la partizione binaria.\\ Fu lo stesso Kakutani a dimostrare che
\begin{thm}
La successione di partizioni di Kakutani $\{\kappa_n\}_{n\in\mathbb{N}}$ \`{e} u.d. per ogni $\alpha\in\ ]0,1[$.
\end{thm}
Al fine di illustrare le $LS$-successioni di partizioni, e le successioni di punti ad esse associate, \`{e} opportuno presentare il concetto di $\rho$-raffinamento introdotto da A. Vol\v{c}i\v{c} in \cite{Volcic}, che generalizza quello di $\alpha$-raffinamento di Kakutani.
\begin{defn}[$\rho$-raffinamento]
\rm{Sia $\rho$ una partizione non banale di $[0,1[$. Definiamo $\rho$\emph{-raffinamento} di una partizione  $\pi$ di $[0,1[$, e lo indicheremo con $\rho\pi$, la suddivisione di tutti gli intervalli di $\pi$ che hanno lunghezza massima in maniera omotetica rispetto a $\rho$.}
\end{defn}
\begin{oss}
\rm{Se consideriamo la partizione $\rho=\{[0,\alpha[,[\alpha,1[\}$ e prendiamo $\pi=\omega$, allora il suo $\rho$-raffinamento coincide con l'$\alpha$-raffinamento di Kakutani. \\ Inoltre, come nel caso della procedura di Kakutani, possiamo considerare il $\rho$-raffinamento di $\rho\pi$ ed indicarlo con $\rho^{2}\pi$ e cos\`{i} via, indicando con $\rho^{n}\pi$ il $\rho$-raffinamento di $\rho^{n-1}\pi$.}
\end{oss}
\begin{defn}
\rm{Data una partizione non banale $\rho$ di $[0,1[$, la \emph{successione dei $\rho$-raffinamenti} $\{\rho^{n}\pi\}_{n\in\mathbb{N}}$ della partizione $\pi$ si definisce come la successione di partizioni ottenuta mediante i successivi $\rho$-raffinamenti di $\pi$.}
\end{defn}
Se con $\{\rho^{n}\omega\}_{n\in\mathbb{N}}$ indichiamo la successione di successivi $\rho$-raffinamenti della partizione banale $\omega$, vale il seguente importante risultato (\cite{Volcic}, Theorem 2.7).
\begin{thm}
La successione $\{\rho^{n}\omega\}_{n\in\mathbb{N}}$ \`{e} u.d..
\end{thm}
Vediamo infine come \`{e} definita la discrepanza per una successione di partizioni e quando questa pu\`{o} dirsi ottimale.
\begin{defn}
\rm{La \emph{discrepanza} di una successione di partizioni $\{\pi_n\}_{n\in\mathbb{N}}$, con $\pi_n=\{[y_i^{(n)}, y_{i+1}^{(n)}[\, :\, 1\leq i\leq k(n)\}$, \`{e} definita come
\begin{equation*}
D(\pi_n)=\sup_{0\leq a<b\leq 1}\left| \frac{1}{k(n)}\sum_{j=1}^{k(n)}\chi_{[a,b[}(y_j^{(n)})-(b-a) \right|\ .
\end{equation*}
}
\end{defn}
\`{E} stato dimostrato che le successioni di partizioni $\{\pi_n\}_{n\in\mathbb{N}}$ aventi la discrepanza asintoticamente migliore sono quelle per cui $D(\pi_n)$ \`{\`{e}} dell'ordine di $\frac{1}{k(n)}$ quando $n\to\infty$. \\ \`{E} stata introdotta, perci\`{o}, la seguente
\begin{defn}
\rm{Una successione di partizioni $\{\pi_n\}_{n\in\mathbb{N}}$, dove $\pi_n=\{[y_i^{(n)}, y_{i+1}^{(n)}[\, :\, 1\leq i\leq k(n)\}$, \`{e} detta \emph{a bassa discrepanza} se
\begin{equation*}
D(\pi_n)=\mathcal{O}(1/k(n))\qquad (n\to \infty)\ .
\end{equation*}}
\end{defn}
\section{$LS$-successioni di partizioni}
In questo paragrafo prenderemo in esame una classe particolare di $\rho$-raffina-menti della partizione banale $\omega$ dell'intervallo $[0,1[$ : le $LS$-successioni di partizioni, che costituiscono una classe numerabile di successioni, al variare di $L$ ed $S$ in $\mathbb{N}$. 
\begin{defn}
\rm{Fissiamo due interi positivi $L$ ed $S$ e sia $\beta$ il numero reale positivo che verifica l'equazione $L\beta + S\beta^{2}=1$. Indichiamo con $\rho_{L,S}$ la partizione definita da $L$ intervalli lunghi aventi lunghezza $\beta$ e da $S$ intervalli corti aventi lunghezza $\beta^{2}$. Chiameremo \emph{$LS$-successione di partizioni} la successione $\{\rho_{LS}^{n}\omega\}_{n\in\mathbb{N}}$ (brevemente, $\{\rho_{L,S}^{n}\}_{n\in\mathbb{N}}$) dei successivi $\rho_{L,S}$-raffinamenti della partizione banale $\omega$ dove, ovviamente, con $\rho_{L,S}^{n}$ denotiamo il $\rho_{L,S}$-raffinamento di $\rho_{L,S}^{n-1}$.}
\end{defn}
Illustriamo meglio questa costruzione attraverso alcuni esempi.
\begin{es}
\label{es}
\rm{\medskip \textbf{1.} Consideriamo $L=S=1$ e la corrispondente $1,1$- successione di partizioni $\{\rho_{1,1}^{n}\}_{n\in\mathbb{N}}$ determinata dalla relazione $\beta+\beta^{2}=1$, la cui soluzione in $[0,1[$ \`{e} il numero aureo $\beta=\frac{\sqrt{5}-1}{2}$. La prima partizione conster\`{a} semplicemente di due intervalli di lunghezza $\beta$ e $\beta^{2}$, rispettivamente, ossia
\begin{equation*}
\rho_{1,1}^{1}=\{[0,\beta[,[\beta,1[\}\ ,
\end{equation*}
come mostra il seguente grafico:\\

\begin{figure}[h!]
\begin{center}
\begin{tikzpicture}[scale=13]
\draw (0, 0) -- (1,0);
\draw (0,-0.01) node[below, black]{0} -- (0,0.01);
\draw (1,-0.01) node[below, black]{1} -- (1,0.01);
\draw (0.61803,-0.01) node[below, black]{$\beta$}-- (0.61803,0.01);
\end{tikzpicture}
\end{center}
\end{figure}

Per costruire la seconda partizione, si suddivide l'intervallo di lunghezza maggiore (il primo) proporzionalmente a $\beta$ e $\beta^{2}$ in modo da ottenere da esso un intervallo lungo ed uno corto, e cio\`{e}
\begin{equation*}
\rho_{1,1}^{2}=\{[0,\beta^{2}[,[\beta^{2},\beta[,[\beta ,1[\}\ .
\end{equation*}

\begin{figure}[h!]
\begin{center}
\begin{tikzpicture}[scale=13]
\draw (0, 0) -- (1,0);
\draw (0,-0.01) node[below, black]{$0$} -- (0,0.01);
\draw (1,-0.01) node[below, black]{$1$} -- (1,0.01);
\draw (0.61803,-0.01) node[below, black]{$\beta$}-- (0.61803,0.01);
\draw (0.38196,-0.01) node[below, black]{$\beta^{2}$}-- (0.38196,0.01);
\end{tikzpicture}
\end{center}
\end{figure}

A questo punto, come si pu\`{o} facilmente osservare, ci sono due intervalli di lunghezza massima, il primo ed il terzo, per cui si procede alla loro sola suddivisione. Otteniamo cos\`{i} la terza partizione 
\begin{equation*}
\rho_{1,1}^{3}=\{[0,\beta^{3}[, [\beta^{3},\beta^{2}[,[\beta^{2},\beta[,[\beta , \beta +\beta^{3}[,[\beta+\beta^{3},1[\}\ ,
\end{equation*}
la cui rappresentazione grafica \`{e}\\

\begin{figure}[h!]
\begin{center}
\begin{tikzpicture}[scale=13]
\draw (0, 0) -- (1,0);
\draw (0,-0.01) node[below, black]{0} -- (0,0.01);
\draw (1,-0.01) node[below, black]{1} -- (1,0.01);
\draw (0.61803,-0.01) node[below, black]{$\beta$}-- (0.61803,0.01);
\draw (0.38196,-0.01) node[below, black]{$\beta^{2}$}-- (0.38196,0.01);
\draw (0.23606,-0.01) node[below, black]{$\beta^{3}$}-- (0.23606,0.01);
\draw (0.85410,-0.01) node[below, black]{$\beta +\beta^{3}$}-- (0.85410,0.01);
\end{tikzpicture}
\end{center}
\end{figure}

e cos\`{i} via in modo da ottenere le altre. \\ Sottolineiamo che questo \`{e} un esempio di $\beta$-raffinamento alla Kakutani, cio\`{e} $\{\rho_{1,1}^{n}\}_{n\in\mathbb{N}}$ \`{e} la successione di partizioni di Kakutani ottenuta per successivi $\beta$-raffinamenti di $\omega$.\\
Passiamo ora ad un esempio un po' pi\`{u} complicato.
\paragraph*{}
\textbf{2.} Consideriamo la successione di partizioni $\{\rho_{2,1}^{n}\}_{n\in\mathbb{N}}$ corrispondente a \linebreak $2\beta+\beta^{2}=1$. Al primo passo avremo una suddivisione dell'intervallo unitario in due intervalli lunghi ed uno corto, in modo da ottenere la seguente partizione
\begin{equation*}
\rho_{2,1}^{1}=\{[0,\beta[,[\beta, 2\beta[,[2\beta,1[\}\ ,
\end{equation*}

\begin{figure}[h!]
\begin{center}
\begin{tikzpicture}[scale=13]
\draw (0, 0) -- (1,0);
\draw (0,-0.01) node[below, black]{\small $0$} -- (0,0.01);
\draw (1,-0.01) node[below, black] {\small $1$} -- (1,0.01);
\draw (0.41421,-0.01) node[below, black]{\small$\beta$}-- (0.41421,0.01);
\draw (0.82842,-0.01) node[below, black]{\small$2\beta$}-- (0.82842,0.01);
\end{tikzpicture}
\end{center}
\end{figure}

Per ottenere la seconda partizione si dovranno pertanto dividere i primi due intervalli, in modo da ottenere da ciascuno di essi due intervalli lunghi e uno corto. Si avr\`{a} dunque
\begin{eqnarray*}
\rho_{2,1}^{2}&=& \{[0,\beta^{2}[,[\beta^{2},2\beta^{2}[,[2\beta^{2},\beta[,[\beta,\beta+\beta^{2}[, \\
& & [\beta+\beta^{2}, \beta+2\beta^{2}[,[\beta+2\beta^{2},2\beta[,[2\beta,1[\}\ ,
\end{eqnarray*}
che \`{e} graficamente rappresentato qui di seguito.\pagebreak\\
\begin{figure}[h!]
\begin{center}
\begin{tikzpicture}[scale=13]
\draw (0, 0) -- (1,0);
\draw (0,-0.01) node[below, black]{\small $0$} -- (0,0.01);
\draw (1,-0.01) node[below, black] {\small $1$} -- (1,0.01);
\draw (0.41421,-0.01) node[below, black]{\small$\beta$}-- (0.41421,0.01);
\draw (0.82842,-0.01) node[below, black]{\small$2\beta$}-- (0.82842,0.01);
\draw (0.17157,-0.01) node[below, black]{\small $\beta^{2}$}-- (0.17157,0.01);
\draw (0.58578,-0.01) node[below, black]{\small $\beta +\beta^{2}$}-- (0.58578,0.01);
\draw (0.34314,-0.01) node[below, black]{\small $2\beta^{2}$}-- (0.34314,0.01);
\draw (0.75735,-0.01) node[below, black]{\small $\beta +2\beta^{2}$}-- (0.75735,0.01);
\end{tikzpicture}
\end{center}
\end{figure}

Al passo successivo gli intervalli da dividere saranno cinque, ciascuno di lunghezza $\beta^{2}$.\\ Iterando questo processo otteniamo le partizioni successive.
\paragraph*{}
\textbf{3.} Infine, presentiamo la successione di partizioni $\{\rho_{1,2}^{n}\}_{n\in\mathbb{N}}$ relativa a \linebreak $\beta+2\beta^{2}=1$. Sebbene (come vedremo in seguito) a differenza delle due precedenti non si tratti di una successione a bassa discrepanza, tale esempio sar\`{a} utile per la comprensione dell'algoritmo nella sua formulazione generale.\\ La prima partizione si ottiene dividendo l'intervallo unitario in un intervallo lungo di lunghezza $\beta=\frac{1}{2}$ e in due intervalli corti di lunghezza $\beta^{2}=\frac{1}{4}$, nel modo seguente:\\

\begin{figure}[h!]
\begin{center}
\begin{tikzpicture}[scale=13]
\draw (0, 0) -- (1,0);
\draw (0,-0.01) node[below, black]{0} -- (0,0.01);
\draw (1,-0.01) node[below, black]{1} -- (1,0.01);
\draw (0.5,-0.01) node[below, black]{$\beta$}-- (0.5,0.01);
\draw (0.75,-0.01) node[below, black]{$\beta+\beta^{2}$}-- (0.75,0.01);
\end{tikzpicture}
\end{center}
\end{figure}

cio\`{i}
\begin{equation*}
\rho_{1,2}^{1}=\{[0,\beta[,[\beta, \beta+\beta^{2}[,[\beta+\beta^{2},1[\}\ . \medskip
\end{equation*}
Al secondo passo dobbiamo suddividere l'unico intervallo di lunghezza massima $\beta$ proporzionalmente a $\beta$ e $\beta^{2}$, ottenendo cos\`{i}
\begin{equation*}
\rho_{1,2}^{2}=\{[0,\beta^{2}[,[\beta^{2}, \beta^{2}+\beta^{3}[,[\beta^{2}+\beta^{3},\beta[, [\beta, \beta+\beta^{2}[,[\beta+\beta^{2},1[\}\ ,
\end{equation*}
che \`{e} rappresentato cos\`{i}:\\

\begin{figure}[h!]
\begin{center}
\begin{tikzpicture}[scale=13]
\draw (0, 0) -- (1,0);
\draw (0,-0.01) node[below, black]{0} -- (0,0.01);
\draw (1,-0.01) node[below, black]{1} -- (1,0.01);
\draw (0.5,-0.01) node[below, black]{$\beta$}-- (0.5,0.01);
\draw (0.75,-0.01) node[below, black]{$\beta+\beta^{2}$}-- (0.75,0.01);
\draw (0.25,-0.01) node[below, black]{$\beta^{2}$}-- (0.25,0.01);
\draw (0.375,-0.01) node[below, black]{$\beta^2 +\beta^{3}$}-- (0.375,0.01);
\end{tikzpicture}
\end{center}
\end{figure}

 Procedendo in maniera analoga a quanto fatto finora si ottengono tutte le altre partizioni.}
\end{es}
Prendiamo in considerazione alcune propriet\`{a} di tali successioni, che mettano in luce il loro carattere ricorsivo e vediamo come tra esse meriti una particolare attenzione la successione ottenuta considerando la coppia di interi $L=S=1$.
\begin{ossi}
\rm{ \textbf{(i)} Notiamo innanzitutto che per ogni partizione $\rho_{L,S}^{n}$ contiene solo due tipi di intervalli: quelli lunghi e quelli corti, i primi di misura $\beta^{n}$, i secondi $\beta^{n+1}$. Se con $t_n$ indichiamo il numero totale di intervalli di $\rho_{L,S}^{n}$, allora vale la seguente equazione alle differenze finite, lineare, omogenea, del secondo ordine
\begin{equation}
t_n=Lt_{n-1}+St_{n-2} \ , \label{eq:2.4} \medskip
\end{equation}
con condizioni iniziali $t_0=1$ e $t_1=L+S$. \\ Se indichiamo con $l_n$ il numero di intervalli lunghi di $\rho_{L,S}^{n}$ e con $s_n$ il numero dei suoi intervalli corti, valgono le seguenti relazioni:
\begin{equation*}
t_n=l_n+s_n\ , \qquad l_n=Ll_{n-1}+s_{n-1}\ , \qquad s_n=Sl_{n-1} \ .
\end{equation*}
Infatti l'$n$-esima partizione si ottiene dalla  somma dei suoi intervalli lunghi e dei suoi intervalli corti. In particolare, gli intervalli lunghi della partizione $n$-esima sono dati dalla suddivisione degli $l_{n-1}$ intervalli lunghi della partizione $\rho_{L,S}^{n-1}$, proporzionalmente in $L$ intervalli lunghi, sommati agli intervalli corti della partizione precedente $\rho_{L,S}^{n-1}$ che diventano lunghi al passo successivo, mentre gli intervalli corti sono quegli intervalli ottenuti suddividendo gli $l_{n-1}$ intervalli lunghi della partizione $\rho_{L,S}^{n-1}$, proporzionalmente in $S$ intervalli corti. Allora, utilizzando queste relazioni si ha che
\begin{eqnarray*}
t_n&=&l_n+s_n=Ll_{n-1}+s_{n-1}+Sl_{n-1}= \\
&=& Ll_{n-1}+Sl_{n-2}+SLl_{n-2}+Ss_{n-2}=\\
&=&L(l_{n-1}+s_{n-1})+S(l_{n-2}+s_{n-2})=\\
&=&Lt_{n-1}+St_{n-2}.
\end{eqnarray*}}
\rm{\textbf{(ii)} Dalla $(2.4)$ segue immediatamente che la successione $\{t_n\}_{n\in\mathbb{N}}$, nel caso $L=S=1$, \`{e} la ben nota successione di Fibonacci. D'altra parte abbiamo gi\`{a} evidenziato nell'Esempio \ref{es} che la successione $\{\rho_{1,1}^{n}\}_{n\in\mathbb{N}}$ \`{e} una successione di Kakutani. Per questi due motivi insieme, la $1,1$-successione di partizioni corrispondente \`{e} stata chiamata in \cite{Carbone} \emph{successione di partizioni di Kakutani-Fibonacci}.}
\end{ossi}
In (\cite{Carbone}, Theorem 2.2) \`{e} stato dimostrato, altres\`{i}, che per $L\geq S$ tutte le $LS$-successioni di partizioni hanno discrepanza ottimale, e cio\`{e} $D(\rho_{L,S}^{n})=\mathcal{O}\left( \frac{1}{t_n} \right)$ per $n\to \infty$, ed \`{e} stata data una stima per la discrepanza in tutti i casi. Osservato che $L\geq S$ equivale a $S<L+1$, e ricordato che $L\beta +S\beta^{2}=1$, si ha il seguente
\begin{thm}
(i) Se $S<L+1$ esistono due costanti positive $c_1$ e $c_2$ tali che per ogni $n\in\mathbb{N}$
\begin{equation*}
\frac{c_1}{t_n}\leq D(\rho_{L,S}^{n})\leq \frac{c_2}{t_n}\ .
\end{equation*}
(ii) Se $S=L+1$ esistono due costanti positive $c_3$ e $c_4$ tali che per ogni $n\in\mathbb{N}$
\begin{equation*}
c_3\frac{\log t_n}{t_n}\leq D(\rho_{L,S}^{n})\leq c_4\frac{\log t_n}{t_n}\ .
\end{equation*}
(ii) Se $S>L+1$ esistono due costanti positive $c_5$ e $c_6$ tali che per ogni $n\in\mathbb{N}$
\begin{equation*}
\frac{c_5}{(t_n)^{\gamma}}\leq D(\rho_{L,S}^{n})\leq \frac{c_6}{(t_n)^{\gamma}}\ ,
\end{equation*}
dove $\gamma=1+\frac{\log S\beta}{\log\beta}< 1$.
\end{thm}
Sottolineiamo, infine, che grazie a questo teorema, per la prima volta in letteratura \`{e} stata presentata una stima per la discrepanza di una successione di Kakutani (che risulta essere, inoltre, a bassa discrepanza).
\section{$LS$-successioni di punti}
Passiamo ora allo studio delle successioni di punti (anche'esse introdotte in \cite{Carbone}) associate alle $LS$-successioni di partizioni. Lo scopo \`{e} quello di presentare il riordinamento ``alla van der Corput'' della successione di punti associata a ciascuna $LS$-successione di partizioni, in modo che la corrispondente successione di punti abbia discrepanza migliore possibile. Con la locuzione ``alla van der Corput'' intendiamo il processo di riordinamento che ricorda la costruzione della successione di van der Corput, studiata nel Paragrafo \ref{Par:1.4}.
\begin{defn}
\label{def:2.3.1}
\rm{La \emph{$LS$-successione di punti} $\{\xi_{L,S}^{n}\}_{n\in\mathbb{N}}$ associata alla successione di partizioni $\{\rho_{L,S}^{n}\}_{n\in\mathbb{N}}$ \`{e} definita come segue. I primi $t_1$ punti sono esattamente gli estremi sinistri degli intervalli di $\rho_{L,S}^{1}$, presi nell'ordine lessicografico. Tale insieme ordinato sar\`{a} denotato con $\Lambda_{L,S}^{1}$ ed i suoi punti saranno indicati con $\xi_1^{(1)},\ldots, \xi_{t_1}^{(1)}$. In generale, per $n>1$, se con $\Lambda_{L,S}^{n}=\Big(\xi_1^{(n)},\ldots, \xi_{t_n}^{(n)}\Big)$ indichiamo i $t_n$ punti della partizione $\rho_{L,S}^{n}$, scritti nell'ordine giusto, i punti della partizione successiva $\rho_{L,S}^{n+1}$ si ottengono applicando agli estremi sinistri degli intervalli lunghi della partizione precedente le seguenti famiglie di funzioni
 \begin{equation}
\varphi^{(n+1)}_i(x)=x+i\beta^{n+1} \quad e \quad \varphi^{(n+1)}_{L,j}(x)=x+L\beta^{n+1}+j{\beta}^{n+2} \ , \label{funzioni_phi}
\end{equation}
con $1\leq i \leq L$ e $1\leq j \leq S-1$. }
\end{defn}
Precisamente, i $t_{n+1}$ punti della partizione $\rho_{L,S}^{n+1}$ sono ordinati nel modo seguente:
\begin{equation*}
\begin{split}
& \Lambda_{L,S}^{n+1} = \Big(\xi_1^{(n)},\ldots, \xi_{t_n}^{(n)},\\
& \varphi_1^{(n+1)}(\xi_1^{(n)}),\ldots , \varphi_1^{(n+1)}(\xi_{l_n}^{(n)}),\ldots , \varphi_L^{(n+1)}(\xi_1^{(n)}), \ldots , \varphi_L^{(n+1)}(\xi_{l_n}^{(n)}) \\
& \varphi_{L,1}^{(n+1)}(\xi_1^{(n)}), \ldots, \varphi_{L,1}^{(n+1)}(\xi_{l_n}^{(n)}),\ldots , \varphi_{L,S-1}^{(n+1)}(\xi_1^{(n)}), \ldots , \varphi_{L,S-1}^{(n+1)}(\xi_{l_n}^{(n)})\Big)\ .
\end{split}
\end{equation*}
In (\cite{Carbone}, Theorem 3.9) \`{e} stato dimostrato che $D(\xi_{L,S}^{1}, \ldots, \xi_{L,S}^{N})=\mathcal{O}\left(\frac{\log N}{N}\right)$ per $n\to\infty$, se $L\geq S$, e cio\`{e}, che a successioni di partizioni $\{\rho_{L,S}^{n}\}$ a bassa discrepanza corrispondono successioni di punti $\{\xi_{L,S}^{n}\}$ a bassa discrepanza. Richiamiamo, a questo punto, il seguente
\begin{thm}
(i) Se $S<L+1$ esiste una costante positiva $k_1$ tale che per ogni $N\in\mathbb{N}$
\begin{equation*}
D_N(\xi_{L,S}^{1}, \xi_{L,S}^{2},\ldots, \xi_{L,S}^{N})\leq k_1\frac{\log N}{N}\ .
\end{equation*}
(ii) Se $S=L+1$ esiste  una costante positiva $k_2$ tale che per ogni $N\in\mathbb{N}$
\begin{equation*}
D_N(\xi_{L,S}^{1}, \xi_{L,S}^{2},\ldots, \xi_{L,S}^{N})\leq k_2\frac{\log^{2} N}{N}\ .
\end{equation*}
(ii) Se $S>L+1$ esiste una costante positiva $k_3$ tale che per ogni $N\in\mathbb{N}$
\begin{equation*}
D_N(\xi_{L,S}^{1}, \xi_{L,S}^{2},\ldots, \xi_{L,S}^{N})\leq k_3\frac{\log N}{N^{\gamma}}\ ,
\end{equation*}
dove $\gamma=1+\frac{\log S\beta}{\log\beta}< 1$.
\end{thm}
\paragraph*{}
Per maggiore chiarezza, facciamo vedere come si ottengono le successioni di punti corrispondenti alle successioni di partizioni presentate negli Esempi \ref{es}.
\begin{es}
\label{es:2.3.2}
\rm{\textbf{1.} Analizziamo dapprima la successione di punti $\{\xi_{1,1}^{n}\}_{n\in\mathbb{N}}$. L'insieme di punti corrispondente alla partizione $\rho_{1,1}^{1}$ \`{e} costituito esattamente dagli estremi sinistri dei due intervalli $[0,\beta[, [\beta,1[$, cio\`{e}
\begin{equation*}
\Lambda_{1,1}^{1} = \Big(0,\beta\Big)=\Big(\xi_{1,1}^{1}, \xi_{1,1}^{2}\Big)\ .
\end{equation*}
A questo punto, osserviamo che in questo esempio non ci sono funzioni (\ref{funzioni_phi}) del tipo $\varphi_{L,j}^{n+1}$ in quanto abbiamo un solo intervallo corto e quindi $j=0$, ma solo la funzione $\varphi_1^{n+1}$. Poich\`{e} $l_1=1$, applicando la funzione $\varphi_1^{(2)}$ al punto zero si ha
\begin{equation*}
\Lambda_{1,1}^{2} = \Big(0,\beta,\varphi_1^{(2)}(0)\Big)=\Big(0,\beta,\beta^{2}\Big)=\Big(\xi_{1,1}^{1}, \xi_{1,1}^{2}, \xi_{1,1}^{3}\Big)\ .
\end{equation*}
Poich\`{e} $l_2=2$, applichiamo la funzione $\varphi_1^{(3)}$ ai punti $0$ e $\beta$, e precisamente
\begin{eqnarray*}
\Lambda_{1,1}^{3} &=& \Big(0,\beta, \beta^{2},\varphi_1^{(3)}(0),\varphi_1^{(3)}(\beta)\Big)\\
&=&\Big(0,\beta,\beta^{2},\beta^{3},\beta+\beta^{3}\Big)=\Big(\xi_{1,1}^{1}, \xi_{1,1}^{2}, \xi_{1,1}^{3}, \xi_{1,1}^{4}, \xi_{1,1}^{5}\Big)\ .
\end{eqnarray*}
Applicando poi la funzione $\varphi_1^{(4)}$ ai primi tre punti (perch\`{e} $l_3=3$), otteniamo $\Lambda_{1,1}^{4}$ di cui riportiamo il grafico che segue.\\

\begin{figure}[h!]
\begin{center}
\begin{tikzpicture}[scale=13]
\draw (0, 0) -- (1,0);
\draw (0,-0.01) node[below, black]{0} -- (0,0.01) node[above, black]{\scriptsize$1^\circ$};
\draw (1,-0.01) -- (1,0.01);
\draw (0.61803,-0.01) node[below, black]{\scriptsize$\beta$}-- (0.61803,0.01) node[above, black]{\scriptsize$2^\circ$};
\draw (0.38196,-0.01) node[below, black]{\scriptsize$\beta^{2}$}-- (0.38196,0.01) node[above, black]{\scriptsize$3^\circ$};
\draw (0.23606,-0.01) node[below, black]{\scriptsize$\beta^{3}$}-- (0.23606,0.01) node[above, black]{\scriptsize$4^\circ$};
\draw (0.85410,-0.01) node[below, black]{\scriptsize$\beta +\beta^{3}$}-- (0.85410,0.01) node[above, black]{\scriptsize$5^\circ$};
\draw (0.14589,-0.01) node[below, black]{\scriptsize$\beta^{4}$}-- (0.14589,0.01) node[above, black]{\scriptsize$6^\circ$};
\draw (0.76392,-0.01) node[below, black]{\scriptsize$\beta +\beta^{4}$}-- (0.76392,0.01) node[above, black]{\scriptsize$7^\circ$};
\draw (0.52785,-0.01) node[below, black]{\scriptsize$\beta +\beta^{3}$}-- (0.52785,0.01) node[above, black]{\scriptsize$8^\circ$};
\end{tikzpicture}
\end{center}
\caption{Primi 8 punti di $\{\xi_{1,1}^{n}\}_{n\in\mathbb{N}}$}\label{fig:11}
\end{figure}

Cos\`{i} continuando otteniamo la successione di punti $\{\xi_{1,1}^{n}\}_{n\in\mathbb{N}}$ associata alla successione di punti di partizioni $\{\rho_{1,1}^{n}\}_{n\in\mathbb{N}}$. \\ Questa successione di punti \`{e} stata chiamata in \cite{Carbone} \emph{successione di punti Kakutani-Fibonacci}, dal momento che discende dalla successione di partizioni di Kakutani-Fibonacci. 
\paragraph*{}
\textbf{2.} Consideriamo, adesso, la successione di punti $\{\xi_{2,1}^{n}\}_{n\in\mathbb{N}}$. Anche in questo caso non ci sono funzioni del tipo $\varphi_{L,j}^{(n+1)}$ da considerare, ma soltanto le funzioni del tipo $\varphi_i^{(n+1)}$, con $i=1, 2$, da applicare ai primi $l_n$ punti dell'insieme $\Lambda_{2,1}^{n}$. Il primo insieme di punti \`{e} 
\begin{equation*}
\Lambda_{2,1}^{1} = \Big(0,\beta,2\beta\Big)=\Big(\xi_{2,1}^{1}, \xi_{2,1}^{2}\Big)\ .
\end{equation*}
Poi, essendo $l_1=2$, se applichiamo le due funzioni $\varphi_1^{(2)}$ e $\varphi_2^{(2)}$ ai primi due punti di $\Lambda_{2,1}^{1}$ abbiamo
\begin{eqnarray*}
\Lambda_{2,1}^{2} &=& \Big(0,\beta,\varphi_1^{(2)}(0),\varphi_1^{(2)}(\beta),\varphi_2^{(2)}(0),\varphi_2^{(2)}(\beta)\Big)\\
&=&\Big(0,\beta,2\beta,\beta^{2},\beta+\beta^{2},2\beta^{2},\beta+2\beta^{2}\Big)\\
&=&\Big(\xi_{2,1}^{1}, \xi_{2,1}^{2}, \xi_{2,1}^{3}, \xi_{2,1}^{4}, \xi_{2,1}^{5}, \xi_{2,1}^{6}, \xi_{2,1}^{7}\Big)\ .
\end{eqnarray*}
Calcoliamo ancora un insieme di punti, quello relativo alla partizione $\rho_{2,1}^{3}$. 
Applicando le funzioni $\varphi_1^{(3)}$ e $\varphi_2^{(3)}$ ai primi cinque punti di $\Lambda_{2,1}^{2}$ (visto che $l_2=5$), si ha
\begin{eqnarray*}
\Lambda_{2,1}^{3} &=&\Big(0,\beta,2\beta,\beta^{2},\beta+\beta^{2},2\beta^{2},\beta+2\beta^{2},\\
& & \varphi_1^{(3)}(0),\varphi_1^{(3)}(\beta),\varphi_1^{(3)}(2\beta),\varphi_1^{(3)}(\beta^{2}),\varphi_1^{(3)}(\beta+\beta^{2}),\\
& & \varphi_2^{(3)}(0),\varphi_2^{(3)}(\beta),\varphi_2^{(3)}(2\beta),\varphi_2^{(3)}(\beta^{2}),\varphi_2^{(3)}(\beta+\beta^{2})\Big)\\
&=& \Big(0,\beta, 2\beta, \beta^{2},\beta+\beta^{2}, 2\beta^{2},\beta+2\beta^{2}, \\
& & \beta^{3},\beta +\beta^{3}, 2\beta+\beta^{3}, \beta^{2}+\beta^{3}, \beta+\beta^{2}+\beta^{3}, \\
& & 2\beta^{3},\beta+2\beta^{3},2\beta+2\beta^{3},\beta^{2}+2\beta^{3},\beta+\beta^{2}+2\beta^{3}\Big)\\
&=&\Big(\xi_{2,1}^{1}, \xi_{2,1}^{2}, \xi_{2,1}^{3}, \xi_{2,1}^{4}, \xi_{2,1}^{5}, \xi_{2,1}^{6}, \xi_{2,1}^{7}, \xi_{2,1}^{8},\\
&=& \xi_{2,1}^{9}, \xi_{2,1}^{10}, \xi_{2,1}^{11}, \xi_{2,1}^{12}, \xi_{2,1}^{13}, \xi_{2,1}^{14}, \xi_{2,1}^{15}, \xi_{2,1}^{16},\xi_{2,1}^{17}\Big) \ .
\end{eqnarray*}
Tali punti hanno la seguente rappresentazione grafica:

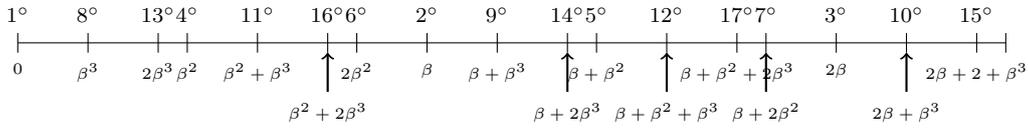
\begin{figure}[h!]
\begin{center}
\begin{tikzpicture}[scale=13]
\draw (0, 0) -- (1,0);
\draw (0,-0.01) node[below, black]{\tiny 0} -- (0,0.01) node[above, black]{\scriptsize $1^\circ$};
\draw (1,-0.01) -- (1,0.01);
\draw (0.41421,-0.01) node[below, black]{\tiny $\beta$}-- (0.41421,0.01) node[above, black]{\scriptsize $2^\circ$};
\draw (0.82842,-0.01) node[below, black]{\tiny $2\beta$}-- (0.82842,0.01) node[above, black]{\scriptsize $3^\circ$};
\draw (0.17157,-0.01) node[below, black]{\tiny $\beta^{2}$}-- (0.17157,0.01) node[above, black]{\scriptsize $4^\circ$};
\draw (0.58578,-0.01) node[below, black]{\tiny $\beta +\beta^{2}$}-- (0.58578,0.01) node[above, black]{\scriptsize $5^\circ$};
\draw (0.34314,-0.01) node[below, black]{\tiny $2\beta^{2}$}-- (0.34314,0.01) node[above, black]{\scriptsize $6^\circ$};
\draw (0.75735,-0.01) -- (0.75735,0.01) node[above, black]{\scriptsize $7^\circ$};
\draw [thick, ->] (0.75735,-0.05) node[below, black]{\tiny $\beta +2\beta^{2}$} -- (0.75735,-0.01);
\draw (0.07107,-0.01) node[below, black]{\tiny $\beta^{3}$}-- (0.07107,0.01) node[above, black]{\scriptsize $8^\circ$};
\draw (0.48528,-0.01) node[below, black]{\tiny $\beta +\beta^{3}$}-- (0.48528,0.01) node[above, black]{\scriptsize $9^\circ$};
\draw (0.89949,-0.01) -- (0.89949,0.01) node[above, black]{\scriptsize $10^\circ$};
\draw [thick, ->] (0.89949,-0.05) node[below, black]{\tiny $2\beta +\beta^{3}$} -- (0.89949,-0.01);
\draw (0.24264,-0.01) node[below, black]{\tiny $\beta^2+\beta^{3}$} -- (0.24264,0.01)node[above, black]{\scriptsize $11^\circ$} ;
\draw (0.65685,-0.01) -- (0.65685,0.01) node[above, black]{\scriptsize $12^\circ$};
\draw [thick, ->] (0.65685,-0.05) node[below, black]{\tiny $\beta +\beta^2+\beta^{3}$} -- (0.65685,-0.01);
\draw (0.14213,-0.01) node[below, black]{\tiny $2\beta^{3}$}-- (0.14213,0.01) node[above, black]{\scriptsize $13^\circ$};
\draw (0.55634,-0.01) -- (0.55634,0.01) node[above, black]{\scriptsize $14^\circ$};
\draw [thick, ->] (0.55634,-0.05) node[below, black]{\tiny $\beta +2\beta^{3}$} -- (0.55634,-0.01);
\draw (0.97055,-0.01) node[below, black]{\tiny $2\beta +2+\beta^{3}$}-- (0.97055,0.01) node[above, black]{\scriptsize $15^\circ$};
\draw (0.3137,-0.01) -- (0.3137,0.01) node[above, black]{\scriptsize $16^\circ$};
\draw [thick, ->] (0.3137,-0.05) node[below, black]{\tiny $\beta^2+2\beta^{3}$} -- (0.3137,-0.01);
\draw (0.72791,-0.01) node[below, black]{\tiny $\beta+\beta^2+2\beta^{3}$}-- (0.72791,0.01) node[above, black]{\scriptsize $17^\circ$};
\end{tikzpicture}
\end{center}
\caption{Primi 17 punti di $\{\xi_{2,1}^{n}\}_{n\in\mathbb{N}}$}\label{fig:21}
\end{figure}

Iterando questo processo riusciamo a determinare tutti i punti della successione $\{\xi_{2,1}^{n}\}_{n\in\mathbb{N}}$.
\paragraph*{}
\textbf{3.} Consideriamo infine l'esempio delle successioni di punti $\{\xi_{1,2}^{n}\}_{n\in\mathbb{N}}$ e determiniamo l'insieme di punti associato a ciascuna delle tre partizioni $\rho_{1,2}^{1}$, $\rho_{1,2}^{2}$ e $\rho_{1,2}^{3}$. Il primo insieme di punti si determina facilmente sempre considerando gli estremi sinistri degli intervalli di $\rho_{1,2}^{1}$, cio\`{e}
\begin{equation*}
\Lambda_{1,2}^{1}=\Big(0,\beta,\beta+\beta^{2}\Big)=\Big(\xi_{1,2}^{1}, \xi_{1,2}^{2},\xi_{1,2}^{3}\Big)\ .
\end{equation*} 
Poich\`{e} $l_1=1$, dobbiamo applicare le funzioni $\varphi_1^{(2)}$ e $\varphi_{1,1}^{(2)}$ solo al primo punto, e cio\`{e} allo zero. Per cui
\begin{equation*}
\Lambda_{1,2}^{2}=\Big(0,\beta,\beta+\beta^{2},\varphi_1^{(2)}(0),\varphi_{1,1}^{(2)}(0)\Big)
\end{equation*}
\begin{eqnarray*}
&=& \Big(0,\beta,\beta+\beta^{2},\beta^{2},\beta^{2}+\beta^{3}\Big)\\
&=& \Big(\xi_{1,2}^{1}, \xi_{1,2}^{2}, \xi_{1,2}^{3}, \xi_{1,2}^{4},\xi_{1,2}^{5}\Big)\ .
\end{eqnarray*} 
Poich\`{e} $l_2=3$, procedendo come nei casi precedenti si trova che
\begin{eqnarray*}
\Lambda_{1,2}^{3}&=& \Big(0,\beta,\beta+\beta^{2},\beta^{2},\beta^{2}+\beta^{3},\\
& & \varphi_1^{(3)}(0),\varphi_1^{(3)}(\beta), \varphi_1^{(3)}(\beta+\beta^{2}),\varphi_{1,1}^{(3)}(0),\varphi_{1,1}^{(3)}(\beta), \varphi_{1,1}^{(3)}(\beta+\beta^{2})\Big)\\
&=& \Big(0,\beta,\beta+\beta^{2},\beta^{2},\beta^{2}+\beta^{3},\beta^{3},\beta+\beta^{3}, \\
& & \beta+\beta^{2}+\beta^{3},\beta^{3}+\beta^{4}, \beta+\beta^{3}+\beta^{4}, \beta+\beta^{2}+\beta^{3}+\beta^{4}\Big)\\
&=& \Big(\xi_{1,2}^{1},\xi_{1,2}^{2},\xi_{1,2}^{3},\xi_{1,2}^{4},\xi_{1,2}^{5}, \xi_{1,2}^{6},\xi_{1,2}^{7},\xi_{1,2}^{8},\xi_{1,2}^{9},\xi_{1,2}^{10}, \xi_{1,2}^{11}\Big) \ .
\end{eqnarray*}
E cos via.
\\L'insieme di punti $\Lambda_{1,2}^{3}$ si pu rappresentare nel modo seguente:\\

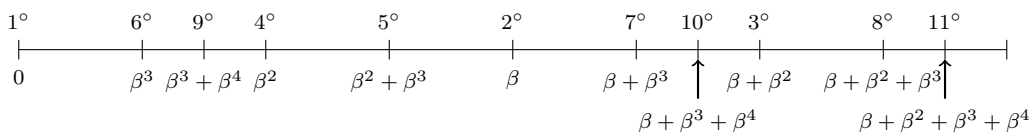
\begin{figure}[h!]
\begin{center}
\begin{tikzpicture}[scale=13]
\draw (0, 0) -- (1,0);
\draw (0,-0.01) node[below, black]{\scriptsize 0} -- (0,0.01) node[above, black]{\scriptsize $1^\circ$};
\draw (1,-0.01) -- (1,0.01);
\draw (0.5,-0.01) node[below, black]{\scriptsize $\beta$}-- (0.5,0.01) node[above, black]{\scriptsize $2^\circ$};
\draw (0.75,-0.01) node[below, black]{\scriptsize $\beta+\beta^{2}$}-- (0.75,0.01) node[above, black]{\scriptsize $3^\circ$};
\draw (0.25,-0.01) node[below, black]{\scriptsize $\beta^{2}$}-- (0.25,0.01) node[above, black]{\scriptsize $4^\circ$};
\draw (0.375,-0.01) node[below, black]{\scriptsize $\beta^2 +\beta^{3}$}-- (0.375,0.01) node[above, black]{\scriptsize $5^\circ$};
\draw (0.125,-0.01) node[below, black]{\scriptsize $\beta^{3}$}-- (0.125,0.01) node[above, black]{\scriptsize $6^\circ$};
\draw (0.625,-0.01) node[below, black]{\scriptsize $\beta +\beta^{3}$}-- (0.625,0.01) node[above, black]{\scriptsize $7^\circ$};
\draw (0.875,-0.01) node[below, black]{\scriptsize $\beta+\beta^2 +\beta^{3}$}-- (0.875,0.01) node[above, black]{\scriptsize $8^\circ$};
\draw (0.1875,-0.01) node[below, black]{\scriptsize$\beta^3 +\beta^{4}$}-- (0.1875,0.01) node[above, black]{\scriptsize $9^\circ$};
\draw (0.6875,-0.01) -- (0.6875,0.01) node[above, black]{\scriptsize $10^\circ$};
\draw [thick, ->] (0.6875,-0.05) node[below, black]{\scriptsize $\beta +\beta^{3}+\beta^4$} -- (0.6875,-0.01);
\draw (0.9375,-0.01) -- (0.9375,0.01)node[above, black]{\scriptsize $11^\circ$} ;
\draw [thick, ->] (0.9375,-0.05) node[below, black]{\scriptsize $\beta+\beta^2 +\beta^{3}+\beta^4$} -- (0.9375,-0.01) ;
\end{tikzpicture}
\end{center}
\caption{Primi 11 punti di $\{\xi_{1,2}^{n}\}_{n\in\mathbb{N}}$}\label{fig:12}
\end{figure}
}
\end{es}
Nel prossimo paragrafo presenteremo un metodo meno complesso per il calcolo dei punti associati alle $LS$-successioni di partizioni.
\section{Procedimento ``alla van der Corput''}
In questo paragrafo costruiamo un certo numero di funzioni definite sull'intervallo unitario, dipendenti da $L$ ed $S$, che ci permetteranno di determinare gli insiemi di punti $\Lambda_{L,S}^{n}$ e, di conseguenza, i punti di $\{\xi_{L,S}^{n}\}_{n\in\mathbb{N}}$ \cite{Carbone2}. Per descrivere la loro costruzione abbiamo bisogno di introdurre la definizione di rappresentazione di un numero rispetto ad una base e una nuova famiglia di funzioni sull'intervallo unitario.
\subsection*{Rappresentazione dei numeri naturali}
\label{rappresentazione}
\begin{defn}
\rm{La rappresentazione di un numero \`{e} detta \emph{posizionale} se esso \`{e} rappresentato da una sequenza di cifre il cui valore dipende dalla posizione all'interno della sequenza.}
\end{defn}
Il nostro sistema di numerazione \`{e} posizionale decimale, ossia utilizza una rappresentazione in base $10$, in cui il valore delle cifre, da $0$ a $9$, \`{e} determinato dalla posizione che assume nella sequenza, per cui $1234$ \`{e} diverso da $4321$.
\paragraph*{}
A questo punto, data l'importanza che la rappresentazione in base $b$ ricoprir\`{a} nel seguito della tesi, chiariamo come ottenere tale rappresentazione.
\begin{defn}[Numeri in base $b\geq 2$]
\rm{Siano $b\geq 2$ un intero ed $n$ un numero naturale. Per rappresentare $n$ in base $b$, si divide $n$ per $b$ e si pu\`{o} presentare uno dei tre casi: 
\begin{enumerate}
\item $n=0$ 
\item $n$ $<$ $b$
\item $n>b$\ .
\end{enumerate}
Se $n=0$, oppure $n<b$, abbiamo finito. Se, invece, $n>b$, siano $a_0$ e $q_0$, rispettivamente, il resto ed il quoziente della divisione di $n$ per $b$. Proseguiamo dividendo $q_0$ per $b$ e siano $a_1$ e $q_1$, il resto ed il quoziente della divisione di $q_0$ per $b$, rispettivamente. Iteriamo questo procedimento fino a che non troviamo un resto $a_M$ nullo e, di conseguenza, una successione di resti del tipo $a_0a_1\ldots a_{M-1}a_M$. Allora la rappresentazione di $n$ in base $b$ \`{e}
\begin{equation*}
[n]_b=a_Ma_{n-1}\ldots a_1a_0\ .\smallskip
\end{equation*}
Ovviamente, il numero $n$ in base dieci \`{e}
\begin{equation}
n=[n]_{10}=a_Mb^{M}+a_{M-1}b^{M-1}+\ldots +a_1b+a_0=\sum_{i=0}^{M}a_ib^{i}\ .
\end{equation}}
\end{defn}
Tenendo conto della definizione precedente, possiamo ora estendere la Definizione \ref{radinv} di funzione radice inversa rispetto ad una generica base $b$.
\begin{defn}
\rm{Siano $b\geq 2$ un intero ed $n \in \mathbb{N}$. Se la rappresentazione in base $b$ di $n$ \`{e} $[n]_b=a_Ma_{M-1}\ldots a_0$, allora la funzione $\phi_b:\mathbb{N}\longrightarrow [0,1]$, definita come
\begin{equation*}
\phi_b(n)=\sum_{i=0}^{M}\dfrac{a_i}{b^{i+1}}\ ,
\end{equation*}
\`{e} detta \emph{funzione radice inversa}, e la sua rappresentazione in base $b$ \`{e} \linebreak$[\phi_b(n)]_b=0.a_0\ldots a_{M-1}a_M$.}
\end{defn}
\begin{eso}
\rm{L'espressione del numero $70$ in base $3$ \`{e} $2121$. Infatti, \linebreak $70:3=23$ con resto $a_0=1$, $23:3=7$ con resto $a_2=2$, $7:3=2$ con resto $a_2=1$ e $2:3=0$ con resto $a_3=2$. Il resto successivo \`{e} nullo e quindi il processo si arresta con $a_3$.\\ Pertanto $[70]_3=2121$ e
\begin{equation*}
[70]_{10}=\sum_{i=0}^{3}a_i3^{i}=1\cdot3^{0}+2\cdot3^{1}+1\cdot3^{2}+2\cdot3^{3}\ .
\end{equation*}
Inoltre $$\phi_3(70)=\sum_{i=0}^{3}\dfrac{a_i}{3^{i+1}}=\frac{1}{3}+\frac{2}{3^{2}}+\frac{1}{3^{3}}+\frac{2}{3^{4}}$$ e $[\phi_3(70)]_3=0.1212$.
}

\end{eso}
Sappiamo che la successione di van der Corput si ottiene considerando la rappresentazione in base $2$ dei numeri naturali e invertendo l'ordine delle sue cifre (si veda Definizione \ref{radinv} e Definizione \ref{van der Corput}). Se poniamo $b=2$ e $\alpha=\frac{1}{2}$, si hanno le seguenti relazioni:

\begin{center}
\begin{tabular}{rlllrrlr}

\begin{footnotesize}
$0\rightarrow$
\end{footnotesize}&\begin{footnotesize}
$[0]_2$
\end{footnotesize}&\begin{footnotesize}
$=0$
\end{footnotesize}& \begin{footnotesize}
$\rightarrow$
\end{footnotesize}& \begin{footnotesize}
$0.0=$
\end{footnotesize}& \begin{footnotesize}
$[\phi_2(0)]_2\rightarrow$
\end{footnotesize}& \begin{footnotesize}
$0$
\end{footnotesize}& \begin{footnotesize}
$=\phi_2(0)$
\end{footnotesize}\\

\begin{footnotesize}
$1\rightarrow$
\end{footnotesize}&\begin{footnotesize}
$[1]_2$
\end{footnotesize}&\begin{footnotesize}
$=01$
\end{footnotesize}& \begin{footnotesize}
$\rightarrow$
\end{footnotesize}& \begin{footnotesize}
$0.10=$
\end{footnotesize}& \begin{footnotesize}
$[\phi_2(1)]_2\rightarrow$
\end{footnotesize}& \begin{footnotesize}
$\alpha$
\end{footnotesize}& \begin{footnotesize}
$=\phi_2(1)$
\end{footnotesize}\\

\begin{footnotesize}
$2\rightarrow$
\end{footnotesize}&\begin{footnotesize}
$[2]_2$
\end{footnotesize}&\begin{footnotesize}
$=10$
\end{footnotesize}& \begin{footnotesize}
$\rightarrow$
\end{footnotesize}& \begin{footnotesize}
$0.01=$
\end{footnotesize}& \begin{footnotesize}
$[\phi_2(2)]_2\rightarrow$
\end{footnotesize}& \begin{footnotesize}
$\alpha^{2}$
\end{footnotesize}& \begin{footnotesize}
$=\phi_2(2)$
\end{footnotesize}\\

\begin{footnotesize}
$3\rightarrow$
\end{footnotesize}&\begin{footnotesize}
$[3]_2$
\end{footnotesize}&\begin{footnotesize}
$=11$
\end{footnotesize}& \begin{footnotesize}
$\rightarrow$
\end{footnotesize}& \begin{footnotesize}
$0.11=$
\end{footnotesize}& \begin{footnotesize}
$[\phi_2(3)]_2\rightarrow$
\end{footnotesize}& \begin{footnotesize}
$\alpha+\alpha^{2}$
\end{footnotesize}& \begin{footnotesize}
$=\phi_2(3)$
\end{footnotesize}\\

\begin{footnotesize}
$4\rightarrow$
\end{footnotesize}&\begin{footnotesize}
$[4]_2$
\end{footnotesize}&\begin{footnotesize}
$=100$
\end{footnotesize}& \begin{footnotesize}
$\rightarrow$
\end{footnotesize}& \begin{footnotesize}
$0.001=$
\end{footnotesize}& \begin{footnotesize}
$[\phi_2(4)]_2\rightarrow$
\end{footnotesize}& \begin{footnotesize}
$\alpha^{3}$
\end{footnotesize}& \begin{footnotesize}
$=\phi_2(4)$
\end{footnotesize}\\

\begin{footnotesize}
$5\rightarrow$
\end{footnotesize}&\begin{footnotesize}
$[5]_2$
\end{footnotesize}&\begin{footnotesize}
$=101$
\end{footnotesize}& \begin{footnotesize}
$\rightarrow$
\end{footnotesize}& \begin{footnotesize}
$0.101=$
\end{footnotesize}& \begin{footnotesize}
$[\phi_2(5)]_2\rightarrow$
\end{footnotesize}& \begin{footnotesize}
$\alpha+\alpha^{3}$
\end{footnotesize}& \begin{footnotesize}
$=\phi_2(5)$
\end{footnotesize}\\

\begin{footnotesize}
$6\rightarrow$
\end{footnotesize}&\begin{footnotesize}
$[6]_2$
\end{footnotesize}&\begin{footnotesize}
$=110$
\end{footnotesize}& \begin{footnotesize}
$\rightarrow$
\end{footnotesize}& \begin{footnotesize}
$0.011=$
\end{footnotesize}& \begin{footnotesize}
$[\phi_2(6)]_2\rightarrow$
\end{footnotesize}& \begin{footnotesize}
$\alpha^{2}+\alpha^{3}$
\end{footnotesize}& \begin{footnotesize}
$=\phi_2(6)$
\end{footnotesize}\\

\begin{footnotesize}
$7\rightarrow$
\end{footnotesize}&\begin{footnotesize}
$[7]_2$
\end{footnotesize}&\begin{footnotesize}
$=111$
\end{footnotesize}& \begin{footnotesize}
$\rightarrow$
\end{footnotesize}& \begin{footnotesize}
$0.111=$
\end{footnotesize}& \begin{footnotesize}
$[\phi_2(7)]_2\rightarrow$
\end{footnotesize}& \begin{footnotesize}
$\alpha+\alpha^{2}+\alpha^{3}$
\end{footnotesize}& \begin{footnotesize}
$=\phi_2(7)$
\end{footnotesize}\\

\begin{footnotesize}
$8\rightarrow$
\end{footnotesize}&\begin{footnotesize}
$[8]_2$
\end{footnotesize}&\begin{footnotesize}
$=1000$
\end{footnotesize}& \begin{footnotesize}
$\rightarrow$
\end{footnotesize}& \begin{footnotesize}
$0.0001=$
\end{footnotesize}& \begin{footnotesize}
$[\phi_2(8)]_2\rightarrow$
\end{footnotesize}& \begin{footnotesize}
$\alpha^{4}$
\end{footnotesize}& \begin{footnotesize}
$=\phi_2(8)$
\end{footnotesize}\\

\begin{footnotesize}
$9\rightarrow$
\end{footnotesize}&\begin{footnotesize}
$[9]_2$
\end{footnotesize}&\begin{footnotesize}
$=1001$
\end{footnotesize}& \begin{footnotesize}
$\rightarrow$
\end{footnotesize}& \begin{footnotesize}
$0.1001=$
\end{footnotesize}& \begin{footnotesize}
$[\phi_2(9)]_2\rightarrow$
\end{footnotesize}& \begin{footnotesize}
$\alpha+\alpha^{4}$
\end{footnotesize}& \begin{footnotesize}
$=\phi_2(9)$
\end{footnotesize}\\

\begin{footnotesize}
$10\rightarrow$
\end{footnotesize}&\begin{footnotesize}
$[10]_2$
\end{footnotesize}&\begin{footnotesize}
$=1010$
\end{footnotesize}& \begin{footnotesize}
$\rightarrow$
\end{footnotesize}& \begin{footnotesize}
$0.0101=$
\end{footnotesize}& \begin{footnotesize}
$[\phi_2(10)]_2\rightarrow$
\end{footnotesize}& \begin{footnotesize}
$\alpha^{2}+\alpha^{4}$
\end{footnotesize}& \begin{footnotesize}
$=\phi_2(10)$
\end{footnotesize}\\
\end{tabular}
\end{center}

\`{E} interessante e utile notare a questo punto che tale successione si ottiene riordinando i punti determinati dalla successione di partizioni di Kakutani per $\alpha=\frac{1}{2}$ (o, altrimenti detto, la successione di partizioni binaria).\medskip \\
\`{E} immediato osservare che le potenze di $\alpha$ dipendono dalla posizione decimale e questo ci permette di conoscere l'$n$-simo termine della successione di van der Corput utilizzando $[n]_2$. \\ Questa continua ad essere un'osservazione valida anche se consideriamo basi diverse da $2$.\\ A questo punto ci chiediamo se esista una corrispondenza simile per i punti delle successioni $\{\xi_{L,S}^{n}\}_{n\in\mathbb{N}}$, e cio\`{e} tra le potenze di $\beta$, dove $L\beta+S\beta^{2}=1$, e la rappresentazione in qualche base. Scopriremo che in un certo senso \`{e} cos\`{i}. \\ Partiamo dagli esempi gi\`{a} considerati nel paragrafo precedente, e vediamo come riottenere i punti delle successioni attraverso l'introduzione di nuove funzioni.
\begin{eso}
\label{Esempio2.4.1}
\rm{In questo esempio prendiamo in considerazione la successione di punti di Kakutani-Fibonacci $\{\xi_{1,1}^{n}\}_{n\in\mathbb{N}}$. Partiamo dalla successione di partizioni di Kakutani-Fibonacci $\{\rho_{1,1}^{n}\}_{n\in\mathbb{N}}$ e definiamo le seguenti funzioni
\begin{eqnarray}
&\psi_0 : [0,1[\longrightarrow [0,\beta[\\
& x \mapsto \beta x \nonumber
\end{eqnarray}
e
\begin{eqnarray}
&\psi_1 : [0,\beta[\longrightarrow [\beta ,1[ \\
& x \mapsto \beta x + \beta \ . \nonumber
\end{eqnarray}
Vedremo immediatamente che tutti i punti della successione $\{\xi_{1,1}^{n}\}_{n\in\mathbb{N}}$ si ottengono attraverso composizioni successive di $\psi_0$ e $\psi_1$.\\
Osserviamo che le composizioni possibili tra le due funzioni sono $\psi_0\circ \psi_0$, $\psi_0\circ \psi_1$ e $\psi_1\circ \psi_0$, mentre $\psi_1\circ \psi_1$ non \`{e} consentita. Indicheremo con $\psi_{ij}$ le composizioni $\psi_i\circ \psi_j$. Osserviamo inoltre che l'insieme dei punti relativi alla prima partizione \`{e} dato da
\begin{equation*}
\Lambda_{1,1}^{1}=\Big(\xi_{1,1}^{1}, \xi_{1,1}^{2}\Big)=\Big(\psi_0(0),\psi_1(0)\Big)=\big(0,\beta\big) \ .
\end{equation*}
Applicando queste funzioni (ovviamente quelle possibili) ai punti trovati e osservando che $\psi_{i0}(0)=\psi_i(0)$ per ogni $i=0,1$, si ha proprio
\begin{eqnarray*}
\Lambda_{1,1}^{2}&=&\Big(\xi_{1,1}^{1}, \xi_{1,1}^{2}, \xi_{1,1}^{3}\Big)=\Big(\psi_0(0),\psi_1(0), \psi_{01}(0)\Big) \\
&=& \Big(\psi_{00}(0), \psi_{10}(0), \psi_{01}(0)\Big)\\
&=& \Big(0, \beta, \beta^{2}\Big)\ .
\end{eqnarray*}
Iterando nuovamente questo procedimento otteniamo

\begin{equation*}
\Lambda_{1,1}^{3} = \Big(\xi_{1,1}^{1}, \xi_{1,1}^{2}\xi_{1,1}^{3}, \xi_{1,1}^{4},\xi_{1,1}^{5}\Big)=\Big(\psi_0(0),\psi_1(0), \psi_{01}(0), \psi_{001}(0),\psi_{101}(0)\Big)
\end{equation*}
\begin{eqnarray*}
&=& \Big(\psi_{000}(0), \psi_{100}(0), \psi_{010}(0), \psi_{001}(0), \psi_{101}(0)\Big)\\
&=& \Big(0, \beta, \beta^{2}, \beta^{3}, \beta + \beta^{3}\Big)\ .
\end{eqnarray*}
Si vede chiaramente come la composizione ricorsiva delle funzioni definite in precedenza, che dipendono esclusivamente da $L$ ed $S$, dia luogo alle successioni di punti associate alle $LS$-successioni di partizioni, ma la caratteristica sorprendente consiste soprattutto nell'ordinamento dei punti, i quali presentano gi\`{a} un ordinamento ``alla van der Corput''. Questo appare particolarmente evidente se consideriamo l'insieme di punti successivo
\begin{eqnarray*}
\Lambda_{1,1}^{4}&=& \Big(\xi_{1,1}^{1}, \xi_{1,1}^{2}\xi_{1,1}^{3}, \xi_{1,1}^{4},\xi_{1,1}^{5},\xi_{1,1}^{6},\xi_{1,1}^{7},\xi_{1,1}^{8}\Big)\\
&=& \Big(\psi_0(0),\psi_1(0), \psi_{01}(0), \psi_{001}(0), \psi_{101}(0),\\
& &\psi_{0001}(0), \psi_{1001}(0), \psi_{0101}(0)\Big)\\
&=& \Big(\psi_{0000}(0), \psi_{1000}(0), \psi_{0100}(0), \psi_{0010}(0), \\
& & \psi_{1010}(0), \psi_{0001}(0), \psi_{1001}(0), \psi_{0101}(0)\Big)\\
&=& \Big(0, \beta, \beta^{2}, \beta^{3}, \beta + \beta^{3}, \beta^{4}, \beta + \beta^{4}, \beta^{2} + \beta^{4}\Big)\ .
\end{eqnarray*}
Cos\`{i} continuando si ottiene la successione di punti $\{\xi_{1,1}^{n}\}_{n\in\mathbb{N}}$.
\\A questo punto si tratta di osservare che esiste la seguente relazione tra le potenze di $\beta$ e la rappresentazione binaria dei numeri naturali:

\begin{center}
\begin{tabular}{rlcrrlcrl} 
\begin{footnotesize}
$0\longrightarrow$
\end{footnotesize}&
\begin{footnotesize}
$[0]_2$
\end{footnotesize}&
\begin{footnotesize}
$=$
\end{footnotesize}&\begin{footnotesize}
$00 \longrightarrow$
\end{footnotesize}&\begin{footnotesize}
$0.00$
\end{footnotesize}&\begin{footnotesize}
$=$
\end{footnotesize}&
\begin{footnotesize}
$[\phi_2(0)]_2$
\end{footnotesize}&
\begin{footnotesize}
$\longrightarrow$
\end{footnotesize}&
\begin{footnotesize}
$0$
\end{footnotesize}\\

\begin{footnotesize}
$1\longrightarrow$
\end{footnotesize}&
\begin{footnotesize}
$[1]_2$\end{footnotesize}&\begin{footnotesize}
$=$
\end{footnotesize}&
\begin{footnotesize}
$01$ $\longrightarrow$
\end{footnotesize}&
\begin{footnotesize}
$0.10$
\end{footnotesize}& \begin{footnotesize}
$=$
\end{footnotesize}&
\begin{footnotesize}
$[\phi_2(1)]_2$
\end{footnotesize}&\begin{footnotesize}
$\longrightarrow$
\end{footnotesize}&
\begin{footnotesize}
$\beta$
\end{footnotesize}\\

\begin{footnotesize}
$2\longrightarrow$
\end{footnotesize}&
\begin{footnotesize}
$[2]_2$\end{footnotesize}&\begin{footnotesize}
$=$
\end{footnotesize}&
\begin{footnotesize}
$10$ $\longrightarrow$
\end{footnotesize}&
\begin{footnotesize}
$0.01$\end{footnotesize}&\begin{footnotesize}
$=$
\end{footnotesize}&
\begin{footnotesize}
$\phi_2(2)]_2$
\end{footnotesize}&
\begin{footnotesize}
$\longrightarrow$
\end{footnotesize}&
\begin{footnotesize}
$\beta^{2}$
\end{footnotesize}\\

\begin{footnotesize}
$3\longrightarrow$
\end{footnotesize}&
\begin{footnotesize}
$[3]_2$\end{footnotesize}&\begin{footnotesize}
$=$
\end{footnotesize}&
\begin{footnotesize}
$11$ $\longrightarrow$
\end{footnotesize}&
\begin{footnotesize}
$0.11$
\end{footnotesize}& \begin{footnotesize}
$=$
\end{footnotesize}&
\begin{footnotesize}
$[\phi_2(3)]_2$
\end{footnotesize}&
\begin{footnotesize}
$\longrightarrow$
\end{footnotesize}&
\begin{footnotesize}
nessun $\beta$
\end{footnotesize}\\

\begin{footnotesize}
$4\longrightarrow$
\end{footnotesize}&
\begin{footnotesize}
$[4]_2$\end{footnotesize}&\begin{footnotesize}
=
\end{footnotesize}&
\begin{footnotesize}
$100$ $\longrightarrow$
\end{footnotesize}&
\begin{footnotesize}
$0.01$
\end{footnotesize}& \begin{footnotesize}
$=$
\end{footnotesize}&
\begin{footnotesize}
$[\phi_2(4)]_2$
\end{footnotesize}&
\begin{footnotesize}
$\longrightarrow$
\end{footnotesize}&
\begin{footnotesize}
$\beta^{3}$
\end{footnotesize}\\

\begin{footnotesize}
$5\longrightarrow$
\end{footnotesize}&
\begin{footnotesize}
$[5]_2$\end{footnotesize}&\begin{footnotesize}
=
\end{footnotesize}&
\begin{footnotesize}
$101$ $\longrightarrow$
\end{footnotesize}&
\begin{footnotesize}
$0.101$
\end{footnotesize}& \begin{footnotesize}
$=$
\end{footnotesize}&
\begin{footnotesize}
$[\phi_2(5)]_2$
\end{footnotesize}&
\begin{footnotesize}
$\longrightarrow$
\end{footnotesize}&
\begin{footnotesize}
$\beta+\beta^{3}$
\end{footnotesize}\\

\begin{footnotesize}
$6\longrightarrow$
\end{footnotesize}&
\begin{footnotesize}
$[6]_2$\end{footnotesize}&\begin{footnotesize}
=
\end{footnotesize}&
\begin{footnotesize}
$110$ $\longrightarrow$
\end{footnotesize}&
\begin{footnotesize}
$0.011$
\end{footnotesize}& \begin{footnotesize}
$=$
\end{footnotesize}&
\begin{footnotesize}
$[\phi_2(6)]_2$
\end{footnotesize}&
\begin{footnotesize}
$\longrightarrow$
\end{footnotesize}&
\begin{footnotesize}
nessun $\beta$
\end{footnotesize}\\

\begin{footnotesize}
$7\longrightarrow$
\end{footnotesize}&
\begin{footnotesize}
$[7]_2$\end{footnotesize}&\begin{footnotesize}
=
\end{footnotesize}&
\begin{footnotesize}
$111$ $\longrightarrow$
\end{footnotesize}&
\begin{footnotesize}
$0.111$
\end{footnotesize}& \begin{footnotesize}
$=$
\end{footnotesize}&
\begin{footnotesize}
$[\phi_2(7)]_2$
\end{footnotesize}&
\begin{footnotesize}
$\longrightarrow$
\end{footnotesize}&
\begin{footnotesize}
nessun $\beta$
\end{footnotesize}\\

\begin{footnotesize}
$8\longrightarrow$
\end{footnotesize}&
\begin{footnotesize}
$[8]_2$\end{footnotesize}&\begin{footnotesize}
=
\end{footnotesize}&
\begin{footnotesize}
$1000$ $\longrightarrow$
\end{footnotesize}&
\begin{footnotesize}
$0.0001$
\end{footnotesize}& \begin{footnotesize}
$=$
\end{footnotesize}&\begin{footnotesize}
$[\phi_2(8)]_2$
\end{footnotesize}&
\begin{footnotesize}
$\longrightarrow$
\end{footnotesize}&
\begin{footnotesize}
$\beta^{4}$
\end{footnotesize}\\

\begin{footnotesize}
$9\longrightarrow$
\end{footnotesize}&
\begin{footnotesize}
$[9]_2$\end{footnotesize}&\begin{footnotesize}
=
\end{footnotesize}&
\begin{footnotesize}
$1001$ $\longrightarrow$
\end{footnotesize}&
\begin{footnotesize}
$0.1001$
\end{footnotesize}& \begin{footnotesize}
$=$
\end{footnotesize}&
\begin{footnotesize}
$[\phi_2(9)]_2$
\end{footnotesize}&
\begin{footnotesize}
$\longrightarrow$
\end{footnotesize}&
\begin{footnotesize}
$\beta+\beta^{4}$
\end{footnotesize}\\

\begin{footnotesize}
$10\longrightarrow$
\end{footnotesize}&
\begin{footnotesize}
$[10]_2$\end{footnotesize}&\begin{footnotesize}
=
\end{footnotesize}&
\begin{footnotesize}
$1010$ $\longrightarrow$
\end{footnotesize}&
\begin{footnotesize}
$0.0101$
\end{footnotesize}& \begin{footnotesize}
$=$
\end{footnotesize}&
\begin{footnotesize}
$[\phi_2(10)]_2$
\end{footnotesize}&
\begin{footnotesize}
$\longrightarrow$
\end{footnotesize}&
\begin{footnotesize}
$\beta^{2}+\beta^{4}$
\end{footnotesize}\\

\end{tabular}
\end{center}
e cos\`{i} via. }
\end{eso}
\begin{ossi}
\label{Oss:2.4.2}
\rm{\textbf{1.} La $1,1$-successione di punti $\{\xi_{1,1}^{n}\}_{n\in\mathbb{N}}$ presenta una forte analogia con la costruzione della successione di van der Corput, anche se ad alcuni numeri espressi in forma binaria non corrisponde alcuna potenza di $\beta$. Questo non ci deve meravigliare, se si considera da una parte che solo le rappresentazioni binarie corrispondenti alle sole composizioni possibili tra $\psi_0$ e $\psi_1$ sono quelle a cui corrispondono potenze di $\beta$, dall'altra che nella costruzione di van der Corput al passo $n$ vengono divisi tutti i $2^{n}$ intervalli, e dunque si raddoppiano i punti (mentre per le $LS$-successioni si dividono solo gli intervalli lunghi).
\paragraph*{}
\textbf{2.} Infine, attraverso questa procedura ``alla van der Corput'', siamo in grado di determinare i punti della successione $\{\xi_{1,1}^{n}\}_{n\in\mathbb{N}}$ svincolandoci dagli insiemi $\Lambda_{1,1}^{n}$. Precisamente, i vantaggi legati all'introduzione di queste due funzioni $\psi_0$ e $\psi_1$ sono due. Il primo \`{e} che per determinare i punti associati all'$n$-esima partizione $\rho_{1,1}^{n}$ non \`{e} pi\`{u} necessario riordinare i punti di $\rho_{1,1}^{n-1}$ perch\`{e} essi si ottengono scrivendo tutte le $n$-uple ordinate di elementi dell'insieme $\{0,1\}$ (che sono $2^{n}$), cancellando le sequenze di $n$ cifre che presentano due cifre consecutive pari a $1$. \\Questo fatto si vede bene in $\Lambda_{1,1}^{2}$, $\Lambda_{1,1}^{3}$ e $\Lambda_{1,1}^{4}$.\\ L'altro vantaggio \`{e} che, mentre nella Definizione \ref{def:2.3.1} e nell'Esempio \ref{es:2.3.2} ogni insieme $\Lambda_{1,1}^{n}$ \`{e} ottenuto attraverso funzioni $\varphi_1^{(n)}$ dipendenti da $n$, in questo caso esistono per tutti gli $n$ solo le funzioni $\psi_0$ e $\psi_1$ (indipendenti da $n$), e ogni punto si ottiene attraverso composizioni successive di queste due funzioni.
\paragraph*{}
\textbf{3.} La $1,1$-successione di punti $\{\xi_{1,1}^{n}\}_{n\in\mathbb{N}}$ pu\`{o} essere ottenuta utilizzando la rappresentazione binaria dei numeri naturali, invertendo l'ordine delle sue cifre (ed ottenendo, cio\`{e}, la rappresentazione binaria della funzione radice inversa) ed eliminando le rappresentazioni che presentano due cifre consecutive pari ad $1$. Inoltre, i coefficienti delle potenze di $\beta$ sono esattamente quelli che compaiono nella rappresentazione binaria di $\phi_2$, come nel caso della successione di van der Corput. Notiamo, inoltre, che c'\`{e} una corrispondenza biunivoca tra i coefficienti delle potenze di $\beta$ e la sua rappresentazione attraverso composizioni successive delle funzioni $\psi_0$ e $\psi_1$. Se consideriamo, ad esempio, $\Lambda_{1,1}^{4}$, vediamo che $\psi_0$, $\psi_1$, $\psi_{01}$ $\psi_{001}$, $\psi_{101}$, $\psi_{0001}$, $\psi_{1001}$ corrispondono rispettivamente a $0.00$, $0.10$, $0.01$, $0.001$, $0.101$, $0.0001$, $0.1001$. E le cancellazioni delle composizioni non lecite (e cio\`{e} tutte quelle che usano $\psi_{1,1}$) corrispondono esattamente (nell'ordine giusto) alle cancellazioni dei valori di $[\phi_2(\cdot)]_2$ le cui sequenze di cifre contengono la coppia $11$.}
\end{ossi}
Passiamo ora agli altri due esempi, prima di fornire la regola generale.
\begin{eso}
\label{Esempio2.4.2}
\rm{Consideriamo la successione di punti $\{\xi_{2,1}^{n}\}_{n\in\mathbb{N}}$ partendo dalla successione di partizioni $\{\rho_{2,1}^{n}\}_{n\in\mathbb{N}}$ ottenuta per $L=2$, $S=1$. Definiamo tre funzioni sui sottointervalli di $[0,1]$ nel modo seguente
\begin{eqnarray}
&\psi_0 : [0,1[\longrightarrow [0,\beta[ \\
& x \mapsto \beta x \nonumber \ ,
\end{eqnarray}
\begin{eqnarray}
&\psi_1 : [0,1[\longrightarrow [\beta ,2\beta[ \\
& x \mapsto \beta x + \beta \nonumber
\end{eqnarray}
e
\begin{eqnarray}
&\psi_2 : [0,\beta[\longrightarrow [2\beta ,1[ \\
& x \mapsto \beta x + 2\beta \ . \nonumber
\end{eqnarray}
In questo caso, ovviamente, le uniche composizioni proibite sono $\psi_2\circ \psi_1$ e $\psi_2\circ \psi_2$. Valutiamo $\psi_0$, $\psi_1$ e $\psi_2$ nello zero e nuovamente otteniamo l'insieme dei punti associati alla prima partizione $\rho_{2,1}^{1}$:
\begin{equation*}
\Lambda_{2,1}^{1}=\Big(\xi_{2,1}^{1},\xi_{2,1}^{2},\xi_{2,1}^{3}\Big) =\Big(\psi_0(0),\psi_1(0),\psi_2(0)\Big)=\Big(0,\beta,2\beta\Big)\ .
\end{equation*}
Calcolando, ove possibile, le funzioni $(2.9)$, $(2.10)$ e $(2.11)$ nei punti ottenuti, ossia componendo tra loro queste funzioni, come nel caso della successione di punti di Kakutani-Fibonacci, si ottiene l'insieme dei punti relativi alla partizione successiva, cio\`{e} 
\begin{eqnarray*}
\Lambda_{2,1}^{2}&=& \Big(\xi_{2,1}^{1},\xi_{2,1}^{2},\xi_{2,1}^{3},\xi_{2,1}^{4},\xi_{2,1}^{5},\xi_{2,1}^{6},\xi_{2,1}^{7}\Big) \\
&=&\Big(\psi_0(0),\psi_1(0),\psi_2(0), \psi_{01}(0), \psi_{11}(0), \psi_{02}(0), \psi_{12}(0) \Big) \\
&=& \Big(\psi_{00}(0), \psi_{10}(0), \psi_{20}(0), \psi_{01}(0), \psi_{11}(0), \psi_{02}(0), \psi_{12}(0)\Big)\\
&=& \Big(0,\beta, 2\beta, \beta^{2},\beta+\beta^{2}, 2\beta^{2},\beta+2\beta^{2}\Big)\ .
\end{eqnarray*}
Anche in questo caso, come nell'esempio precedente, si ha $\psi_{i0}=\psi_i(0)$ per ogni $i=0,1,2$.
Iterando un'ultima volta questa procedura si ha
\begin{equation*}
\Lambda_{2,1}^{3}=\Big(\xi_{2,1}^{1},\xi_{2,1}^{2},\xi_{2,1}^{3},\xi_{2,1}^{4},\xi_{2,1}^{5},\xi_{2,1}^{6},\xi_{2,1}^{7},\xi_{2,1}^{8}, \Big.
\end{equation*}
\begin{eqnarray*}
& & \big.\xi_{2,1}^{9},\xi_{2,1}^{10},\xi_{2,1}^{11},\xi_{2,1}^{12},\xi_{2,1}^{15},\xi_{2,1}^{16},\xi_{2,1}^{17}\Big) \\
&=&\Big(\psi_{0}(0), \psi_{1}(0), \psi_{2}(0), \psi_{01}(0), \psi_{11}(0), \psi_{02}(0), \\
& & \psi_{12}(0),\psi_{001}(0),\psi_{101}(0), \psi_{011}(0), \psi_{111}(0),\\
& & \psi_{002}(0), \psi_{102}(0), \psi_{012}(0), \psi_{112}(0)\Big)\\
&=& \Big(\psi_{000}(0), \psi_{100}(0), \psi_{200}(0), \psi_{010}(0), \psi_{110}(0), \psi_{020}(0), \\
& & \psi_{120}(0), \psi_{001}(0), \psi_{101}(0), \psi_{201}(0), \psi_{011}(0), \psi_{111}(0), \\
& & \psi_{002}(0),\psi_{102}(0), \psi_{202}(2), \psi_{012}(0) ,\psi_{112}(0) \Big)\\
&=& \Big(0,\beta, 2\beta, \beta^{2},\beta+\beta^{2}, 2\beta^{2},\beta+2\beta^{2}, \beta^{3},\\
& & \beta +\beta^{3}, 2\beta+\beta^{3}, \beta^{2}+\beta^{3}, \beta+\beta^{2}+\beta^{3}, \\
& & 2\beta^{3},\beta+2\beta^{3},2\beta+2\beta^{3},\beta^{2}+2\beta^{3},\beta+\beta^{2}+2\beta^{3}\Big)\ ,
\end{eqnarray*}
e cos\`{i} via si ottiene $\{\xi_{2,1}^{n}\}_{n\in\mathbb{N}}$.\\ I punti ottenuti da composizioni ricorsive di queste funzioni risultano essere disposti secondo l'ordine giusto.\\ Vogliamo vedere quale relazione intercorre tra queste funzioni e la rappresentazione dei numeri naturali in base $3$ e se \`{e} lecito supporre che le composizioni non consentite corrispondano a rappresentazioni di numeri naturali che non danno luogo ad alcun punto della successione. Infatti
\begin{center}
\begin{tabular}{rlcrrlcrl} 
\begin{footnotesize}
$0\longrightarrow$
\end{footnotesize}&
\begin{footnotesize}
$[0]_3$
\end{footnotesize}&
\begin{footnotesize}
$=$
\end{footnotesize}&\begin{footnotesize}
$00 \longrightarrow$
\end{footnotesize}&\begin{footnotesize}
$0.00$
\end{footnotesize}&\begin{footnotesize}
$=$
\end{footnotesize}&
\begin{footnotesize}
$[\phi_3(0)]_3$
\end{footnotesize}&
\begin{footnotesize}
$\longrightarrow$
\end{footnotesize}&
\begin{footnotesize}
$0$
\end{footnotesize}\\

\begin{footnotesize}
$1\longrightarrow$
\end{footnotesize}&
\begin{footnotesize}
$[1]_3$\end{footnotesize}&\begin{footnotesize}
$=$
\end{footnotesize}&
\begin{footnotesize}
$01$ $\longrightarrow$
\end{footnotesize}&
\begin{footnotesize}
$0.10$
\end{footnotesize}& \begin{footnotesize}
$=$
\end{footnotesize}&
\begin{footnotesize}
$[\phi_3(1)]_3$
\end{footnotesize}&\begin{footnotesize}
$\longrightarrow$
\end{footnotesize}&
\begin{footnotesize}
$\beta$
\end{footnotesize}\\

\begin{footnotesize}
$2\longrightarrow$
\end{footnotesize}&
\begin{footnotesize}
$[2]_3$\end{footnotesize}&\begin{footnotesize}
$=$
\end{footnotesize}&
\begin{footnotesize}
$02$ $\longrightarrow$
\end{footnotesize}&
\begin{footnotesize}
$0.20$\end{footnotesize}&\begin{footnotesize}
$=$
\end{footnotesize}&
\begin{footnotesize}
$\phi_3(2)]_3$
\end{footnotesize}&
\begin{footnotesize}
$\longrightarrow$
\end{footnotesize}&
\begin{footnotesize}
$2\beta$
\end{footnotesize}\\

\begin{footnotesize}
$3\longrightarrow$
\end{footnotesize}&
\begin{footnotesize}
$[3]_3$\end{footnotesize}&\begin{footnotesize}
$=$
\end{footnotesize}&
\begin{footnotesize}
$10$ $\longrightarrow$
\end{footnotesize}&
\begin{footnotesize}
$0.01$
\end{footnotesize}& \begin{footnotesize}
$=$
\end{footnotesize}&
\begin{footnotesize}
$[\phi_3(3)]_3$
\end{footnotesize}&
\begin{footnotesize}
$\longrightarrow$
\end{footnotesize}&
\begin{footnotesize}
$\beta^{2}$
\end{footnotesize}\\

\begin{footnotesize}
$4\longrightarrow$
\end{footnotesize}&
\begin{footnotesize}
$[4]_3$\end{footnotesize}&\begin{footnotesize}
=
\end{footnotesize}&
\begin{footnotesize}
$11$ $\longrightarrow$
\end{footnotesize}&
\begin{footnotesize}
$0.11$
\end{footnotesize}& \begin{footnotesize}
$=$
\end{footnotesize}&
\begin{footnotesize}
$[\phi_3(4)]_3$
\end{footnotesize}&
\begin{footnotesize}
$\longrightarrow$
\end{footnotesize}&
\begin{footnotesize}
$\beta+\beta^{2}$
\end{footnotesize}\\

\begin{footnotesize}
$5\longrightarrow$
\end{footnotesize}&
\begin{footnotesize}
$[5]_3$\end{footnotesize}&\begin{footnotesize}
=
\end{footnotesize}&
\begin{footnotesize}
$12$ $\longrightarrow$
\end{footnotesize}&
\begin{footnotesize}
$0.21$
\end{footnotesize}& \begin{footnotesize}
$=$
\end{footnotesize}&
\begin{footnotesize}
$[\phi_3(5)]_3$
\end{footnotesize}&
\begin{footnotesize}
$\longrightarrow$
\end{footnotesize}&
\begin{footnotesize}
nessun $\beta$
\end{footnotesize}\\

\begin{footnotesize}
$6\longrightarrow$
\end{footnotesize}&
\begin{footnotesize}
$[6]_3$\end{footnotesize}&\begin{footnotesize}
=
\end{footnotesize}&
\begin{footnotesize}
$20$ $\longrightarrow$
\end{footnotesize}&
\begin{footnotesize}
$0.02$
\end{footnotesize}& \begin{footnotesize}
$=$
\end{footnotesize}&
\begin{footnotesize}
$[\phi_3(6)]_3$
\end{footnotesize}&
\begin{footnotesize}
$\longrightarrow$
\end{footnotesize}&
\begin{footnotesize}
$2\beta^{2}$
\end{footnotesize}\\

\begin{footnotesize}
$7\longrightarrow$
\end{footnotesize}&
\begin{footnotesize}
$[7]_3$\end{footnotesize}&\begin{footnotesize}
=
\end{footnotesize}&
\begin{footnotesize}
$21$ $\longrightarrow$
\end{footnotesize}&
\begin{footnotesize}
$0.12$
\end{footnotesize}& \begin{footnotesize}
$=$
\end{footnotesize}&
\begin{footnotesize}
$[\phi_3(7)]_3$
\end{footnotesize}&
\begin{footnotesize}
$\longrightarrow$
\end{footnotesize}&
\begin{footnotesize}
$\beta+2\beta^{2}$
\end{footnotesize}\\

\begin{footnotesize}
$8\longrightarrow$
\end{footnotesize}&
\begin{footnotesize}
$[8]_3$\end{footnotesize}&\begin{footnotesize}
=
\end{footnotesize}&
\begin{footnotesize}
$22$ $\longrightarrow$
\end{footnotesize}&
\begin{footnotesize}
$0.22$
\end{footnotesize}& \begin{footnotesize}
$=$
\end{footnotesize}&\begin{footnotesize}
$[\phi_3(8)]_3$
\end{footnotesize}&
\begin{footnotesize}
$\longrightarrow$
\end{footnotesize}&
\begin{footnotesize}
nessun $\beta$
\end{footnotesize}\\

\begin{footnotesize}
$9\longrightarrow$
\end{footnotesize}&
\begin{footnotesize}
$[9]_3$\end{footnotesize}&\begin{footnotesize}
=
\end{footnotesize}&
\begin{footnotesize}
$100$ $\longrightarrow$
\end{footnotesize}&
\begin{footnotesize}
$0.001$
\end{footnotesize}& \begin{footnotesize}
$=$
\end{footnotesize}&
\begin{footnotesize}
$[\phi_3(9)]_3$
\end{footnotesize}&
\begin{footnotesize}
$\longrightarrow$
\end{footnotesize}&
\begin{footnotesize}
$\beta^{3}$
\end{footnotesize}\\

\begin{footnotesize}
$10\longrightarrow$
\end{footnotesize}&
\begin{footnotesize}
$[10]_3$\end{footnotesize}&\begin{footnotesize}
=
\end{footnotesize}&
\begin{footnotesize}
$101$ $\longrightarrow$
\end{footnotesize}&
\begin{footnotesize}
$0.101$
\end{footnotesize}& \begin{footnotesize}
$=$
\end{footnotesize}&
\begin{footnotesize}
$[\phi_3(10)]_3$
\end{footnotesize}&
\begin{footnotesize}
$\longrightarrow$
\end{footnotesize}&
\begin{footnotesize}
$\beta+\beta^{3}$
\end{footnotesize}\\

\end{tabular}
\end{center}
e cos\`{i} via.}
\end{eso}
\begin{oss}
\rm{La successione di punti $\{\xi_{2,1}^{n}\}_{n\in\mathbb{N}}$ pu\`{o} essere ottenuta dalla rappresentazione ternaria dei numeri naturali, invertendo l'ordine delle sue cifre ed eliminando le rappresentazioni che presentano due cifre consecutive del tipo 21, 22. Le potenze di $\beta$ possono essere calcolate come nel caso della successione di van der Corput, tenendo conto dei coefficienti moltiplicativi; precisamente, le cifre nella rappresentazione di $\phi_3(n)$ in base $3$ sono proprio i coefficienti di $\beta$.}
\end{oss}
\begin{eso}
\label{Esempio2.4.3}
\rm{Consideriamo come ultimo esempio la successione di punti $\{\xi_{1,2}^{n}\}_{n\in\mathbb{N}}$ e determiniamo la famiglia di funzioni per la successione di partizioni $\rho_{1,2}^{n}$, ottenuta mediante successivi $\rho$-raffinamenti della partizione banale $\omega$, con $\rho=\{[0,\beta[,[\beta,\beta+\beta^{2}[,[\beta+\beta^{2},1[\}$. Definiamo le seguenti funzioni:
\begin{eqnarray}
&\psi_0 : [0,1[\longrightarrow [0,\beta[ \\
& x \mapsto \beta x \nonumber \ ,
\end{eqnarray}
\begin{eqnarray}
&\psi_1 : [0,\beta[\longrightarrow [\beta ,\beta+\beta^{2}[ \\
& x \mapsto \beta x + \beta \nonumber
\end{eqnarray}
e
\begin{eqnarray}
&\psi_2 : [0,\beta[\longrightarrow [\beta+\beta^{2} ,1[ \\
& x \mapsto \beta x + \beta+\beta^{2} \ . \nonumber
\end{eqnarray}
Anche in questo caso analizziamo le composizioni possibili tra queste funzioni. Osservato che ben quattro delle nove possibili composizioni non sono lecite, e cio\`{e} $\psi_1\circ\psi_1$, $\psi_1\circ\psi_2$, $\psi_2\circ\psi_1$ e $\psi_2\circ\psi_2$, passiamo subito alla determinazione dei primi tre insiemi di punti relativi alle partizioni $\rho_{1,2}^{1}$, $\rho_{1,2}^{2}$, $\rho_{1,2}^{3}$.
\begin{equation*}
\Lambda_{1,2}^{1}=\Big(\xi_{1,2}^{1},\xi_{1,2}^{2},\xi_{1,2}^{3}\Big)=\Big(\psi_0(0),\psi_1(0),\psi_2(0)\Big)=\Big(0,\beta,\beta+\beta^{2}\Big) \ .
\end{equation*}
Componendo le funzioni in maniera lecita otteniamo
\begin{eqnarray*}
\Lambda_{1,2}^{2}&=&\Big(\xi_{1,2}^{1},\xi_{1,2}^{2},\xi_{1,2}^{3},\xi_{1,2}^{4},\xi_{1,2}^{5}\big) \\
&=&\Big(\psi_0(0),\psi_1(0),\psi_2(0), \psi_{01}(0),\psi_{02}(0)\Big)\\
&=& \Big(\psi_{00}(0), \psi_{10}(0),\psi_{20}(0), \psi_{01}(0),\psi_{02}(0)\Big)\\
&=& \Big(0,\beta,\beta+\beta^{2},\beta^{2},\beta^{2}+\beta^{3}\Big)\ .
\end{eqnarray*}
Iterando nuovamente questo procedimento ai punti ottenuti si ha
\begin{equation*}
\Lambda_{1,2}^{3} = \Big(\xi_{1,2}^{1},\xi_{1,2}^{2},\xi_{1,2}^{3},\xi_{1,2}^{4},\xi_{1,2}^{5}, \xi_{1,2}^{6},\xi_{1,2}^{7},\xi_{1,2}^{8},\xi_{1,2}^{9},\xi_{1,2}^{10}, \xi_{1,2}^{11}\Big)
\end{equation*}
\begin{eqnarray*}
&=&\Big(\psi_0(0),\psi_1(0),\psi_2(0), \psi_{01}(0),\psi_{02}(0),\psi_{001}(0),\\
& & \psi_{101}(0),\psi_{201}(0),\psi_{002}(0),\psi_{102}(0),\psi_{202}(0)\Big)\\
&=& \Big(\psi_{000}(0), \psi_{100}(0),\psi_{200}(0), \psi_{010}(0),\psi_{020}(0),\psi_{001}(0),\\
& & \psi_{101}(0),\psi_{201}(0), \psi_{002}(0),\psi_{102}(0),\psi_{202}(0)\Big)\\
&=& \big(0,\beta,\beta+\beta^{2},\beta^{2},\beta^{2}+\beta^{3},\beta^{3},\beta+\beta^{3},\\
& & \beta+\beta^{2}+\beta^{3}, \beta^{3}+\beta^{4}, \beta+\beta^{3}+\beta^{4}, \beta+\beta^{2}+\beta^{3}+\beta^{4}\Big)\ .
\end{eqnarray*}
Procedendo in questo modo si ottiene la successione di punti $\{\xi_{1,2}^{n}\}$. 
\\ Osserviamo che anche in questo caso si tratta di considerare la rappresentazione dei numeri naturali in base $3$. Vedremo che ci sono due cancellazioni in pi\`{u} da effettuare rispetto al caso $L=2$, $S=1$, relative alle composizioni non possibili, ma l'osservazione cruciale riguarder\`{a} i coefficienti di $\beta$.\\ Vediamo, a tal scopo, cosa succede partendo dalla rappresentazione in base $3$ dei primi numeri naturali, cos\`{i} come abbiamo fatto nei due esempi precedenti.
\begin{center}
\begin{tabular}{rlcrrlcrl} 
\begin{footnotesize}
$0\longrightarrow$
\end{footnotesize}&
\begin{footnotesize}
$[0]_3$
\end{footnotesize}&
\begin{footnotesize}
$=$
\end{footnotesize}&\begin{footnotesize}
$00 \longrightarrow$
\end{footnotesize}&\begin{footnotesize}
$0.00$
\end{footnotesize}&\begin{footnotesize}
$=$
\end{footnotesize}&
\begin{footnotesize}
$[\phi_3(0)]_3$
\end{footnotesize}&
\begin{footnotesize}
$\longrightarrow$
\end{footnotesize}&
\begin{footnotesize}
$0$
\end{footnotesize}\\

\begin{footnotesize}
$1\longrightarrow$
\end{footnotesize}&
\begin{footnotesize}
$[1]_3$\end{footnotesize}&\begin{footnotesize}
$=$
\end{footnotesize}&
\begin{footnotesize}
$01$ $\longrightarrow$
\end{footnotesize}&
\begin{footnotesize}
$0.10$
\end{footnotesize}& \begin{footnotesize}
$=$
\end{footnotesize}&
\begin{footnotesize}
$[\phi_3(1)]_3$
\end{footnotesize}&\begin{footnotesize}
$\longrightarrow$
\end{footnotesize}&
\begin{footnotesize}
$\beta$
\end{footnotesize}\\

\begin{footnotesize}
$2\longrightarrow$
\end{footnotesize}&
\begin{footnotesize}
$[2]_3$\end{footnotesize}&\begin{footnotesize}
$=$
\end{footnotesize}&
\begin{footnotesize}
$02$ $\longrightarrow$
\end{footnotesize}&
\begin{footnotesize}
$0.20$\end{footnotesize}&\begin{footnotesize}
$=$
\end{footnotesize}&
\begin{footnotesize}
$\phi_3(2)]_3$
\end{footnotesize}&
\begin{footnotesize}
$\longrightarrow$
\end{footnotesize}&
\begin{footnotesize}
$(1+\beta)\beta$
\end{footnotesize}\\

\begin{footnotesize}
$3\longrightarrow$
\end{footnotesize}&
\begin{footnotesize}
$[3]_3$\end{footnotesize}&\begin{footnotesize}
$=$
\end{footnotesize}&
\begin{footnotesize}
$10$ $\longrightarrow$
\end{footnotesize}&
\begin{footnotesize}
$0.01$
\end{footnotesize}& \begin{footnotesize}
$=$
\end{footnotesize}&
\begin{footnotesize}
$[\phi_3(3)]_3$
\end{footnotesize}&
\begin{footnotesize}
$\longrightarrow$
\end{footnotesize}&
\begin{footnotesize}
$\beta^{2}$
\end{footnotesize}\\

\begin{footnotesize}
$4\longrightarrow$
\end{footnotesize}&
\begin{footnotesize}
$[4]_3$\end{footnotesize}&\begin{footnotesize}
=
\end{footnotesize}&
\begin{footnotesize}
$11$ $\longrightarrow$
\end{footnotesize}&
\begin{footnotesize}
$0.11$
\end{footnotesize}& \begin{footnotesize}
$=$
\end{footnotesize}&
\begin{footnotesize}
$[\phi_3(4)]_3$
\end{footnotesize}&
\begin{footnotesize}
$\longrightarrow$
\end{footnotesize}&
\begin{footnotesize}
nessun $\beta$
\end{footnotesize}\\

\begin{footnotesize}
$5\longrightarrow$
\end{footnotesize}&
\begin{footnotesize}
$[5]_3$\end{footnotesize}&\begin{footnotesize}
=
\end{footnotesize}&
\begin{footnotesize}
$12$ $\longrightarrow$
\end{footnotesize}&
\begin{footnotesize}
$0.21$
\end{footnotesize}& \begin{footnotesize}
$=$
\end{footnotesize}&
\begin{footnotesize}
$[\phi_3(5)]_3$
\end{footnotesize}&
\begin{footnotesize}
$\longrightarrow$
\end{footnotesize}&
\begin{footnotesize}
nessun $\beta$
\end{footnotesize}\\

\begin{footnotesize}
$6\longrightarrow$
\end{footnotesize}&
\begin{footnotesize}
$[6]_3$\end{footnotesize}&\begin{footnotesize}
=
\end{footnotesize}&
\begin{footnotesize}
$20$ $\longrightarrow$
\end{footnotesize}&
\begin{footnotesize}
$0.02$
\end{footnotesize}& \begin{footnotesize}
$=$
\end{footnotesize}&
\begin{footnotesize}
$[\phi_3(6)]_3$
\end{footnotesize}&
\begin{footnotesize}
$\longrightarrow$
\end{footnotesize}&
\begin{footnotesize}
$(1+\beta)\beta^{2}$
\end{footnotesize}\\

\begin{footnotesize}
$7\longrightarrow$
\end{footnotesize}&
\begin{footnotesize}
$[7]_3$\end{footnotesize}&\begin{footnotesize}
=
\end{footnotesize}&
\begin{footnotesize}
$21$ $\longrightarrow$
\end{footnotesize}&
\begin{footnotesize}
$0.12$
\end{footnotesize}& \begin{footnotesize}
$=$
\end{footnotesize}&
\begin{footnotesize}
$[\phi_3(7)]_3$
\end{footnotesize}&
\begin{footnotesize}
$\longrightarrow$
\end{footnotesize}&
\begin{footnotesize}
nessun $\beta$
\end{footnotesize}\\

\begin{footnotesize}
$8\longrightarrow$
\end{footnotesize}&
\begin{footnotesize}
$[8]_3$\end{footnotesize}&\begin{footnotesize}
=
\end{footnotesize}&
\begin{footnotesize}
$22$ $\longrightarrow$
\end{footnotesize}&
\begin{footnotesize}
$0.22$
\end{footnotesize}& \begin{footnotesize}
$=$
\end{footnotesize}&\begin{footnotesize}
$[\phi_3(8)]_3$
\end{footnotesize}&
\begin{footnotesize}
$\longrightarrow$
\end{footnotesize}&
\begin{footnotesize}
nessun $\beta$
\end{footnotesize}\\

\begin{footnotesize}
$9\longrightarrow$
\end{footnotesize}&
\begin{footnotesize}
$[9]_3$\end{footnotesize}&\begin{footnotesize}
=
\end{footnotesize}&
\begin{footnotesize}
$100$ $\longrightarrow$
\end{footnotesize}&
\begin{footnotesize}
$0.001$
\end{footnotesize}& \begin{footnotesize}
$=$
\end{footnotesize}&
\begin{footnotesize}
$[\phi_3(9)]_3$
\end{footnotesize}&
\begin{footnotesize}
$\longrightarrow$
\end{footnotesize}&
\begin{footnotesize}
$\beta^{3}$
\end{footnotesize}\\

\begin{footnotesize}
$10\longrightarrow$
\end{footnotesize}&
\begin{footnotesize}
$[10]_3$\end{footnotesize}&\begin{footnotesize}
=
\end{footnotesize}&
\begin{footnotesize}
$101$ $\longrightarrow$
\end{footnotesize}&
\begin{footnotesize}
$0.101$
\end{footnotesize}& \begin{footnotesize}
$=$
\end{footnotesize}&
\begin{footnotesize}
$[\phi_3(10)]_3$
\end{footnotesize}&
\begin{footnotesize}
$\longrightarrow$
\end{footnotesize}&
\begin{footnotesize}
$\beta+\beta^{3}$
\end{footnotesize}\\

\end{tabular}
\end{center}
e cos\`{i} via.}
\end{eso}
\begin{oss}
\rm{Come nei casi precedenti, anche in quest'ultimo la successione di punti $\{\xi_{1,2}^{n}\}_{n\in\mathbb{N}}$ pu\`{o} essere ottenuta dalla rappresentazione in base $3$ dei numeri naturali, invertendo l'ordine delle sue cifre ed eliminando le rappresentazioni corrispondenti a composizioni non lecite. Ci\`{o} che rende questo esempio particolarmente interessante \`{e} l'espressione dei coefficienti delle potenze di $\beta$. Infatti, mentre per gli esempi corrispondenti alle successioni di punti $\{\xi_{1,1}^{n}\}_{n\in\mathbb{N}}$ e $\{\xi_{2,1}^{n}\}_{n\in\mathbb{N}}$, rispettivamente, questi erano proprio le cifre di $\phi_2(n)$ e $\phi_3(n)$, rispettivamente, in base $2$ e $3$, rispettivamente, in questo caso ci\`{o} \`{e} falso per la cifra $2$. Al posto di $2$ bisogna sostituire l'espressione $1+\beta$.}
\end{oss}
\paragraph*{}
Siamo pronti per considerare il caso generale per $L, S\in \mathbb{N}$. 
\subsection*{Il caso generale $\{\xi_{L,S}^{n}\}_{n\in\mathbb{N}}$}
Supponiamo di voler suddividere l'intervallo unitario in $L$ intervalli lunghi di ampiezza $\beta$ ed $S$ intervalli corti di ampiezza $\beta^{2}$, con $L\beta+S\beta^{2}=1$. Definiamo le seguenti funzioni
\begin{eqnarray}
&\psi_j : [0,1[\longrightarrow [j\beta,(j+1)\beta[ \label{eq:2.15}\\
& x \mapsto \beta x + j\beta\nonumber
\end{eqnarray}
con $j=0,\ldots, L-1$ e
\begin{eqnarray}
&\psi_{L,k} : [0,\beta[\longrightarrow [L\beta+k\beta^{2} ,L\beta+(k+1)\beta^{2}[ \label{eq:2.16}\\
& x \mapsto \beta x +L\beta+k\beta^{2} \ , \nonumber
\end{eqnarray}
con $k=0,\ldots, S-1$. \\
Si vede immediatamente che gli insiemi di punti $\Lambda_{L,S}^{n}$ sono ottenuti dalla composizione di queste funzioni, ma ovviamente non tutte le composizioni sono possibili. Precisamente, le composizioni non ammissibili sono quelle del tipo $\psi_{L,k}\circ\psi_j$ per ogni valore di $k\in\{0,\ldots,S-1\}$ e per ogni valore di $j\in\{1,\ldots,L-1\}$ e $\psi_{L,k_i}\circ\psi_{L,k_s}$ per ogni $k_i, k_s$. \\Allora, valutando queste due famiglie di funzioni nello zero, troviamo l'insieme 
\begin{equation*}
\Lambda_{L,S}^{1}=\{\psi_0(0),\psi_1(0),\ldots, \psi_{L-1}(0),\psi_{L,0}(0),\psi_{L,1}(0)\ldots,\psi_{L,S-1}(0) \}
\end{equation*}
dei primi $L+S$ punti di $\{\xi_{L,S}^{n}\}_{n\in\mathbb{N}}$. Componendo in maniera lecita le funzioni, si ottengono tutti gli altri insiemi, come abbiamo gi\`{a} visto nei tre esempi precedenti nei casi particolari delle $LS$-successioni di punti, con $L=S=1$, $L=2, S=1$ e $L=1, S=2$. \smallskip\\ Anche nel caso generale, ovviamente, esiste uno stretto legame tra queste funzioni e le potenze di $\beta$ (e la rappresentazione dei numeri in base $L+S$). Illustriamo qui di seguito la natura di tale legame.\\ Innanzitutto, per rappresentare un numero in base $L+S$ abbiamo bisogno delle cifre da $0$ a $L+S-1$, ossia dei resti nella divisione di un generico numero naturale per $L+S$. Poi, dopo aver espresso ciascun numero naturale in base $L+S$, si inverte tale rappresentazione (ottenendo, cos\`{i}, la rappresentazione in base $L+S$ della funzione radice inversa calcolata in tale numero) e si eliminano le sequenze di cifre che contengono al loro interno gruppi di cifre consecutive corrispondenti a composizioni non lecite delle funzioni (\ref{eq:2.15}) e (\ref{eq:2.16}). Come negli esempi precedenti, queste cifre rappresentano (nella maniera appresso specificata) i ``coefficienti'' delle combinazioni lineari delle potenze di $\beta$ che ci danno i punti della successione. Precisamente, le cifre corrispondenti a valori compresi tra $0$ ed $L$ sono proprio i coefficienti di $\beta$, mentre per quelle comprese tra $L+1$ ed $L+S-1$ vale la seguente relazione:
\begin{equation}
\label{funzioni_$LS$}
k \mapsto L+k\beta 
\end{equation}
con $k=1,\ldots,S-1$, e cio\`{e} alla cifra $L+k$ si sostituisce $L+k\beta$. \\ Chiaramente, per $L=S=1$, $L=2$ e $S=1$ e $L=1$ e $S=2$ otteniamo esattamente i risultati presentati gli esempi \ref{Esempio2.4.1}, \ref{Esempio2.4.2} e \ref{Esempio2.4.3}.
\begin{oss}
\rm{Sia la Definizione \ref{def:2.3.1} che il metodo che abbiamo appena esposto nel caso generale utilizzano ciascuna una propria famiglia di funzioni da applicare in maniera ricorsiva ai punti degli insiemi precedentemente ordinati $\Lambda_{L,S}^{n-1}$ per la determinazione degli insiemi di punti $\Lambda_{L,S}^{n}$. Tuttavia, sebbene a prima vista le due definizioni possano sembrare sostanzialmente equiparabili, il secondo metodo che usa le funzioni (\ref{eq:2.15}) e (\ref{eq:2.16}) risulta particolarmente efficiente per via di due considerazioni che sono la naturale generalizzazione di quelle esposte nelle Osservazioni \ref{Oss:2.4.2}. La prima riguarda la natura delle funzioni (\ref{eq:2.15}) e (\ref{eq:2.16}), che in questo caso vengono definite all'inizio e quindi non pi\`{u} ridefinite per ogni singolo blocco, a differenza delle funzioni (\ref{funzioni_phi}), che dipendono da $n$. La seconda considerazione \`{e} forse pi\`{u} importante della prima e dipende dal legame stretto che intercorre tra queste funzioni e la rappresentazione in base $L+S$, come abbiamo evidenziato qui sopra: infatti, gli insiemi di punti $\Lambda_{L,S}^{n}$ vengono determinati indipendentemente dall'insieme $\Lambda_{L,S}^{n-1}$ in quanto questo calcolo viene tradotto in termini di rappresentazione dei numeri naturali rispetto alla base $L+S$. Questo fa s\`{i} che se vogliamo calcolare i primi $N$ punti di $\{\xi_{L,S}^{n}\}_{n\in\mathbb{N}}$, con $t_{n-1}<N<t_{n}$, possiamo farlo senza dover necessariamente calcolare tutti i punti di $\Lambda_{L,S}^{n-1}$, perch\`{e} non si tratta pi\`{u} di applicare le funzioni (\ref{funzioni_phi}) all'insieme $\Lambda_{L,S}^{n-1}$, ma di scrivere i primi $N$ punti nella base $L+S$, ovviamente al netto delle opportune cancellazioni, cos\`{i} come spiegato in precedenza.} 
\end{oss}
Ora che abbiamo schematizzato il nostro ragionamento, possiamo passare alla presentazione dell'algoritmo per calcolare le successioni di punti $\{\xi_{L,S}^{n}\}$.
\section{Algoritmo}
Concludiamo questo capitolo con la descrizione dell'algoritmo veloce, discusso nel paragrafo precedente. Esso sar\`{a} scritto utilizzando uno \emph{pseudocodice}, ossia mediante un linguaggio di programmazione fittizio, non direttamente compilabile, allo scopo di renderlo pi\`{u} leggibile e comprensibile. \\ In seguito, per determinare le successioni di punti $\{\xi_{L,S}^{n}\}_{n\in\mathbb{N}}$, si \`{e} scelto di tradurre l'algoritmo in un linguaggio implementabile dal software Matlab. 
\paragraph*{}
Procediamo alla descrizione dell'algoritmo in tre passi, la cui parte principale sar\`{a} costituita dal corpo centrale.\smallskip\\
\textbf{1.} \\
\begin{algorithm}[H]
\TitleOfAlgo{base\_inversa}
\DontPrintSemicolon
\KwIn{$(n\in \mathbb{N}, b=L+S)$}
\KwOut{$[n]_{\frac{1}{b}}$}
\KwData{$int\ j=\lfloor \log(n)/log(b)\rfloor +1 $ \;
 $int\ a[j]=0$\;
 \Begin{
 \For{$i = 1 \to j$}{ 
 $a[i] \gets mod(n,b)$\;
 $n \gets \lfloor \frac{n}{b} \rfloor$
 }
 }
 }
\KwResult{$a$}
\smallskip
\end{algorithm}

Questa seconda sequenza di comandi, alquanto complessa visto il numero di cicli innestati, richiama anche l'algoritmo precedente, che abbiamo chiamato funzione base\_inversa e rappresenta la prima parte vera e propria del nostro algoritmo, in cui vengono effettuate le cancellazioni dei vettori $a$ che presentano cifre consecutive corrispondenti alle composizioni non ammissibili. Il risultato di questo secondo algoritmo sar\`{a} una matrice $A$, le cui righe sono i vettori $a$ che restano dopo aver effettuato le cancellazioni (che corrispondono ai coefficienti delle potenze di $\beta$ nella maniera che vedremo nel terzo algoritmo). \\ \`{E} dunque sufficiente fornire in input solo i tre valori di $n, L, S$. \\Notiamo innanzitutto che, poich\`{e} non tutte le rappresentazioni dei numeri naturali in base $L+S$ corrispondono a punti della successione, non sappiamo a priori a quale sequenza di cifre corrisponder\`{a} l'$n$-esima riga della matrice (che non pu\`{o}, evidentemente, corrispondere al numero naturale $n$). A causa di questa incertezza sul numero di righe, \`{e} dunque necessario inizializzare un vettore riga $A$. Per la stessa ragione abbiamo poi scelto che questo vettore riga abbia dimensione $100$, anche se questo numero potrebbe essere ridotto, dal momento che con $100$ cifre possiamo esprimere un numero estremamente grande in qualsiasi base. Inoltre, poich\`{e} il nostro scopo \`{e} quello di ottenere la matrice $A$, aggiungiamo passo passo altre righe al vettore $A$, le cui componenti sono tutte nulle per rappresentare lo zero. Queste righe successive saranno identificate con un vettore ausiliario, che indichiamo con $c$. Ci\`{o} che faremo sar\`{a} controllare ad ogni passo il numero di righe della matrice $A$, che denotiamo con $r$. Fintanto che il numero di righe \`{e} minore del numero fissato $n$, prenderemo in esame la rappresentazione del numero successivo, indicato con $i$. A questo punto bisogna controllare se questo numero ha una rappresentazione che contiene delle cifre corrispondenti a composizioni di funzioni ammissibili oppure no, e per farlo utilizziamo una variabile ausiliaria $d$, il cui valore sar\`{a} $0$ o $1$ nel caso in cui il numero debba essere inserito nella matrice o no, rispettivamente. La lunghezza del vettore $a$ \`{e} variabile e dipende dal numero che stiamo rappresentando. Ogni volta che rappresentiamo un numero naturale, otteniamo $a$ con l'algoritmo 1, immagazziniamo le sue componenti ed effettuiamo un controllo su esse. Se due componenti consecutive di $a$ coincidono con le composizioni non lecite, poniamo $d=d+1$, altrimenti non si fa nulla. Nel primo caso, incrementiamo semplicemente il valore di $i$, mentre nel secondo, in cui il valore di $d$ \`{e} ancora impostato sullo zero, memorizziamo nelle prime cifre di $c$ (numero pari alla lunghezza di $a$) gli stessi valori contenuti in $a$ e controlliamo se sono presenti valori maggiori di $L$. In tal caso riproponiamo il ragionamento fatto nel Paragrafo \ref{funzioni_$LS$} nel cui caso sostituiamo a ciascuna cifra maggiore di $L$ l'espressione (\ref{funzioni_$LS$}). Infine incrementiamo il valore dell'intero positivo successivo $i$ e aggiungiamo la riga $c$ alla matrice $A$. Ora, il numero di righe $r$ della matrice viene calcolato nuovamente e si ripetono cos\`{i} tutti i cicli. Alla fine, quando il numero di righe della matrice \`{e} $n$, si calcola il numero di cifre $e$ necessario per rappresentare l'ultimo elemento $i$, in modo che la matrice finale $A$ sia ripulita dalle colonne superflue, ossia $A$ sar\`{a} la matrice costituita dalle prime $r$ righe ed $e$ colonne della matrice di partenza. \smallskip \\
\textbf{3.} \\
\begin{algorithm}[H]
\TitleOfAlgo{Algoritmo}
\DontPrintSemicolon
\KwIn{$(L,S,n\in \mathbb{N})$}
\KwOut{$\{\xi_{L,S}^{n}\}$}
\KwData{$A=cancellazioni(L,S,n)$ \;
 $\beta=\frac{-L+\sqrt{L^{2}+4S}}{2S}$}
 \Begin{
 $[\cdot,c]= size(A)$\;
 \For{$k = 1 \to c$}{ 
 $s(k)=\beta^{k}$
 }
 $B=As'$
 }
\KwResult{$B$}
\smallskip
\end{algorithm}

Infine quest'ultimo algoritmo restituisce esattamente il numero di punti da noi indicato, ossia, fissato $n\in\mathbb{N}$, esso restituir\`{a} proprio i primi $n$ termini della successione di punti $\{\xi_{L,S}^{n}\}_{n\in\mathbb{N}}$.
Esso prende in input la matrice $A$, calcola il numero delle sue colonne e lo indica con $c$. Poi costruisce un vettore $s$ di $c$ componenti, ciascuna delle quali \`{e} una potenza crescente di $\beta$. Quindi il nostro insieme di punti $(\xi_{L,S}^{1}, \xi_{L,S}^{2},\ldots, \xi_{L,S}^{n})$ sar\`{a} il prodotto della matrice $A$ con il vettore $s$.\medskip \\

Poich\`{e}, come abbiamo gi\`{a} detto, il precedente algoritmo \`{e} stato implementato con Matlab (che \`{e} un software numerico) la risposta che otteniamo \`{e} il valore numerico delle potenze di $\beta$ corrispondenti. \`{E} ovviamente possibile anche una risposta simbolica da parte del programma, ma abbiamo scelto la prima per due ragioni: una puramente tecnica riguarda i grafici, che non sarebbero stati possibili attraverso un parametro simbolico, l'altra riguarda la velocit\`{a} di risposta del computer, che impiega molto pi\`{u} tempo per elaborare una risposta simbolica. \\ Riportiamo le risposte del programma nei tre Esempi \ref{es:2.3.2} considerati in precedenza.
\begin{enumerate}
\item I primi $8$ punti della successione $\{\xi_{1,1}^{n}\}_{n\in\mathbb{N}}$, corrispondenti a $\Lambda_{1,1}^{4}$ sono
\begin{equation*}
0, 0.6180, 0.3820, 0.2361, 0.8541, 0.1459, 0.7639, 0.5279\ .
\end{equation*}
\item I primi $17$ punti della successione $\{\xi_{2,1}^{n}\}_{n\in\mathbb{N}}$, corrispondenti a $\Lambda_{2,1}^{3}$ sono
\begin{equation*}
0,  0.4142,  0.8284,  0.1716,  0.5858, 0.3431,  0.7574,  0.0711,  0.4853,
\end{equation*}
\begin{equation*}
0.8995,  0.2426,  0.6569,  0.1421,  0.5563,  0.9706,  0.3137,  0.7279 \ .
\end{equation*}
\item Infine, i primi $11$ punti della successione $\{\xi_{1,2}^{n}\}_{n\in\mathbb{N}}$, corrispondenti a $\Lambda_{1,2}^{3}$ sono
\begin{equation*}
0,  0.5000,  0.7500,  0.2500,  0.3750,  0.1250, 
\end{equation*}
\begin{equation*}
0.6250,  0.8750,  0.1875,  0.6875,  0.9375 \ .
\end{equation*}
\end{enumerate}  

\chapter{$LS$-successioni di punti in $[0,1[\times [0,1[$}
\graphicspath{{Capitolo3/}}
In questo terzo ed ultimo capitolo estenderemo la definizione di $LS$-successioni di punti e l'algoritmo presentato nel capitolo precedente al quadrato di lato unitario $I^{2} = [0, 1[x[0, 1[$. Quest'ultimo ci sar\`{a} fondamentale per la visualizzazione, mediante dei grafici, di alcuni dei risultati ottenuti. \\Per far ci\`{o}, prenderemo in esame le successioni introdotte da van der Corput in $I^{2}$ \cite{vanderCorput} e da Halton in $I^{s}$ \cite{Halton}, restringendo la nostra attenzione al caso bidimensionale.
\section{Successioni di van der Corput e di Halton}
Richiamiamo preliminarmente alcuni concetti esposti nei capitoli precedenti. Ricordiamo che se $b$ \`{e} un intero positivo, allora per ogni numero naturale $n=a_0+a_1b+\ldots+a_{M-1}b^{M-1}+a_Mb^{M}$ l'espressione
\begin{equation*}
[n]_b=a_Ma_{M-1}\ldots a_2a_1a_0
\end{equation*}
indica la rappresentazione di $n$ in base $b$, dove $M=\lfloor \log_b n \rfloor$ \`{e} il numero di cifre necessarie per esprimerla. \\Come abbiamo gi\`{a} avuto modo di osservare nella Definizione \ref{radinv}, invertendo l'ordine delle cifre $a_i$, con $i\in \{0,\ldots,M\}$, si ottiene il numero compreso tra 0 e 1 detto radice inversa di $n$, e cio\`{e}
\begin{equation}
\phi_b(n)= a_0b^{-1}+a_1b^{-2}+\ldots+a_Mb^{-M-1}\ ,\label{eq:3.1}
\end{equation}
la cui rappresentazione in base $b$ \`{e} $[\phi_b(n)]_b=0.a_0a_1\ldots a_{M-1}a_M$.\smallskip\\ Siamo adesso pronti per ricordare la definizione di alcune importanti successioni che fanno uso di tale notazione e che in qualche modo generalizzano in spazi multidimensionali la successione di van der Corput gi\`{a} presentata nella Definizione \ref{van der Corput}. \\ La prima interessante generalizzazione in dimensione $2$ risale allo stesso van der Corput che in \cite{vanderCorput} ha introdotto la seguente successione finita.
\begin{defn}[Successione di van der Corput]
\rm{La \emph{successione di van der Corput} di ordine $N\in\mathbb{N}$ nel quadrato $I^{2}$ \`{e} definita come}
\begin{equation}
\mathbf{x}_n=\left(\frac{n}{N},\phi_2(n)\right)\qquad \text{con\ } n=1,\ldots, N\ . \label{eq:3.2}
\end{equation}
\end{defn}
L'immagine successiva riporta la rappresentazione grafica della successione $\left(\frac{n}{5000},\phi_2(n)\right)$ per $n=1,\ldots, 5000$.\\
\begin{figure}[h!]
\centering
\setlength{\unitlength}{1cm}
\includegraphics[scale=0.88, keepaspectratio]{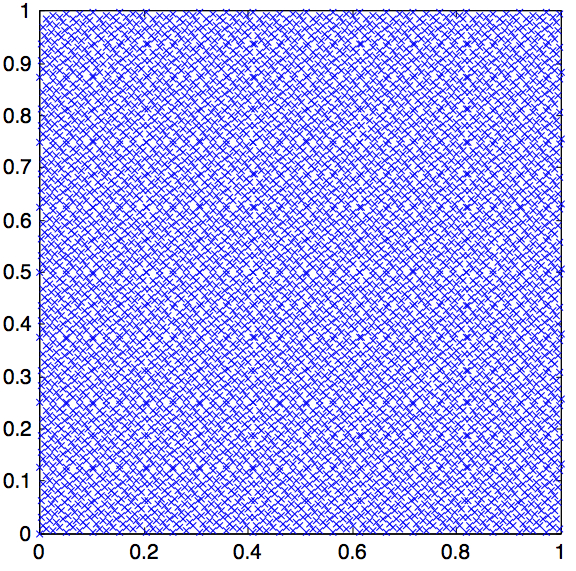}
\caption{Successione di van der Corput $\left(\frac{n}{5000},\phi_2(n)\right)$ per $1\leq n\leq 5000$}\label{fig:unif.-vdc}
\end{figure}\\
In seguito, fu il matematico inglese J. M.Hammersley che in \cite{Hammersley} propose una versione $s$-dimensionale della definizione precedente. Egli, infatti, prende i primi $s-1$ numeri primi, che indica con $b_1, b_2,\ldots, b_{s-1}$. Considera poi la rappresentazione dei primi $N$ numeri naturali rispetto a queste basi e il processo di riflessione di tali espressioni in modo da ottenere frazioni nell'intervallo $[0,1]$ del tipo (\ref{eq:3.1}). \\Si ha dunque la seguente
\begin{defn}[Successione di Hammersley]
\rm{La \emph{successione di Hammersley} di ordine $N\in \mathbb{N}$ nell'ipercubo $I^{s}$ \`{e} la successione $s$-dimensionale definita come}
\begin{equation}
\mathbf{x}_n=\left(\frac{n}{N},\phi_{b_1}(n),\phi_{b_2}(n),\ldots,\phi_{b_{s-1}}(n)\right)\qquad \text{con}\ n=1,\ldots,N \ , \label{hammersley}
\end{equation}
\rm{dove $b_1,\ldots,b_{s-1}$ sono i primi $s-1$ numeri primi e le $\phi_{b_j}$ sono le funzioni radice inversa che compaiono nella Definizione \ref{radinv}.}
\end{defn}
Osserviamo che la successione di Hammersley \`{e} un insieme di $N$ punti nell'ipercubo $I^{s}$ e che quindi tale insieme non pu\`{o} essere esteso ad una successione infinita: infatti, fissato il numero $N$, non si possono aggiungere ulteriori punti a questa sequenza, a meno di non aumentare il valore di  $N$, ma ci\`{o} comporta che tutti i termini vengano ricalcolati di conseguenza. \medskip\\ La successione di van der Corput fu generalizzata in un'altra direzione da Halton nel modo seguente.
\begin{defn}[Successione di Halton]
\rm{Definiamo \emph{successione di Halton} $s$-dimensionale $(\mathbf{x})_{n\in\mathbb{N}}$ la successione definita sull'ipercubo $I^{s}$ come}
\begin{equation}
\mathbf{x}_n=\left(\phi_{b_1}(n),\phi_{b_2}(n),\ldots,\phi_{b_s}(n)\right)\qquad \text{per\ ogni}\ n\in \mathbb{N}\ , \label{eq:3.3}
\end{equation}
\rm{dove $b_1,b_2,\ldots,b_s$ sono interi positivi a due a due coprimi.}
\end{defn}
\`{E} immediato osservare che per la successione di Halton la dipendenza da $N$ \`{e} stata eliminata. 
Nella \figurename~\ref{fig:vdc23} riportiamo il grafico dei primi 5000 punti della successione di Halton bidimensionale, relativi ai due numeri coprimi $b_1=2$ e $b_2=3$.\\
\begin{figure}[h!]
\centering
\setlength{\unitlength}{1cm}
\includegraphics[scale=0.88, keepaspectratio]{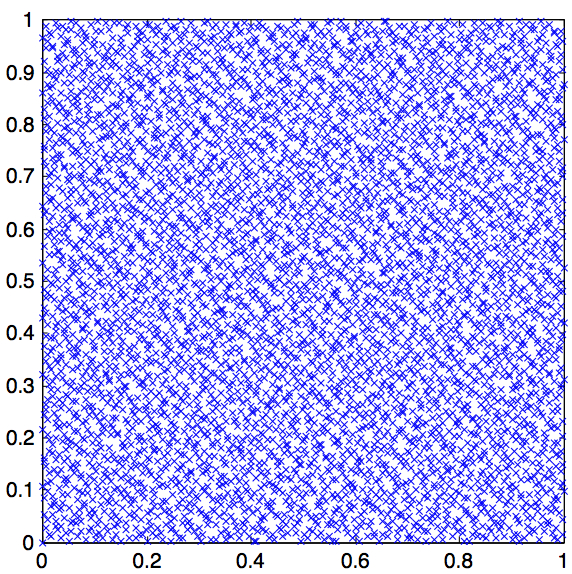}
\caption{Successione di Halton $\{(\phi_2(n),\phi_3(n))\}_{n\in\mathbb{N}}$ per $1\leq n\leq 5000$}\label{fig:vdc23}
\end{figure}\\
Ricordiamo che le successioni di punti in $I^{s}$ per cui la discrepanza risulta essere uguale a $\mathcal{O}(\log^{s} N/N)$ per $N\to \infty$ sono dette successioni a bassa discrepanza.\\ 
Lo stesso Halton in \cite{Halton} ha dimostrato che la (\ref{eq:3.3}) \`{e} una successione a bassa discrepanza, cio\`{e} vale il seguente risultato
\begin{prop}
La discrepanza della successione di Halton $s$-dimensio-\\nale $(\textbf{x}_n)_{n\in\mathbb{N}}$ soddisfa $$D_N(\textbf{x}_1,\ldots, \textbf{x}_N)\leq\frac{\log^{s} N}{N}$$
per ogni $n\in\mathbb{N}$.
\end{prop}
Si pu\`{o} infatti osservare che i punti della successione $\left\{\big(\varphi_2(n),\varphi_3(n)\big)\right\}_{n\in\mathbb{N}}$, illustrati nella \figurename~\ref{fig:vdc23}, si distribuiscono in modo da ricoprire il quadrato in maniera uniforme, come ci si aspetta da una successione a bassa discrepanza.
\paragraph*{}
Anche le successioni di van der Corput e Hammersley sono esempi di successioni a bassa discrepanza, ed entrambi gli autori presentarono risultati importanti riguardanti la stima della discrepanza.\\ Le successioni a bassa discrepanza rivestono un ruolo molto importante nelle applicazioni perch\`{e} si tratta di successioni utilizzate nei metodi Quasi-Monte Carlo per il calcolo approssimato di integrali. \smallskip\\
Sappiamo che esistono vari metodi di integrazione numerica per il calcolo di integrali di funzioni di cui non riusciamo a determinare analiticamente la primitiva. Questi metodi approssimano il valore dell'integrale della funzione $f$ con una combinazione lineare di valori assunti dalla $f$, ad esempio, nel caso unidimensionale
\begin{equation*}
\int _I f(x) dx \approx \sum_{n=1}^{N}w_i f(x_i)\ .
\end{equation*}
Ovviamente differenti formule di quadratura corrispondono a differenti scelte dei punti $x_i$ e dei pesi $w_i$.
Il metodo di integrazione Monte Carlo si basa sull'utilizzo di numeri $x_i$ casuali, con pesi $w_i=\frac{1}{N}$ per $i=1,\ldots, N$. \\ Se, ad esempio, vogliamo determinare l'integrale di una funzione $f$ sul quadrato $I^{2}$, consideriamo i vettori casuali $\textbf{x}_1,\textbf{x}_2,\ldots,\textbf{x}_N$ e approssimiamo il valore dell'integrale con la media dei valori assunti dalla funzione in detti punti, ossia
\begin{equation*}
\int_{I^{2}}f(\textbf{x})dx\approx \frac{1}{N}\sum_{n=1}^{N}f(\textbf{x}_n)\ .
\end{equation*}
La caratteristica peculiare di tale metodo consiste nel fatto che l'errore assoluto $$\left| \int_I f(x) dx - \frac{1}{N}\sum_{n=1}^{N}f(x_i) \right|\ ,$$ ossia l'errore commesso nella stima dell'integrale attraverso la media dei valori della $f$ negli $N$ punti $x_i$, decresce in media come $\frac{1}{\sqrt{N}}$, in virt\`{u} del Teorema del Limite Centrale, indipendentemente dalla dimensione dello spazio.\\ Il metodo Monte Carlo pu\`{o} dunque sembrare alquanto insoddisfacente se confrontato con altri metodi di quadratura noti per l'approssimazione di integrali di funzioni in una variabile come il metodo dei trapezi, di Cavalieri-Simpson o di Gauss, che convergono molto pi\`{u} rapidamente, nel senso che l'errore decresce pi\`{u} velocemente. \\Quando per\`{o} si passa alla stima di integrali multidimensionali la situazione cambia notevolmente, perch\`{e} mentre l'errore con la stima fornita dal metodo Monte Carlo risulta sempre dello stesso ordine, con gli altri metodi l'ordine di convergenza decresce al crescere della dimensione dello spazio.\\
Poich\`{e} la generazione di numeri o vettori casuali non \`{e} agevole, specialmente se $N$ \`{e} grande, a partire da von Neumann, e quindi dai primi calcolatori, \`{e} stata sviluppata l'idea di considerare successioni deterministiche che simulino la casualit\`{a}. Nei vari linguaggi di programmazione c'\`{e} di solito un comando denominato $\emph{rand}$, o simili, che produce successioni di numeri di $[0,1[$ che vari test statistici non distinguono da successioni casuali. Tali successioni si chiamano \emph{pseudo-casuali}. \\ Successivamente \`{e} nata l'idea di rinunciare ad imitare la casualit\`{a} delle successioni e di considerare invece delle successioni in un certo senso ottimali per l'uso che si ha in mente, cio\`{e} per il calcolo degli integrali. Tali successioni si chiamano \emph{quasi-casuali} (anche se il nome pu\`{o} ingannare, perch\`{e} non hanno nulla di casuale), che poi sono le nostre successioni a bassa discrepanza. Queste successioni danno luogo al metodo Quasi-Monte Carlo.\\ Questa scelta pone rimedio a due aspetti: il primo riguarda la difficolt\`{a} di scegliere dei punti in maniera del tutto casuale, il secondo invece riguarda l'accuratezza nella convergenza dei metodi. Infatti sappiamo che una successione di punti in $I$ \`{e} a bassa discrepanza se essa \`{e} asintoticamente equivalente a $\log N/N$ e quindi l'utilizzo di questo secondo metodo garantisce che la stima dell'integrale converger\`{a} al suo valore reale come $\frac{\log N}{N}$, ossia pi\`{u} velocemente di $\frac{1}{\sqrt{N}}$.
\paragraph{}
Nei due articoli di Hammersley e Halton viene presa in esame la stima dell'integrale di una funzione $f$ continua su $I^{s}$, e cio\`{e} di
\begin{equation}
\int_0^{1}dx_1\ldots \int_0^{1}dx_s f(x_1,\ldots,x_s) \ . \label{*}
\end{equation}
Hammersley in \cite{Hammersley} afferma che, sebbene nei metodi numerici per il calcolo approssimato di integrali si utilizzi, come abbiamo gi\`{a} visto, la formula $\sum_i w_if(x_i)$, in assenza di altre informazioni sulla $f$, non possiamo far altro che considerare i pesi $w_i$ tutti uguali. Poich\`{e} per ipotesi la $f$ \`{e} necessariamente integrabile, essa pu\`{o} essere approssimata da una successione di funzioni semplici integrabili. Inoltre, poich\`{e} gli insiemi su cui la $f$ \`{e} costante si possono costruire mediante unione ed intersezione di intervalli, \`{e} ragionevole testare l'efficienza del metodo di approssimazione dell'integrale (\ref{*}) supponendo che la $f$ sia la funzione caratteristica di un sottointervallo $J^{*}=[\textbf{0}, \textbf{b}[$ contenuto nell'ipercubo unitario $I^{s}$. 
\\ Van der Corput, nel costruire la successione finita (\ref{eq:3.2}) sul quadrato, ha inoltre dimostrato che esiste una costante $C$ tale che, posto $\textbf{x}_n=\left(\frac{n}{N}, \varphi_2(n) \right)$ per ogni $1\leq n \leq N$, si ha
\begin{equation*}
\sup_{b_1,b_2\in [0,1[} \left|\sum_{n=1}^{N}\chi_{[\textbf{0}, \textbf{b}[}(\textbf{x}_n)-N(b_1b_2)\right| < C \ln N \ ,
\end{equation*}
dove $\textbf{b}=(b_1,b_2)$.\\ Hammersley, dopo aver introdotto la successione finita (\ref{hammersley}), si chiede in \cite{Hammersley} se \`{e} vero che 
\begin{equation*}
\sup_{J^{*}\subset I^{s}} \left|\sum_{n=1}^{N}\chi_{J^{*}}(\textbf{x}_n)-N\lambda(J^{*})\right|\smallskip
\end{equation*}
sia maggiorato da un multiplo di una qualche potenza di $\log N$.\\
Halton dimostra in \cite{Halton} che per la successione di Hammersley $(\textbf{x}_n)_{1\leq n\leq N}$ definita da (\ref{hammersley}) vale la seguente stima:
\begin{equation*}
\sup_{J^{*}\subset I^{s}}\left|\sum_{n=1}^{N}\chi_{J^{*}}(\textbf{x}_n)-N\lambda(J^{*})\right| < C_s (\ln N)^{s-1}\ ,\smallskip
\end{equation*}
per ogni $n\in\mathbb{N}$, con $C_s$ esplicita costante positiva.
\section{$LS$-successioni ``alla van der Corput''}
In questo capitolo presentiamo una prima estensione al caso bidimensionale delle $LS$-successioni di punti. Vogliamo presentare una loro variante in maniera simile a quanto fatto da van der Corput nel passaggio dalla successione unidimensionale $\{\phi_2(n)\}_{n\in\mathbb{N}}$ alla successione finita $\left(\frac{n}{N},\phi_2(n)\right)_{1\leq n\leq N}$ nel quadrato.
\begin{defn}
\label{Def:3.2.1}
\rm{Una \emph{$LS$-successione ``alla van der Corput''} di ordine $N\in\mathbb{N}$ nel quadrato $I^{2}$ \`{e} definita come
\begin{equation}
\textbf{x}_n=\left(\frac{n}{N}, \xi_{L,S}^{n}\right)\qquad \text{con}\ n=1,\ldots, N\ , \label{eq:3.5}
\end{equation}
dove $\{\xi_{L,S}^{n}\}_{n\in\mathbb{N}}$ \`{e} una $LS$-successione.}
\end{defn}
Ricordiamo brevemente la definizione di discrepanza per una successione in $I^{2}$, gi\`{a} presentata nel pi\`{u} generale contesto multidimensionale in Par. 1.3, perch\`{e} in questo paragrafo ci occuperemo esclusivamente del caso bidimensionale.
\begin{defn}
\rm{Sia $\textbf{x}_1,\ldots, \textbf{x}_N$ una successione finita nel quadrato $I^{2}$. La discrepanza $D_N$ di questi punti \`{e} definita come
\begin{equation*}
D_N=D_N(\textbf{x}_1,\ldots, \textbf{x}_N)=\sup_J \left|\frac{1}{N}\sum_{n=1}^{N} \chi_{J}(\{\textbf{x}_n\})-\lambda(J) \right| \ ,
\end{equation*}
dove $J=[a_1,b_1[\times[a_2,b_2[\subset I^{2}$ varia in tutti i modi possibili, con $0\leq a_j < b_j< 1$, $j=1,2$.}
\end{defn}

Di seguito riportiamo alcuni grafici relativi alle successioni di tipo (\ref{eq:3.5}) che ci permettano di visualizzare la distribuzione delle coppie nel quadrato e di fare qualche ipotesi sulla loro distribuzione. In particolare, sarebbe interessante scoprire se si tratta ancora di successioni uniformemente distribuite, cos\`{i} come accade con la successione di van der Corput nel passaggio dall'intervallo unitario al quadrato.\\
\begin{figure}[h!]
\centering
\setlength{\unitlength}{1cm}
\includegraphics[scale=0.88, keepaspectratio]{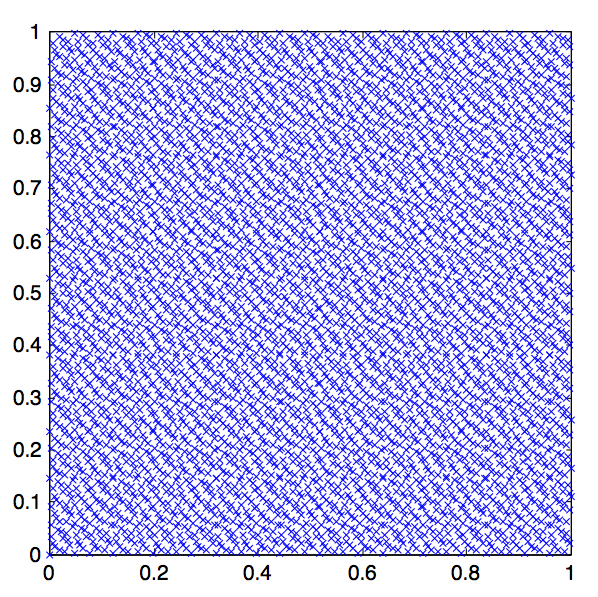}
\caption{Successione $\left(\frac{n}{5000}, \xi_{1,1}^{n}\right)$ per $1\leq n\leq 5000$}\label{fig:unif.-11}
\end{figure}
\begin{figure}[h!]
\centering
\setlength{\unitlength}{1cm}
\includegraphics[scale=0.88, keepaspectratio]{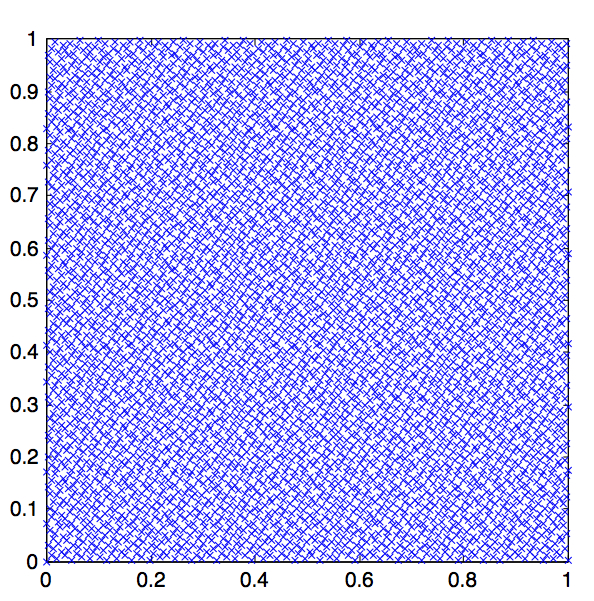}
\caption{Successione $\left(\frac{n}{5000}, \xi_{2,1}^{n}\right)$ per $1\leq n\leq 5000$}\label{fig:unif.-21}
\end{figure}
\begin{figure}[h!]
\centering
\setlength{\unitlength}{1cm}
\includegraphics[scale=0.88, keepaspectratio]{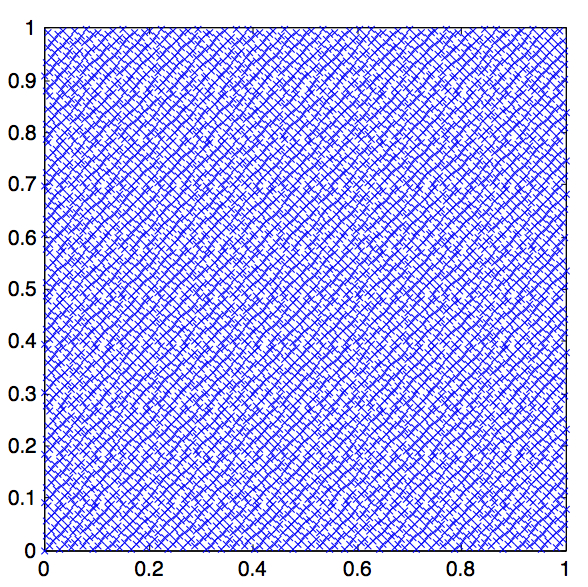}
\caption{Successione $\left(\frac{n}{5000}, \xi_{3,1}^{n}\right)$ per $1\leq n\leq 5000$}\label{fig:unif.-31}
\end{figure}
\bigskip
\\Le immagini seguenti mostrano una distribuzione delle coppie che riveste il quadrato in maniera piuttosto uniforme. Non ci sono infatti zone in cui i punti si addensano n\`{e} zone di rarefazione. La distribuzione dei punti in questi ultimi due casi, inoltre, appare addirittura migliore di quella di Halton in \figurename~\ref{fig:vdc23} e confrontabile con quella di van der Corput in \figurename~\ref{fig:unif.-vdc}. \pagebreak \\
Nel prossimo ed ultimo paragrafo introdurremo delle $LS$-successioni infinite, prenderemo in esame alcuni significativi esempi e cercheremo di determinare quale successione ha il comportamento migliore.
\section{$LS$-successioni ``alla Halton''}
Volgiamo ora la nostra attenzione ad una seconda, ultima e pi\`{u} complessa generalizzazione delle $LS$-successioni di punti nel quadrato unitario. Esse vengono definite in questo contesto in modo naturale, in maniera analoga a quanto visto per le successioni di Halton. Infatti, l'algoritmo presentato nel capitolo precedente in Par. 2.5 ci permette di calcolare i punti della successione $\{\xi_{L,S}^{n}\}_{n\in\mathbb{N}}$ senza doverci preoccupare della posizione di tali punti in qualcuno dei blocchi $\Lambda_{L,S}^{n}$, a differenza di quanto succedeva nella prima definizione di $LS$-successioni di punti introdotte in \cite{Carbone}. \\ Allora appare del tutto naturale considerare coppie di punti appartenenti a due $LS$-successioni (di punti).
\begin{defn}
\rm{Definiamo \emph{$LS$-successione ``alla Halton''} nel quadrato $I^{2}$ la successione $\{(\xi_{L_1S_1}^{n},\xi_{L_2S_2}^{n})\}_{n\in\mathbb{N}}$, dove $\{\xi_{L_1S_1}^{n}\}_{n\in\mathbb{N}}$ e $\{\xi_{L_2S_2}^{n}\}_{n\in\mathbb{N}}$ sono due $LS$-successioni di punti.}
\end{defn}
A differenza della Definizione \ref{Def:3.2.1} con la quale abbiamo introdotto una successione finita di coppie di punti in $I^{2}$, che varia al variare di $N\in \mathbb{N}$, in questo caso abbiamo definito una successione infinita.\\ Non solo. Questa generalizzazione fa intravvedere due possibili generalizzazioni alle dimensioni $s>2$. Potremmo infatti considerare $s$-uple di $LS$-successioni, oppure $s$-uple la cui prima componente sia la successione $\left( \frac{n}{N} \right)$, e le restanti $s-1$ siano costituite da $LS$-successioni.\\
Possiamo implementare l'algoritmo 2.5 separatamente per le successioni in ascissa e per quelle in ordinata e considerare poi le coppie di punti di tipo $LS$ ottenute. In particolare abbiamo preso in esame, sia in ascissa che in ordinata, successioni di punti $\{\xi_{LS}^{n}\}$ a bassa discrepanza, ossia quelle relative alla relazione $L\beta+S\beta^{2}=1$, con $L\geq S$. \\Qui di seguito riportiamo i grafici pi\`{u} interessanti tra quelli ottenuti in seguito a queste implementazioni.\\ 
\begin{figure}[h!]
\centering
\setlength{\unitlength}{1cm}
\includegraphics[scale=0.88, keepaspectratio]{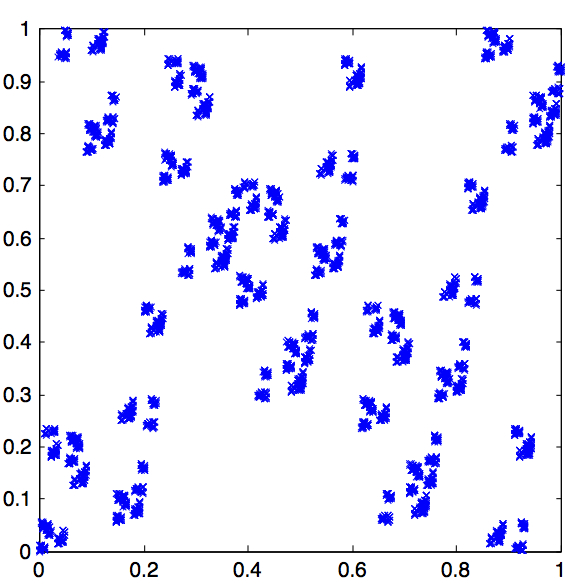}
\caption{Successione $\{(\xi_{1,1}^{n},\xi_{4,1}^{n})\}_{n\in\mathbb{N}}$ per $1\leq n\leq 5000$}\label{fig:11-41}
\end{figure}
\pagebreak
\begin{figure}[h!]
\centering
\setlength{\unitlength}{1cm}
\includegraphics[scale=0.88, keepaspectratio]{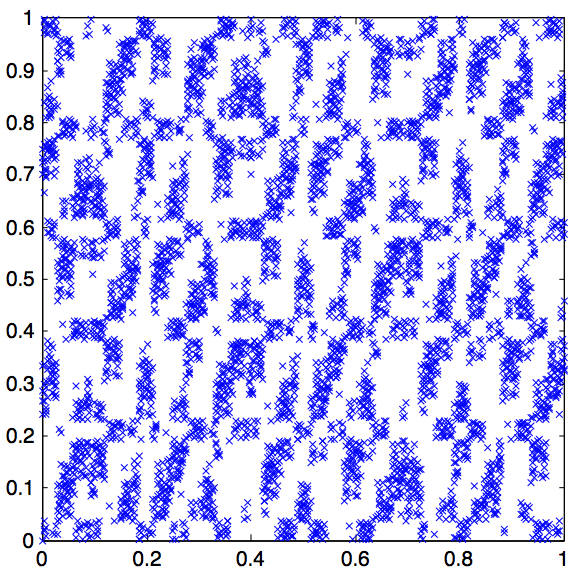}
\caption{Successione $\{(\xi_{3,1}^{n},\xi_{5,1}^{n})\}_{n\in\mathbb{N}}$ per $1\leq n\leq 5000$}\label{fig:31-51}
\end{figure}
\begin{figure}[h!]
\centering
\setlength{\unitlength}{1cm}
\includegraphics[scale=0.88, keepaspectratio]{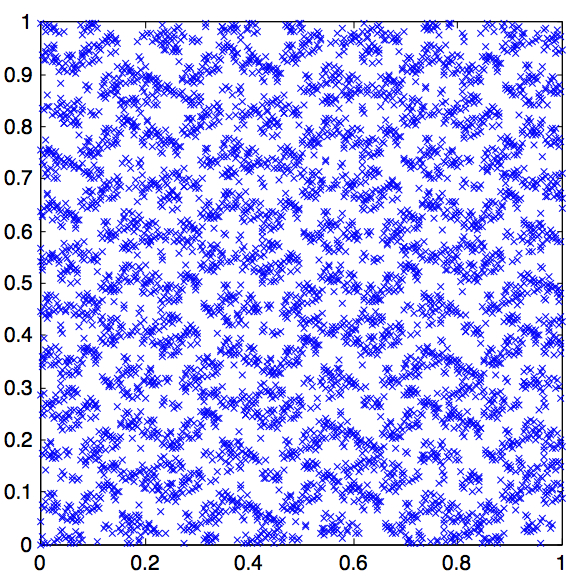}
\caption{Successione $\{(\xi_{5,1}^{n},\xi_{7,1}^{n})\}_{n\in\mathbb{N}}$ per $1\leq n\leq 5000$}\label{fig:51-71}
\end{figure}

\begin{figure}[h!]
\centering
\setlength{\unitlength}{1cm}
\includegraphics[scale=0.88, keepaspectratio]{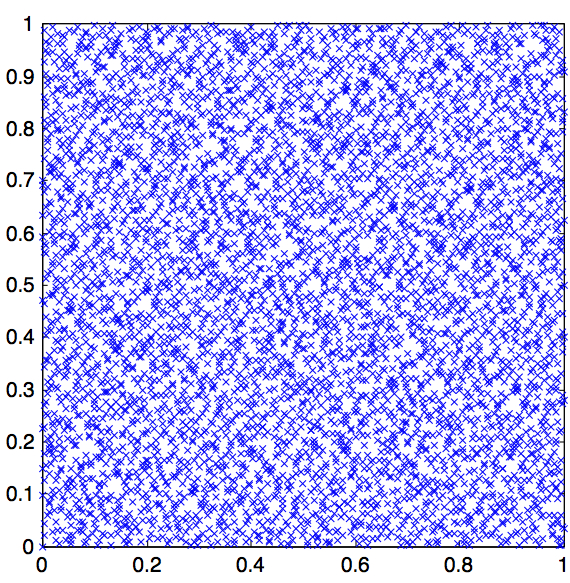}
\caption{Successione $\{(\xi_{1,1}^{n},\xi_{5,1}^{n})\}_{n\in\mathbb{N}}$ per $1\leq n\leq 5000$}\label{fig:11-51}
\end{figure}
\begin{figure}[h!]
\centering
\setlength{\unitlength}{1cm}
\includegraphics[scale=0.88, keepaspectratio]{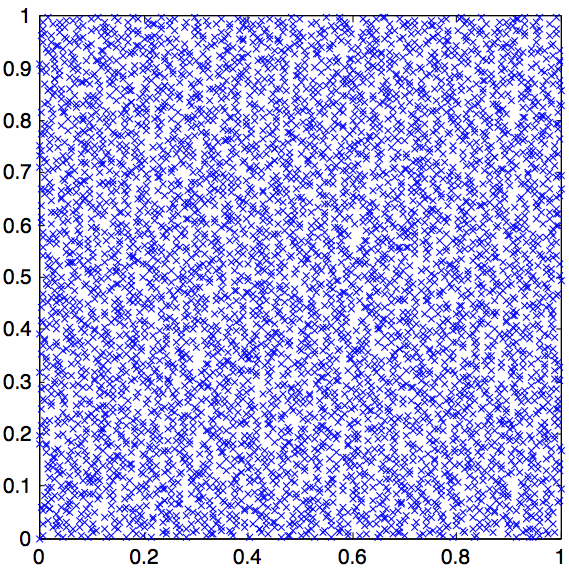}
\caption{Successione $\{(\xi_{3,1}^{n},\xi_{4,1}^{n})\}_{n\in\mathbb{N}}$ per $1\leq n\leq 5000$}\label{fig:31-41}
\end{figure}
Come si pu\`{o} osservare, si tratta di risultati molto diversi tra loro, sebbene tutti discendano da successioni a bassa discrepanza su ciascun lato del quadrato. In particolare, risultano emblematici i primi due grafici, relativi alle successioni di coppie $\{(\xi_{1,1}^{n},\xi_{4,1}^{n})\}_{n\in\mathbb{N}}$ e $\{(\xi_{3,1}^{n},\xi_{5,1}^{n})\}_{n\in\mathbb{N}}$, per cui non solo non si conserva la bassa discrepanza, ma addirittura si \`{e} portati a ritenere che non si tratti di successioni uniformemente distribuite.
\\ Per quanto riguarda il primo grafico, la risposta a questa irregolarit\`{a} potrebbe trovarsi nell'espressione dei $\beta$ che risolvono le equazioni corrispondenti, e cio\`{e}:
\begin{equation*}
\beta+\beta^{2}=1 \quad \Rightarrow \quad \beta=\dfrac{\sqrt{5}-1}{2}
\end{equation*} 
\begin{equation*}
4\beta+\beta^{2}=1 \quad\Rightarrow \quad \beta=\sqrt{5}-2\ .
\end{equation*}
I due valori risultano essere razionalmente dipendenti, ossia una loro combinazione lineare a coefficienti razionali d\`{a} come risultato un intero. Infatti $$2\left(\frac{\sqrt{5}-1}{2}\right)-(\sqrt{5}-2)=1\ .$$ \`{E} dunque lecito supporre che sussista un legame tra le due successioni di punti $\{\xi_{1,1}^{n}\}_{n\in\mathbb{N}}$ e $\{\xi_{4,1}^{n}\}_{n\in\mathbb{N}}$ e tra le corrispondenti successioni di partizioni $\{\rho_{1,1}^{n}\}_{n\in\mathbb{N}}$ e $\{\rho_{4,1}^{n}\}_{n\in\mathbb{N}}$. \\ Effettivamente, in questo caso la risposta \`{e} affermativa. \\ Precisamente, vale il seguente risultato.
\begin{prop}
\label{prop:3.3.2}
Se $t_{n}$ \`{e} il numero totale degli intervalli di $\rho_{1,1}^{n}$ e $t_n^{'}$ il numero totale degli intervalli della partizione $\rho_{4,1}^{n}$, allora si ha che
\begin{equation}
t^{'}_n=t_{3n}\qquad \textrm{per\ ogni}\ n \in \mathbb{N}\ . \label{eq:3.4}
\end{equation}
\end{prop}
\begin{proof}
Dimostriamo tale relazione per induzione. La (\ref{eq:3.4}) risulta vera per $n=1$ poich\`{e} vale l'uguaglianza $t^{'}_1=5=t_{3}$. Supponiamo ora che la (\ref{eq:3.4}) sia vera per $n-1$, ossia $t^{'}_{n-1}=t_{3(n-1)}$ e dimostriamo che \`{e} vera per $n$. \\Per quanto visto nel capitolo precedente, sappiamo che valgono le equazioni alle differenze (\ref{eq:2.4}), e cio\`{e}
\begin{equation}
t_n=t_{n-1}+t_{n-2}\label{eq:3.6}
\end{equation}
e
\begin{equation}
t^{'}_n=4t^{'}_{n-1}+t^{'}_{n-2} \ . \label{eq:3.7}\smallskip
\end{equation}
Allora, utilizzando la prima di queste relazioni, si ha
\begin{eqnarray*}
t_{3n}&=&t_{3n-1}+t_{3n-2}=t_{3n-2}+t_{3n-3}+t_{3n-2}= \\
&=& t_{3n-3}+t_{3n-4}+t_{3n-3}+t_{3n-3}+t_{3n-4}= \\
&=& 2(t_{3n-3}+t_{3n-4})+t_{3n-3}=3t_{3n-3}+2t_{3n-5}+2t_{3n-6}= \\
&=& 3t_{3n-3}+t_{3n-5}+t_{3n-3}-t_{3n-4}+2t_{3n-6}= \\
&=& 4t_{3n-3}+t_{3n-6}\ ,
\end{eqnarray*}
ma per ipotesi induttiva $4t_{3(n-1)}+t_{3(n-2)}=4t^{'}_{n-1}+t^{'}_{n-2}$, ossia la tesi. 
\end{proof}
Attraverso la proposizione precedente abbiamo dimostrato che la successione $\{t_n^{'}\}_{n\in\mathbb{N}}$ degli intervalli di $\{\rho_{4,1}^{n}\}$ \`{e} una sottosuccessione della successione di Fibonacci $\{t_n\}_{n\in\mathbb{N}}$ richiamata in (\ref{eq:3.6}). \\ A questo punto, vogliamo vedere se esiste una simile relazione anche tra $\{\rho_{3,1}^{n}\}$ e $\{\rho_{5,1}^{n}\}$ e, pi\`{u} in generale, se \`{e} vero che, ad esempio, come nella \figurename~\ref{fig:31-51} e nella \figurename~\ref{fig:51-71}, la pessima distribuzione delle prime $5000$ coppie della successione $\{(\xi_{L,S}^{n},\xi_{L',S'}^{n})\}_{n\in\mathbb{N}}$ nel quadrato dipenda dal fatto che la successione $\{t_n\}_{n\in\mathbb{N}}$ associata a $\{\rho_{L,S}^{n}\}_{n\in\mathbb{N}}$ sia una sottosuccessione di $\{t_n^{'}\}_{n\in\mathbb{N}}$ associata a $\{\rho_{L^{'},S^{'}}^{n}\}_{n\in\mathbb{N}}$. \\ Per determinare il valore che $k$ deve assumere affinch\`{e} sia vera l'equazione $t^{'}_n=t_{kn}$ per ogni $n\in\mathbb{N}$, osserviamo innanzitutto che essa deve essere vera anche per $n=1$. Dunque, per risolvere il problema \`{e} sufficiente considerare l'equazione ottenuta per $n=1$, e cio\`{e}
\begin{equation*}
t^{'}_1=t_k \ .
\end{equation*}
Infatti, poich\`{e} $t^{'}_1=L'+S'$, il problema si riduce a trovare il $k$ corrispondente a $t_k=L'+S'$. \\ Nel caso delle successioni di partizioni $\{\rho_{1,1}^{n}\}_{n\in\mathbb{N}}$ e $\{\rho_{4,1}^{n}\}_{n\in\mathbb{N}}$ bisogna osservare che $t^{'}_1=4+1=5=t_k$ che corrisponde a $k=3$. \\ Per quanto riguarda la coppia di successioni di partizioni $\{\rho_{3,1}^{n}\}$ e $\{\rho_{5,1}^{n}\}$ non esiste una relazione simile a quella trovata per $\{\rho_{1,1}^{n}\}$ e $\{\rho_{4,1}^{n}\}$ perch\`{e} $t^{'}_1=L'+S'=5+1=6$ non corrisponde ad alcun valore di $t_k$. Dobbiamo dunque dedurre che non esiste alcun $k$ che verifichi l'equazione $t^{'}_1=t_k$.
\paragraph*{}
Vediamo infine un confronto tra la successione di van der Corput di ordine $N$, $\left( \frac{n}{N}, \phi_2(n)\right)$, la successione di Halton $\{(\phi_2(n), \phi_3(n))\}_{n\in\mathbb{N}}$, la $LS$-successione ``alla van der Corput'' di ordine $N$, $(\frac{n}{N}, \xi_{3,1}^{n})$ e la $LS$-successione ``alla Halton''$\{(\xi_{3,1}^{n},\xi_{4,1}^{n})\}_{n\in\mathbb{N}}$ presentate in questo capitolo, dal punto di vista dell'uniforme distribuzione. Per effettuare tale confronto prendiamo in esame la stima di un integrale doppio di una funzione continua di cui conosciamo il valore esatto e vediamo quale successione fornisce la stima migliore.\smallskip\\ Sia $f(x,y)=2x+3y^{2}$ con $(x,y)\in I^{2}$. Allora naturalmente
\begin{equation*}
\int_{I^{2}}f(x,y)dx dy=\int_{I^{2}}2x+3y^{2}dx dy = 2\ .
\end{equation*}
Nella tabella seguente riportiamo i valori di $\frac{1}{N}\sum_{n=1}^{N}f(x,y)$, \linebreak con $N=100, 500, 1000, 2000$.\\
\begin{table}[h!]
\centering
\begin{tabular}{|c|c|c|c|c| c|}
\hline
$N$  & $\left( \frac{n}{N}, \phi_2(n)\right)$ &$\{(\phi_2(n),\phi_3(n))\}$ & $(\frac{n}{N}, \xi_{3,1}^{n})$ & $\{(\xi_{3,1}^{n},\xi_{4,1}^{n})\}$ & Pseudo \\
 & & & & & Random\\
 \hline
$100$ & $1.9560$ & $1.9346$ & $1.9440$ & $1.958$ & $2.0592$\\
\hline
$500$ & $1.9906$ & $1.9837$ & $1.9898$ & $1.9911$ & $1.9495$\\
\hline
 $1000$ & $1.9953$ & $1.9925$ & $1.9935$ & $1.9949$ & $2.0506$\\
\hline
 $2000$ & $1.9977$ & $1.9959$ & $1.9979$ & $1.9976$ & $1.9925$\\
\hline
\end{tabular}
\caption{Stime dell'integrale di $f$}\label{tabella}
\end{table}\\

Da questi calcoli risulta evidente che l'insieme di punti che meglio approssima il valore dell'integrale di $f$ \`{e} quello corrispondente alla $LS$-successione ``alla van der Corput'' di ordine $N$, $(\frac{n}{N}, \xi_{3,1}^{n})$. In particolare, questo risultato mostra un comportamento migliore di quest'ultima successione anche rispetto alla successione di Halton.\\ Per contro, si pu\`{o} osservare come sia poco soddisfacente l'approssimazione ottenuta con i numeri pseudo-casuali che presenta oscillazioni notevoli nell'errore di approssimazione.\medskip\\ Visti i risultati sperimentali, riteniamo di poter formulare la seguente
\begin{congettura}
\label{congettura}
Le $LS$-successioni ``alla van der Corput'' nel quadrato sono u.d..
\end{congettura}
Questo rimane comunque un problema aperto, come resta aperto il problema della stima della discrepanza di queste successioni. \smallskip \\ Si tratta prevedibilmente di un problema difficile, poich\`{e} le analoghe dimostrazioni per le successioni associate ai numeri primi (di van der Corput, Hammersley e Halton) hanno rappresentato e rappresentano delle pietre miliari nel campo della distribuzione uniforme.\smallskip \\ La congettura \`{e} d'altra parte supportata non solo dall'effetto visivo che ci danno le figure di questo paragrafo, ma anche dai risultati numerici che abbiamo presentato nella Tabella \ref{tabella} e che mostrano come i metodi sviluppati in questo capitolo siano senza dubbio competitivi con quelli noti nella letteratura.

\chapter{Osservazioni conclusive}
La parte finale di questa tesi \`{e} stata dedicata alla costruzione di nuove successioni di punti nel quadrato $[0,1[\times [0,1[$ utilizzando uno strumento nuovo, introdotto recentemente in \cite{Carbone}, e cio\`{e} le $LS$-successioni. Abbiamo, infatti, introdotto due classi numerabili di successioni di punti, che abbiamo chiamato $LS$-successioni ``alla van der Corput'' e $LS$-successioni ``alla Halton''.  Successioni che appartengono alla prima classe sono di ordine $N$ per ogni intero positivo $N$ fissato (e, cio\`{e}, sono costituite da $N$ elementi), mentre quelle appartenenti alla seconda classe sono infinite.\\

Relativamente a questo problema abbiamo fatto molta sperimentazione sia grafica che numerica trovando dei risultati variegati e comunque interessanti. Alcuni esperimenti hanno dato degli esiti francamente sorprendenti e che avranno bisogno di ulteriori approfondimenti teorici per spiegare certi fenomeni. \\

Sono molto confortanti i risultati ottenuti considerando su una delle due coordinate la successione equidistribuita e sull'altra una $LS$-successione ``alla van der Corput'' (Figure \ref{fig:unif.-11}, \ref{fig:unif.-21} e \ref{fig:unif.-31}). Sia i risultati grafici che i risultati numerici ci hanno consentito di formulare la Congettura \ref{congettura} che si pu\`{o} ritenere molto plausibile. \\

Gli esperimenti che abbiamo fatto invece con coppie di $LS$-successioni ``alla Halton'' hanno dato dei risultati contrastanti. Da un lato abbiamo ottenuto delle successioni nel quadrato che appaiono uniformemente distribuite sia con un numero basso di punti sia con un numero elevato di punti (Figure \ref{fig:11-51} e \ref{fig:31-41}) e sembrano competitive con le successioni di Halton per il calcolo dell'integrale di funzioni di due variabili. Ma abbiamo anche ottenuto delle successioni che palesemente non sono uniformemente distribuite e che danno luogo a concentrazioni di punti in certe zone, accompagnate da altre zone prive di punti (Figure \ref{fig:11-41}, \ref{fig:31-51} e \ref{fig:51-71}). Un fenomeno analogo si verifica con le successioni di Halton associate a numeri che non siano coprimi. Ma mentre per le successioni di Halton la spiegazione era stata data da Halton stesso, per le $LS$-successioni manca ancora una spiegazione teorica per un comportamento tanto variegato. Una spiegazione (Proposizione \ref{prop:3.3.2}) \`{e} stata trovata per l'accoppiamento $\{\xi_{1,1}^{n}\}_{n\in\mathbb{N}}$ e $\{\xi_{4,1}^{n}\}_{n\in\mathbb{N}}$, quello forse pi\`{u} sorprendente (Figura \ref{fig:11-41}). Per esso si verifica lo stesso fenomeno che si pu\`{o} osservare con l'accoppiamento di successioni di Halton quando si prendono basi non coprime, e cio\`{e} una ``risonanza'  in corrispondenza al minimo comune multiplo delle due basi. Si sono per\`{o} scoperti successivamente altri accoppiamenti che non sembrano dar luogo a successioni u.d. nel quadrato, per i quali tale risonanza non c'\`{e}. \\

Questo fenomeno merita uno studio successivo, probabilmente non semplice, se si pensa alla lentezza del progresso della conoscenza in questo settore: ai lavori di van der Corput sono seguiti oltre vent'anni dopo quelli di Hammersley (che si era limitato a fare una congettura e degli esperimenti numerici) e di Halton, che  ha fornito una base teorica per spiegare il buon comportamento delle successioni introdotte da Hammersley e ad introdurre a sua volta un'altra generalizzazione del procedimento di van der Corput. \\

Un ulteriore filone di ricerca, qualora la congettura venga dimostrata, sarebbe indirizzato allo studio della discrepanza delle $LS$-successioni u.d..

\end{document}